\documentclass[12pt]{amsart}
\usepackage{amsfonts}
\usepackage{amsmath}
\usepackage{amsthm}
\usepackage{amssymb}
\usepackage{latexsym}
\usepackage{graphicx}
\usepackage{epsfig}
\usepackage{url}
\usepackage{multicol}
\makeindex

\usepackage{tabularx}
\usepackage{enumerate,mathrsfs}

\usepackage{comment}

\usepackage[pagebackref,colorlinks,linkcolor=red,citecolor=blue,urlcolor=blue,hypertexnames=true]{hyperref}

\usepackage{titletoc}

\def\bes{\begin{equation*} }
\def\ees{\end{equation*} }

\DeclareMathOperator{\rank}{rank}
\DeclareMathOperator{\Id}{Id}

\def\beq{\begin{equation}}
\def\eeq{\end{equation}}

\def\ben{\begin{enumerate}}
\def\een{\end{enumerate}}

\def\Rnn{\R^{\nu\times\nu}}
\def\Rll{\R^{\ell\times\ell}}

\def\R{\mathbb{R}}
\def\RR{\mathbb{R}}
\def\CC{\mathbb{C}}
\def\NN{\mathbb{N}}

\def\ax{\langle x \rangle}
\def\axs{\langle \x, \xs \rangle}
\def\N{\mathbb{N}}

\def\x{\mathcal{X}}

\def\tr{\mathrm{Tr}}

\def\id{\mathrm{Id}}

\def\J{\mathcal{I}}

\def\cB{ {\mathcal B} }
\def\cC{ {\mathcal C} }
\def\cD{ {\mathcal D} }

\def\cG{ {\mathcal G} }
\def\cH{ {\mathcal H} }
\def\cI{ {\mathcal I} }
\def\cJ{ {\mathcal J}}

\def\cN{ {\mathcal N} }

\def\cU{\mathcal U}

\def\cZ{ {\mathcal Z} }

\def\hom{\mbox{\tiny hom}}
\def\sym{\operatorname{Sym}}
\def\free{\operatorname{Fr}}

\newcommand{\setmult}[2]{#1 #2}
\newcommand{\leftact}[2]{#1 \cdot #2}
\newcommand{\lead}[1]{\operatorname{Tip}(#1)}
\newcommand{\nonlead}[1]{\operatorname{NonTip}(#1)}
\def\finite{\mbox{\rm\tiny finite}}
\def\Md{M}
\def\x{x}
\def\xs{x^{\ast}}
\def\real{\rm re}

\def\matng{\left(\RR^{n \times n}\right)^g}

\newcommand{\Lradd}[3]{\sqrt[(#1,#2)]{#3}}
\newcommand{\LraddS}[3]{\sqrt[(#1,#2)+]{#3}}
\newcommand{\Lrad}[2]{\sqrt[(#1)]{#2}}
\newcommand{\LradS}[2]{\sqrt[(#1)+]{#2}}

\def\rC{\mathfrak{C}}
\def\rA{\mathfrak A}
\def\rB{\mathfrak B}
\def\mbS{\mathbb{S}}

\newtheorem{theorem}{Theorem}[section]
\newtheorem{lemma}[theorem]{Lemma}
\newtheorem{lem}[theorem]{Lemma}
\newtheorem{thm}[theorem]{Theorem}
\newtheorem{prop}[theorem]{Proposition}
\newtheorem{cor}[theorem]{Corollary}

\newtheorem{corollary}[theorem]{Corollary}

\theoremstyle{definition}
\newtheorem{remark}[theorem]{Remark}
\newtheorem{exa}[theorem]{Example}

\newtheorem{definition}[theorem]{Definition}

\def\subset{\subseteq}
\def\supset{\supseteq}

\newcommand{\rr}[1]{\sqrt[\real]{#1}}

\newcommand{\df}[1]{{\bf{#1}}{\index{#1}}}

\numberwithin{equation}{section}

\begin{document}

\title
[Polynomials nonnegative on a variety intersect a convex set]
{Noncommutative  polynomials nonnegative on a\\[.1cm] variety intersect a convex set}

\author[J. William Helton]{J. William Helton${}^1$}
\address{J. William Helton, Department of Mathematics\\
  University of California \\
  San Diego}
\email{helton@math.ucsd.edu}
\thanks{${}^1$Research supported by the National Science Foundation (NSF) grants
DMS 0700758, DMS 0757212, DMS 1201498, and the Ford Motor Co.}
\author[Igor Klep]{Igor Klep${}^{2}$}
\address{Igor Klep, Department of Mathematics, 
The University of Auckland, New Zealand}
\email{igor.klep@auckland.ac.nz}
\thanks{${}^2$Supported by the Faculty Research Development Fund (FRDF) of The
University of Auckland (project no. 3701119). Partially supported by the Slovenian Research Agency grant P1-0222.}

\author[Christopher S. Nelson]{Christopher S. Nelson${}^3$}
\address{Christopher S. Nelson, Department of Mathematics\\
  University of California \\
  San Diego}
\email{csnelson@math.ucsd.edu}
\thanks{${}^3$Partly supported by the National Science Foundation, DMS 1201498}



\date{\today}

\setcounter{tocdepth}{3}
\contentsmargin{2.55em} 
\dottedcontents{section}[3.8em]{}{2.3em}{.4pc} 
\dottedcontents{subsection}[6.1em]{}{3.2em}{.4pc}
\dottedcontents{subsubsection}[8.4em]{}{4.1em}{.4pc}

\linespread{1.05}
\begin{abstract}
By a  result of Helton and McCullough \cite{HM3},
open 
 bounded convex free  semialgebraic sets
 are exactly open  (matricial) solution sets $\cD_L^\circ$
of a linear matrix inequality (LMI) $L(X)\succ0$.
This paper gives
a precise  algebraic certificate  for a polynomial
being nonnegative    on a convex semialgebraic set intersect a variety, a so-called ``Perfect'' Positivstellensatz.

For example, given a generic convex free  semialgebraic set $\cD_L^\circ$
we determine all ``(strong sense) defining polynomials" $p$ for $\cD_L^\circ$.
 Such polynomials must have the form
\[
 p = L \left( \sum_{i}^{\rm finite} q_i^*q_i \right) L + \sum_{j}^{\rm finite}
\left(r_jL + C_j \right)^*L\left(r_jL + C_j\right),
\]
where $q_i,r_j$ are matrices of  polynomials,
and $C_j$  are real matrices  satisfying $C_jL = LC_j$.

This follows from our general result for  a
given  linear pencil $L$ and a finite set $I$ of rows of polynomials.
A matrix polynomial $p$ is positive where $L$ is positive and all $\iota\in I$ vanish
if and only if
\beq\label{eq:posCert}\tag{P}
 p = \sum_{i}^{\finite}p_i^{\ast}p_i + \sum_{j}^{\finite} q_j^{\ast}Lq_j +
  \sum_k^{\finite} (r_k^{\ast}\iota_k + \iota_k^{\ast}r_k),
\eeq
where each $p_i$, $q_j$ and $r_k$ are matrices of polynomials of appropriate
dimension, and each $\iota_k$ is an element of the ``$L$-real radical'' of $I$.
In this representation, we can restrict $p_i$, $q_i$, $\iota_k$ and $r_k$ to be elements of a
low-dimensional subspace of matrices of  polynomials, and in
particular, their degrees depend in a very tame way only on the degree of $p$ and the degrees
of the elements of $I$.
Further, this paper gives an efficient algorithm for computing the $L$-real radical
of $I$.

This Positivstellensatz extends and unifies two different lines of results:
\ben[\rm(1)]
\item
the free real Nullstellensatz of \cite{chmn, N} which gives an algebraic
certificate corresponding to one polynomial being zero on the free
 variety where others are zero; this is \eqref{eq:posCert} with $L=1$;
\item  the  convex Positivstellensatz of \cite{HKMb,KS11} which is \eqref{eq:posCert}
without $I$; i.e., $I=\{0\}$.
\een
The representation \eqref{eq:posCert} has a number of additional consequences which will be
presented.
\end{abstract} 

\maketitle
\linespread{1.18}

\section{Introduction}
\label{sect:intro}

In this section we introduce the main concepts which will be used throughout
this paper.
Subsections \ref{sec:intronotation}, \ref{sec:introncpolys}, \ref{sec:introlinpen},
\ref{sec:preLrad}
give basic definitions and include
 Theorem \ref{thm:detBdry2} which
characterizes ``defining polynomials" 
of a convex free  semialgebraic set.
Subsequently \S\ref{sub:rightChip}, \S\ref{sec:introLrad} give more definitions and then
in \S\ref{sub:overview} we state
our main general  result Theorem \ref{thm:mainNMon}
followed by  statements of corollaries.
The paper is devoted to proving these things and giving algorithms;
 a guide to the presentation is \S\ref{sec:intro guide}. 
 
\subsection{Notation}\label{sec:intronotation}
Given positive integers $\nu$ and $\ell$,
let $\RR^{\nu \times \ell}$ denote the space of $\nu
\times \ell$ real matrices.
We use $E_{ij} \in \RR^{\nu \times \ell}$
\index{Eij@$E_{ij}$} to denote
the matrix unit with a $1$ as the $ij^{th}$ entry and a $0$ for all other entries.
If $\nu = 1$, then $E_{1j} = e_j \in \RR^{1 \times \ell}$ \index{ej@$e_j$} is
the
row vector with $1$ as
the $j^{th}$ entry and a $0$ as all other entries.
Let $\operatorname{Id}_\nu \in \RR^{\nu \times \nu}$ denote the $\nu \times \nu$
identity matrix.
The transpose of a matrix $A \in \RR^{\nu \times \ell}$ will be denoted by $A^* \in \RR^{\ell \times \nu}$, and 
$\mathbb{S}^{k} \subset \RR^{k \times k}$\index{S^k@$\mathbb{S}^k$} is
the space of real symmetric $k \times k$ matrices.

\subsection{Non-Commutative Polynomials}\label{sec:introncpolys}
Let $\axs$\index{$\axs$} be the monoid freely generated by $x = (x_1,
\ldots, x_g)$ and $x^* = (x_1^*, \ldots, x_g^*)$---that is, $\axs$ consists of
words in the $2g$ free letters $x_1, \ldots, x_g, x_1^*, \ldots, x_g^*$,
including the empty word $\varnothing$, which plays the role of the identity $1$.
Let  $\RR\axs$\index{Rxxs@$\RR\axs$} denote the
$\RR$-algebra freely generated by $\axs$, i.e., the
elements of $\RR\axs$
are polynomials in the non-commuting variables $\axs$
with coefficients
in $\RR$.  Call elements of $\RR\axs$
\textbf{non-commutative}\index{non-commutative (NC) polynomials} or \textbf{NC}
polynomials.

The \textbf{involution}\index{involution} on
$\RR\axs$ is defined linearly so that $(x_i^*)^* = x_i$ for each variable $x_i$
and
$(pq)^* = q^*p^*$ for each $p, q \in \RR\axs$.
For example,
\[
 \left( x_1x_2x_3 + 2x_3^*x_1 - x_3\right)^* = x_3^*x_2^*x_1^* + 2x_1^*x_3 - x_3^*
\]

\subsubsection{Evaluation of NC Polynomials}

NC polynomials can be evaluated at a tuple of matrices in a natural way.
 Let $X = (X_1, \ldots, X_g)$ be a tuple of real $n \times n$ matrices,
 that is $X \in \matng$.  Given $p
\in
\RR\axs$, let $p(X)$ denote the matrix defined by
replacing each $x_i$ in $p$ with $X_i$, each $x_i^*$ in $p$ with $X_i^{\ast}$,
and replacing the empty word with $\operatorname{Id}_{n}$.  Observe that $p^*(X) =
p(X)^*$ for all $p \in \RR\axs$.

For example, if
\[
p(x) = x_1^2 - 2x_1x_2^* -3, \quad
X_1 = \begin{pmatrix}
1&2\\
2&4
\end{pmatrix}
\quad 
\mbox{and}
\quad
X_2 = \begin{pmatrix}
0&-1\\
1&-1
\end{pmatrix}
\]
then
\begin{align}
 \notag
p(X) &= X_1^2 - 2X_1 X_2^* - 3 \operatorname{Id}_2 \\
\notag
&=
\begin{pmatrix}
1&2\\
2&4
\end{pmatrix}\begin{pmatrix}
1&2\\
2&4
\end{pmatrix}
- 2 \begin{pmatrix}
1&2\\
2&4
\end{pmatrix}
\begin{pmatrix}
0&1\\
-1&-1
\end{pmatrix}
-
\begin{pmatrix}
3&0\\
0&3
\end{pmatrix}
\\
\notag
&
=
\left(
\begin{array}{cc}
 6 & 12 \\
 18 & 21 \\
\end{array}
\right)
\end{align}

\subsubsection{Matrices of NC Polynomials}

The space of $\nu \times \ell$ matrices with entries in $\RR\axs$
will be
denoted as $\RR^{\nu \times \ell}\axs$. Each $p \in \RR^{\nu \times \ell}\axs$
can be expressed as
\[
p = \sum_{w \in \axs} A_w \otimes w \in \RR^{\nu \times \ell}\otimes \RR\axs.
\]
Given a tuple $X$ of real $n \times n$ matrices, let $p(X)$ denote
\[
p(X) = \sum_{w \in \axs} A_w \otimes w(X) \in \RR^{\nu n \times \ell n},
\]
where $\otimes$ denotes the Kronecker tensor product.
The involution on $\RR^{\nu \times \ell}\axs$ is
given by
\[
p^*= \left(\sum_{w \in \axs} A_w \otimes w\right)^* = \sum_{w \in \axs}
A_w^{\ast} \otimes w^* \in \RR^{\ell \times \nu}\axs.
\]
Note that if $X$ is a tuple of matrices, then $p^*(X) = p(X)^*$.
If $p \in \RR^{\nu \times \nu}\axs$, we say $p$
is \textbf{symmetric}\index{symmetric
polynomial} if $p = p^*$.

\subsubsection{Degree of NC Polynomials}
Let $|w|$ denote the \textbf{length}\index{length of a word in $\axs$} of a word
$w \in
\axs$.
A \textbf{monomial}\index{monomial} in $\RR^{\nu \times \ell}\axs$ is a
polynomial
of the form
$E_{ij} \otimes m$, where $m \in \axs$.
The {\bf length} or
 \textbf{degree}\index{degree} of a monomial $E_{ij} \otimes m$ is $|E_{ij}
\otimes m| := |m|$.
 The set of all monomials in $\RR^{\nu \times
\ell}\axs$ is a vector space basis
for $\RR^{\nu \times \ell}\axs$.

If $p$ is a NC polynomial, define the degree of $p$, denoted $\deg(p)$, to be
the
largest degree of any monomial appearing in $p$.
A NC polynomial $p$ is \textbf{homogeneous of
degree
$d$}\index{degree!homogeneous}
if every monomial appearing in $p$ has degree $d$.
If $W$ is a subspace of $\RR^{\nu \times \ell}\ax$, define $W_d$
\index{Wd for a vector
space W in Rxxs@ $W_d$ for a vector space $W \subset
\RR\axs$}\index{RRvlxxsd@$\RR^{\nu \times \ell}\axs_d$} to be
the space spanned by all elements of $W$ with degree at most $d$, and
 define $W_d^{\hom}$ to be
the space spanned by all elements of $W$ which are homogeneous of degree $d$.

\subsubsection{Operations on Sets}

If $A, B \subset \RR^{\nu \times
\ell}\axs$, then define $A + B$ to be the set
of
polynomials of the form $a + b$, with $a \in A$, $b \in B$.  In the case that
$A \cap B = \{0\}$, we also denote $A + B$ as $A \oplus B$\index{$\oplus$}---the
expression $A \oplus B$ always asserts that $A \cap B =
\{0\}$.
If $A \subset \RR^{\nu \times \ell} \ax$ and $B
\subset \RR^{\ell \times \rho} \ax$,
let $\setmult{A}{B} \subset \RR^{\nu \times \rho}\axs$\index{$\setmult{A}{B}$}
denote the span of all polynomials of the form $ab$, with $a \in A$, $b
\in B$.
If $A \subset \RR^{\nu \times \ell}\axs$, then $A^{\ast}  = \{a^* \mid a
\in A\} \subset
\RR^{\ell
\times \nu}
\axs$.
If $A \subset \RR^{\nu \times \ell}$ and $B \subset \RR\axs$, then $A \otimes B$
is defined to be the span of all simple tensors $a \otimes b$, where $a \in A$
and $b \in B$.

If $p \in \RR^{\nu \times \ell}\axs$, then expressions of the form $p + A$, $pB$,
$Cp$, $D \otimes p$, where $A$, $B$, $C$, and $D$, are sets, denote $\{p\} +
A$, $\{p\}B$, $C\{p\}$, and $D \otimes \{p\}$ respectively.

\subsubsection{Positivity sets}\label{sec:introposset}

\def\dbcD{ {  \widehat{\partial \cD} }}

Given  a symmetric matrix of NC polynomials $p$,
define its \df{positivity domain}  $\cD_p$ by
$$
\cD_p(n):=  \ \{X \in \matng : p(X) \succeq 0 \}\subseteq \matng  \qquad \qquad  \cD_p:= \bigcup_n \cD_p(n).
$$
If $p(0)\succ0$ we also introduce
$$
\cD_p(n)  ^\circ:= \text{principal  component  of } \{X \in \matng : p(X) \succ 0 \}
 \qquad \qquad  \cD_p^\circ:= \cup_n \cD_p(n)^\circ
$$
and its \df{(detailed) boundary}  $\dbcD^\circ_p$ defined by
$$
\dbcD_p^\circ := \{(X,v) :   X \in \overline{ \cD_p ^\circ}    , \   p(X)v =0 \}
$$

\subsection{Linear Pencils}\label{sec:introlinpen}

A \textbf{linear pencil}\index{linear pencil} is a
symmetric polynomial $L \in
\RR^{\nu \times \nu}\axs$, for some $\nu \in \NN$, with $\deg(L) \leq 1$.
Every
$\nu \times \nu$ linear
pencil can be
expressed as 
\[
L = A_0 + A_1 \otimes x_1 + \cdots + A_g \otimes x_g + A_1^{\ast} \otimes
x_1^* +  \cdots + A_g^{\ast} \otimes x_g^*,
\]
where each $A_i \in \RR^{\nu \times \nu}$ and $A_0$ is symmetric.
A linear pencil is \textbf{monic}\index{linear pencil!monic} if $A_0=L(0) =
\operatorname{Id}_{\nu}$.
For the purposes of this paper, we still call $L$  a linear pencil even if
$A_0 \not 
= 0$.

A \textbf{linear matrix inequality}\index{linear matrix inequality} or {\bf (LMI)} is an
expression of the form $L(x) \succeq 0$, where $L$ is a linear pencil and $x$ is
a tuple of real scalar variables.  
When $x$ is a tuple of real scalar variables, the set 
$\cD_L(1)$  is the {\bf positivity set of $L$} or the
{\bf spectrahedron} defined by $L$. Optimization of linear objective functions
over spectrahedra is called semidefinite programming (SDP) \cite{BV96, To01,
WSV00}, and is an important subfield of convex optimization.

One problem which arises in SDP is dealing with spectrahedra with
empty interior.  Every convex set with empty interior is contained in an affine
hyperplane; we call these \df{thin convex sets}.
A spectrahedron which is not thin will be called \df{thick}. 
Hence  a spectrahedron is thick if it is the closure of its interior.
Correspondingly we refer to thin and thick linear pencils $L$ as those for
which $\cD_L(1)$ is thin, or respectively thick.
A paper of Klep and
Schweighofer gives an iterative process for
finding a set of linear polynomials in $\RR[x]$ whose zero set defines the
affine subspace in which a spectrahedron lies \cite[\S3]{KS}.

A {\bf matricial relaxation of an LMI} is an expression of the form $L(X)
\succeq 0$, where $X$ is a tuple of square matrix variables.
Over $\RR\axs$, the matricial relaxation of an LMI is important because every
convex bounded noncommutative basic open semialgebraic set
$\cD_p^\circ$
is the positivity set  
$\cD_L^\circ $
of a some linear pencil $L$; see \cite{HM3}.
Sets of the form $\cD_L^\circ$ are called  \df{free open spectrahedra}, while $\cD_L$ 
are \df{free spectrahedra}.
Further,  one can use results on the matricial relaxation 
of an LMI to prove new results about the original, scalar LMI.

\subsection{Behavior of Polynomials on Real Zero Sets}\label{sec:preLrad}

 One of our main themes is taking into account behavior of zero sets.
For the free algebra $\RR\axs$, there is a ``Real Nullstellensatz''. Let $p_1, \ldots, p_k, q \in
\RR\axs$.  If $q(X)v = 0$ for every $(X,v) \in \bigcup_{n \in \NN} \left( \RR^{n \times n}
\right)^g \times \RR^n$ such that $p_1(X)v = \cdots = p_k(X)v = 0$, then $q$ is an element
of the ``real radical'' of the left ideal generated by $p_1, \ldots, p_k$, see 
\cite{chmn}.  In \cite{N} this result was 
 generalized  to $\RR^{\nu \times \ell}\axs$.
For sake of completeness we mention the
free analog of Hilbert's (complex) Nullstellensatz is given in
\cite{HMP07}.
 
Now we lay out noncommutative analogs of several notions of
classical (commutative) real algebraic geometry.

\subsubsection{Left \texorpdfstring{$\RR\axs$}{[R x xs]}-Modules}\label{sec:introLmod}
For matrices of NC polynomials we need to
adapt the notion of left ideal and real left ideal.
The space
$\RR^{\nu \times \ell}\axs$ is a left
$\RR\axs$-module with the following action:
if $q \in \RR\axs$, $A  \in \RR^{\nu \times \ell}$ and $r \in \RR\axs$, then
\[
\leftact{q}{(A \otimes r)} :=  (\Id_{\nu} \otimes \, q)(A \otimes r)=
A \otimes qr,
\]
where $\operatorname{Id}_{\nu} \in \RR^{\nu \times \nu}$ denotes the $\nu \times
\nu$ identity matrix.
In the sequel, we will simplify notation by identifying $q$ with
$\operatorname{Id}_{\nu}
\otimes\, q$ and
simply writing $q(A \otimes r)$ when we mean $\leftact{q}{(A \otimes r)}$.
We will also simplify our terminology by referring to left
$\RR\axs$-submodules $I \subset \RR^{\nu \times \ell}\axs$ as
 {\bf left modules}.

\subsubsection{Zero Sets of Left \texorpdfstring{$\RR\axs$}{[R x xs]}-Modules}

For $S \subset \RR^{1 \times \ell}\axs$, for each $n \in \NN$,
 define $V(S)^{(n)}$
to be
\[ V(S)^{(n)} := \{ (X,v) \in (\RR^{n \times n})^g \times \RR^{\ell n}
\mid p(X)v = 0\ \text{for every}\ p \in S\},\]
and let $V(S)$  be
\[ V(S) := \bigcup_{n \in \NN} V(S)^{(n)}.\]
If $W \subset \bigcup_{n \in \NN} (\RR^{n \times n})^g \times
\RR^{\ell n}$, define $\cI(W)$ to be
\[ \cI(W) := \{ p \in \RR^{1 \times \ell}\axs \mid p(X)v = 0\ \text{for every}\
(X,v) \in W\}.\]
The set $\cI(W) \subset \RR^{1 \times \ell}\axs$ is clearly a left
module.
If $I \subset \RR^{1 \times \ell}\axs$ is a left module, define the (vanishing)
\textbf{radical}\index{radical} of $I$ to be
\[
\sqrt{I} := \cI(V(I)).
\]
Finally, we define the \df{free Zariski closure, $\cZ(W)$}, of  $W$ 
 to be 
$$\cZ(W):=  V( \cI(W) ). $$

Before launching into  full generality we give an appealing corollary
of our main results;
our concern here is the nature of defining polynomials for $\cD_L$;
namely, a polynomial $p$ which is nonnegative on $\cD_L$
with $\cD_p^\circ =\cD_L^\circ$.
The following theorem applies to monic pencils $L$ with the 
rather natural \df{zero determining property}:
\beq\label{eq:conj+}
\cZ( \dbcD_L^\circ) = V(L) = \{(X,v): \ L(X)v =0 \}.
\eeq

\begin{thm}[Randstellensatz]
\label{thm:detBdry2} 

 Let $L\in \RR^{\ell \times \ell}\axs$ 
 be a monic linear pencil 
with the zero determining property. 
Let  $p \in \RR^{\ell \times \ell}\axs$.
Then
\[ 
\cD_L\subseteq \cD_p \qquad \text{and} \qquad \dbcD_L^\circ\subseteq   \dbcD_p^\circ
\]
if and only if 
\[
 p = L \left( \sum_{i}^{\rm finite} q_i^*q_i \right) L + \sum_{j}^{\rm finite}
\left(r_jL + C_j \right)^*L\left(r_jL + C_j\right)
\]
where each $q_i \in \RR^{1 \times \ell}\axs$, each $r_j \in \RR^{\ell \times
\ell}\axs$, and each $C_j \in \RR^{\ell \times \ell}$ satisfies $C_jL = LC_j$.
\end{thm}

\begin{proof}
See \S\ref{sec:thick}. 
\end{proof}

This describes 
all  $p$ which are  defining polynomials of $\cD_L$
with boundary containment happening in a strong sense.
It is a slight superset of this class, since if $q_jL$ an  $r_j L  + C_j$
all vanish simultaneously on a big enough set, then $p$ might
define a smaller set than $\cD_L$.

  The zero determining property holds for an $\ell\times\ell$ pencil 
  $L$ provided that 
 \ben[\rm(a)]
 \item
   $\deg (\det L) =  \ell$; and
   \item
   $\det L$ is the smallest degree polynomial vanishing on $\partial\cD_L(1)$;
   \een
     see Corollary \ref{cor:mindegyes}(2).
These properties  are  easy to check with computer algebra, and 
{\it they hold generically} (see Corollary \ref{cor:genZDP}).

We  now
 move towards the
 presentation  of our main theorem. Its generality forces a number of definitions.

\subsubsection{Real Left Modules}\label{subsec:Lreal}
In classical real algebraic geometry \cite{BCR98, Las10, Lau09, Mar08, PD01, Put93, Sce09} at the core of the 
real Nullstellensatz are real ideals and the real radical of an ideal.
These correspond to vanishing ideals of a variety. Now we shall study
a variety intersect a positivity domain $\cD_L$. The appropriate
notion in free algebras is what we call $L$-real left modules
and $L$-real radicals. 
We now introduce them.

Let $I \subset \RR^{1 \times \ell}\axs$ be a left module, and
$L\in\RR^{\nu\times\nu}\axs$.
We say that $I$ is $L$-\textbf{real}\index{real!left module}
if whenever
\[ \sum_i^{\finite} p_i^*p_i + \sum_j^{\finite} q_j^* L q_j \in \setmult{\RR^{\ell \times 1}}{I}
+
\setmult{I^{\ast}}{\RR^{1 \times \ell}} \]
for some $p_i \in \RR^{1 \times \ell}\axs$ and $q_j\in\RR^{\nu\times\ell}\axs$,
then each $p_i \in I$ and each $Lq_j\in \R^{\nu\times 1}I$.
Note that $\RR^{\ell \times 1}I$ is the subspace of $\ell \times \ell$
matrices whose rows are elements of $I$, and $(\RR^{\ell
\times 1}I)^* = I^*\RR^{1 \times \ell}$ is the subspace of $\ell \times \ell$
matrices whose columns are elements of $I^*$. We call $I$  \df{real} if
it is $L$-real for $L=1$.

The following result shows that no generality
is lost by 
defining a real left module in terms of only rows
of matrices.

\begin{prop}
\label{prop:diffMiReal}
 A left module $I \subset \RR^{1 \times \ell}\axs$
is $L$-real if and only if
whenever
\begin{equation}
 \label{eq:diffMi}
\sum_i^{\finite} p_i^*p_i 
+ \sum_j^{\finite} q_j^* L q_j
\in \setmult{\RR^{\ell \times 1}}{I} +
\setmult{I^{\ast}}{\RR^{1 \times \ell}},
\end{equation}
for some $p_i \in
\RR^{\nu_i
\times \ell}\axs$ and $q_j\in\RR^{\nu\times\ell}\axs$, then
each $p_i \in \setmult{\RR^{\nu_i \times 1}}{I}$ and each $Lq_j\in \R^{\nu\times 1}I$.
\end{prop}

\begin{proof}
One direction is clear. For the converse, suppose $I$ is $L$-real, and suppose that
\eqref{eq:diffMi} holds. For polynomials $p_i \in \RR^{\nu_i
\times \ell}\axs$
we have
\[ p_i^*p_i =p_i^*\operatorname{Id}_{\nu_i}p_i = \sum_{j=1}^{\nu_i} p_i^*  E_{jj
} p _ i =
\sum_{j=1}^{\nu_i}
(e_j^*p_i)^*(e_j^*p_i),\]
so that
\[ \sum_{i}^{\finite} p_i^*p_i + \sum_j^{\finite} q_j^* L q_j= \sum_{i}^{\finite} \sum_{j=1}^{\nu_i}
(e_j^*p_i)^*(e_j^*p_i) +  \sum_j^{\finite} q_j^* L q_j \in \setmult{\RR^{\ell \times 1}}{I} +
\setmult{I^*}{\RR^{1 \times \ell}}.\]
Since $I$ is $L$-real, each $e_j^*p_i
\in I$ and each $Lq_j\in \R^{\nu\times 1}I$.
Therefore, for each $i$,
\[p_i = \operatorname{Id}_{\nu_i}p_i = \sum_{j=1}^{\nu_i} e_je_j^*p_i \in
\setmult{\RR^{\nu_i \times
1}}{I}. \qedhere
\]
\end{proof}

\begin{cor}
\label{cor:diffMiReal}
 A left module $I \subset \RR^{1 \times \ell}\axs$
is real if and only if
whenever
\begin{equation}
 \label{eq:diffMi2}
\sum_i^{\finite} p_i^*p_i 
\in \setmult{\RR^{\ell \times 1}}{I} +
\setmult{I^{\ast}}{\RR^{1 \times \ell}},
\end{equation}
for some $p_i \in
\RR^{\nu_i
\times \ell}\axs$, then
each $p_i \in \setmult{\RR^{\nu_i \times 1}}{I}$.
\end{cor}

Here is a connection between vanishing sets and real left modules:

\begin{prop}
\label{prop:cIcCIsReal}
Let $V \subset \bigcup_{n \in
\NN} (\RR^{n \times n})^g \times
\RR^{\ell n}$. The space  
\[
\cJ_L(V) := \{p\in\R^{1\times\ell}\axs \mid p(X)v=0 \text{ for all } (X,v)\in \cI(V) 
\text{ satisfying } L(X)\succeq0\}
\]
is an $L$-real left
module.
\end{prop}

\begin{proof}
 Suppose
\[\sum_i^{\finite} p_i^*p_i + \sum_j^{\finite} q_j^*Lq_j \in \RR^{\ell \times 1} \cJ_L(V) +  \cJ_L(V)^*\RR^{1
\times \ell},\]
where each $p_i \in \RR^{1 \times \ell}\axs$ and each $q_j\in\R^{\nu\times\ell}$.
For each $(X,v) \in V$ with $L(X)\succeq0$,  we have
\[\sum_i^{\finite} v^* p_i(X)^*p_i(X)v
+ \sum_j^{\finite} v^* q_j(X)^* L(X) q_j(X) v  = 0.
\]
Therefore $p_i(X)v = 0$ and $L(X)q_j(X)v=0$, which implies that each $p_i \in \cJ_L(V)$,
and each $Lq_j\in\R^{\nu\times1}\cJ_L(V)$.
\end{proof}

\begin{cor}
\label{cor:cIcCIsReal}
Let $V \subset \bigcup_{n \in
\NN} (\RR^{n \times n})^g \times
\RR^{\ell n}$. The space $\cI(V) \subset \RR^{1 \times \ell}\axs$ 
is a real left
module.
\end{cor}

\subsubsection{The Real Radical}
We now introduce a generalization of the real radical of a left module
for use in studying the positivity set of a linear pencil $L$.
Just like the vanishing radical of $I$ consists of polynomials vanishing on the variety $V(I)$, 
the $L$-real radical of $I$ consists of polynomials vanishing on the intersection of $V(I)$ with the positivity
set $\cD_L$ of $L$, see Proposition \ref{prop:zeroOnModPosIntro} below.

An intersection of $L$-real left modules is itself an $L$-real left module.
Define the
\textbf{$L$-real radical}\index{real!radical} of a left module $I \subset \RR^{1
\times \ell}\axs$ to be
\[\Lrad{L}{I} = \bigcap_{\substack{J \supseteq I\\ J\ \mbox{\tiny $L$-real}}} J
=\text{the smallest $L$-real left module containing } I.\]
The $L$-real radical of $I$ with $L=1$ is called the \df{real radical} of $I$ and
denoted by $\rr I$.
As we explain later,
the article \cite{N} in
 \S9.1 presents an algorithm for computing
$\rr{I}$ for a finitely-generated left module $I \subset \RR^{1 \times
\ell}\axs$.

Proposition \ref{prop:cIcCIsReal} implies that for each left module
$I \subset \RR^{1\times \ell}\axs$,
\[ I \subset \rr{I} \subset \sqrt{I}.\]
Much more difficult to prove is $ \rr{I} = \sqrt{I}$
for finitely generated $I$,
and this is \cite[Theorem 1.3]{N}. 
We also describe this here in  \S\ref{subsub:lnss}
in the context of our more general theory.

That $L$-real radicals are closely related to vanishing-positivity is shown in the next proposition,
where we show that  $\Lrad{L}{I}$ is the vanishing ideal 
(i.e., a ``free Zariski closure'') of $V(I)\cap \cD_L$. 
More precisely,
$ \Lrad{L}{I}= \cI\big( \cZ ( V(I) \cap \cD_L ) \big).$

\begin{prop}
 \label{prop:zeroOnModPosIntro}
 Let $L \in \RR^{\nu \times \nu}\axs$
be a linear pencil. 
Let $I \subset \RR^{1 \times \ell}\axs$ be a finitely-generated left module,
and let $p \in \RR^{1 \times \ell}\axs$.
Then $p(X)v = 0$ whenever $(X,v) \in
V(I)$ and $L(X)
\succeq 0$ if and only if $p \in \Lrad{L}{I}$.
\end{prop}

The proof requires some of the heaviest 
results of this paper and is presented in \S\ref{subsec:cor410}.

\subsection{Right Chip Spaces}
\label{sub:rightChip}

We now introduce  a natural class of polynomials  
needed for the proofs, chip spaces.
Also we state our main theorems in terms of chip spaces
since
keeping track of the chip space where each polynomial lies
 adds significant generality, and leads to optimal degree and size bounds; cf.~\cite{KP10}.

Consider $\RR^{1\times\ell}\axs$.
 A monomial $e_i \otimes a$ \textbf{divides} another monomial $e_j
\otimes b$ \textbf{on the right}\index{divisibility on the right} if $i = j$ and
$b = w a$ for some $w \in \axs$ so that $w(e_i\otimes a) =
e_j \otimes b$.
If additionally $e_i \otimes a \neq e_j \otimes b$, then
$e_i \otimes a$
\textbf{properly divides}\index{divisibility on the right!proper} $e_j \otimes
b$ \textbf{on the right}.
We call $e_i \otimes a$ a \textbf{(proper) right chip} of
$e_j \otimes b$ if $e_i \otimes a$ (properly) divides it on the right.

A vector subspace $\rC \subset \RR^{1 \times \ell}\axs$
 is a \textbf{right chip
space}\index{right chip space}
if $\rC$ is spanned by a set of monomials such that
whenever $e_i \otimes w_1w_2w_3,\,
e_i \otimes w_3 \in \rC$
for some $w_1, w_2, w_3 \in \axs$, then $e_i \otimes w_2w_3 \in \rC$.
A right chip space $\rC$ is \textbf{finite}\index{right chip
space!finite}
if $\rC$ is finite dimensional.
A right chip space $\rC$ is \textbf{full} \index{right chip
space!full}
if for each $e_i \otimes w \in \rC$, all right chips of
$e_i \otimes w$ are in $\rC$ as well.

\begin{exa}\label{ex:ezChip}
 The space $\RR^{1 \times \ell}\axs_d$, the space of all $1\times\ell$ matrix NC polynomials of
degree bounded by $d$, is a full, finite right chip space.
\end{exa}

For a finitely-generated left module $I \subset \RR^{1 \times \ell}\axs$ we
can  find full, finite right chip spaces $\rC \subset \RR^{1 \times
\ell}\axs$ such that the generators of $I$ are in the space $\RR\axs_1 \rC$.

\begin{exa}
 Let $I \subset \RR^{1 \times \ell}\axs$ be generated by some
polynomials in the span of the monomials $m_1, \ldots, m_k$. The space
\[\rC := \operatorname{span} \{ m \in \RR^{1 \times \ell}\axs \mid m \mbox{
a proper right chip of some } m_i\}\]
is a full, finite right chip space such that $\RR\axs_1 \rC$ contains all the
generators of $I$.
\end{exa}

At first reading the main results of this paper, soon to be stated, the reader
should just think of $\rC$ as being  $\RR^{1 \times \ell}\axs_d$, cf.~Example \ref{ex:ezChip}.

An appeal of right chip spaces is they are easily
computable.

\subsection{\texorpdfstring{$(L, \rC)$}{[(L, M)]}-Real Radical Modules}\label{sec:introLrad}
Chip spaces 
 lead to an extension of the notion of an
 $L$-real radical of a left module.

Let $I \subset \RR^{1 \times \ell}\axs$ be a left module, let $L \in
\RR^{\nu \times \nu}\axs$, and let $\rC \subset \RR^{1 \times
\ell}\axs$ be a right chip space.
We say that $I$ is
\textbf{$(L, \rC)$-real}\index{L,M
real@$(L,\rC)$-real} if
whenever 
\begin{align}
\label{eq:qradCond}
\sum_{i}^{\finite} p_i^{\ast}p_i + \sum_{j}^{\finite}
q_j^{\ast}Lq_j \in \RR^{\ell \times 1} I + I^{\ast}\RR^{1 \times \ell}
\end{align}
for some $p_i \in \rC$ and $q_j \in
\RR^{\nu \times 1}\rC$, then each
$p_i \in I$ and each $Lq_j \in \RR^{\nu \times 1}I$.
We say $I$ is \textbf{strongly
$(L, \rC)$-real}\index{L,M real@$(L,d)$-real!strong(ly)} if
whenever
\eqref{eq:qradCond} holds, then
each $p_i \in I$ and each $q_j \in \RR^{\nu \times 1}I$.

Define the
\textbf{$(L,\rC)$-real radical}\index{$\Lradd{L}{\rC}{I}$} of $I$
to be
\[ \Lradd{L}{\rC}{I} = \bigcap_{\substack{J \supset I\\J\
(L,\rC)\text{-real}}}
J =
\text{smallest}\ (L,\rC)\text{-real left module containing }  I.\]
Define the \textbf{strong $(L,\rC)$-real radical} of $I$ to be
\[ \LraddS{L}{\rC}{I} = \bigcap_{\substack{J \supset I\\J\
\text{strongly}\\ (L,\rC)\text{-real}}} J =
\text{smallest strongly}\ (L,\rC)\text{-real left module containing }  I.\]
If $\rC=\RR^{1 \times \ell}\axs$, we omit it and talk about $L$-real modules,
and the (strongly) $L$-real radical of $I$.
These definitions extend the notion of a real module as given in \S\ref{sec:preLrad}, 
e.g.~a left module $I$ is real if it is $1$-real.
As we will see in  \S\ref{sect:CompRad} all these real radicals are algorithmically  computable.

\subsection{Overview. The main results and some consequences}
\label{sub:overview}

The main general result of this paper is the following theorem, proved in \S\ref{sec:mainResults}.
At first reading the reader is advised to think of
$\rC$ as all of $\R^{1\times\ell}\axs$.

\begin{theorem}
\label{thm:mainNMon}
Suppose $L \in \RR^{\nu \times \nu}\axs$ is a linear pencil.
Let $\rC \subset
\RR^{1 \times \ell}\axs$
be a finite chip space,
let $I
\subset \RR^{1 \times \ell}\axs$ be a
left $\RR\axs$-module generated by polynomials in $\RR\axs_1 \rC$,
and let $p \in \rC^*\RR\axs_1 \rC$ be a symmetric polynomial.
\begin{enumerate}[\rm(1)]
 \item $v^{\ast}p(X)v \geq 0$ whenever $(X,v) \in V(I)$ and $L(X) \succeq
0$ if and only if $p$ is of the form
\begin{equation}
\label{eq:clasPosIdeal}
  p = \sum_{i}^{\finite}p_i^{\ast}p_i + \sum_{j}^{\finite} q_j^{\ast}Lq_j +
  \sum_k^{\finite} (r_k^{\ast}\iota_k + \iota_k^{\ast}r_k)
\end{equation}
where each $p_i, r_k \in \rC$, each $q_j \in
\RR^{\ell \times 1} \rC$ and each $\iota_k \in \Lradd{L}{\rC}{I} \cap \RR\axs_1
\rC$.
\item $v^{\ast}p(X)v \geq 0$ whenever $(X,v) \in V(I)$ and $L(X) \succ 0$,
if and only if $p$ is of the form
\eqref{eq:clasPosIdeal}
where each $p_i, r_k \in \rC$, each $q_j \in
\RR^{\nu \times 1}\rC$ and
each $\iota_k \in \LraddS{L}{\rC}{I} \cap \RR\axs_1 \rC$
\end{enumerate}
\end{theorem}

\begin{remark}
Our machinery of chip spaces allows us to give additional information on
the size of the testing matrices $X$
in (1) and (2).
Indeed, for certain cases we shall obtain provably optimal size and degree bounds;
 see \S\ref{subsec:sizeBd}  for details.
\end{remark}

Theorem \ref{thm:mainNMon} is a general theorem
from which we  deduce several interesting corollaries.

\subsubsection{A Real Nullstellensatz for $\RR^{\nu \times \ell}\axs$}

One corollary is  \cite[Theorem 1.3]{N},
which is that paper's main theorem and is
a generalization of the 
Real Nullstellensatz from \cite{chmn}.
The heavy machinery developed in \cite{N} which is used to prove this
is also essential to many proofs in this paper.

\begin{corollary}[\protect{\cite[Theorem 1.3]{N}}]
\label{thm:mainFromNotes}
 Let $p_1, \ldots, p_k$ be such that each
$p_i \in \RR^{\nu_i \times \ell}\axs$ for some $\nu_i \in \NN$.
Suppose $q \in \RR^{\nu \times \ell}\axs$, with $\nu \in \NN$, has the property
that
whenever $p_1(X)v, \ldots, p_k(X)v = 0$,
where $(X,v) \in  \bigcup_{n \in \NN} (\RR^{n \times n})^g \times
\RR^{\ell n}$, then $q(X)v = 0$.
Then $q$ is an element of the space $I$ defined by
\[
 I := \setmult{\RR^{\nu \times 1}}{\rr{\sum_{i=1}^k
\setmult{\RR^{1
\times \nu_i}\axs}{ p_i}}}.
\]

Consequently, if the left module
\begin{equation}
\label{eq:checkIReal}
\sum_{i=1}^k
\RR^{1
\times \nu_i}\axs p_i
\end{equation}
is real, and if $q(X)v = 0$
whenever $p_1(X)v, \ldots, p_k(X)v = 0$, then
$q$ is of the form
\[q = r_1p_1 + \cdots + r_k p_k\]
for some $r_i \in \RR^{\nu \times \nu_i}\axs$.
\end{corollary}

\begin{proof}
See \S\ref{subsub:lnss}.
\end{proof}

\subsubsection{Convex Positivstellensatz}

Applying Theorem \ref{thm:mainNMon} in the case  $I = \{0\}$ gives an
extension of the Convex Positivstellensatz of Helton, Klep, and McCullough
\cite{HKMb} to the case where the positivity set $\cD_L$ of a linear pencil $L$
may
have  empty interior.  This  is given in Corollary \ref{cor:radZero}.
Note also that Corollary \ref{cor:radZero} gives a substantial refinement of the
degree bounds 
obtained in \cite{HKMb} by using right chip spaces. 

\subsubsection{Thick Pencils}

Some basic properties of $L$-real radicals follow from Theorem \ref{thm:mainNMon}.

\begin{prop} 
\label{prop:LradL}
 Let $L \in \RR^{\ell \times \ell}\axs$ be a monic linear pencil and let $I_L =
\RR^{1 \times \ell}\axs L$.
 Then
 $I_L= \sqrt[\real]{I_L} = \sqrt{I_L}$.
\end{prop}
This 
will be proved  in \S\ref{sec:thick}.

We note 
 $L$ having the zero determining property
\eqref{eq:conj+}
 is equivalent to   the statement
 $\Lrad{L}{I_L} = I_L$. 
This is a consequence of Propositions \ref{prop:LradL} and 
\ref{prop:zeroOnModPosIntro}, 
and by Corollary \ref{cor:mindegyes} this holds for a generic and 
computationally checkable $L$.
Of course for any $L$ we have  $\cZ(  \dbcD_L^\circ) $ is contained in $V(L)$,
that is, $\Lrad{L}{I_L} \supset I_L$.

\subsubsection{Thin Pencils}

In \S\ref{sect:L0} we will use the main theorem of this paper, Theorem
\ref{thm:mainNMon},
to prove results about LMIs with empty interior, i.e., thin spectrahedra.  
More precisely, if $L$ is a linear pencil which defines a thin spectrahedron we 
will apply the main theorem to $L$ with $I = \{0\}$ to prove results 
about thin spectrahedra. 
In \S\ref{alg:Lradd0} we give an efficient algorithm for computing 
the affine hull 
(i.e., the smallest affine subspace containing 
it)
of the thin spectrahedron  $\cD_L(1)$.

There is an appealing connection between the space $\Lrad{L}{\{0\}}$ and its
corresponding ideal in $\RR[x]$, the space of polynomials in commuting
variables.
We say that the {\bf commutative collapse} of a polynomial $p \in \RR\axs$ to
$\RR[x]$ is
the polynomial produced by letting the variables in $p$ commute and setting $x =
x^*$.  The {\bf
commutative collapse} of a subset $S \subset \RR\axs$ to $\RR[x]$ is
the set of projections of all the elements of $S$ to $\RR[x]$.
There is a natural decomposition of a thin linear pencil in terms of
a thick one restricted to a special subspace as we now describe.

\begin{thm}
\label{thm:geom}
Let $L \in \RR^{\nu \times \nu}[x]$ be a linear pencil, where $x = (x_1,
\ldots, x_g)$ is a tuple of commuting variables.   Let $I \subset \RR[x]$ be the
commutative collapse of $\Lrad{L}{\{0\}}$ onto $\RR[x]$. 
\begin{enumerate}[\rm(1)]
\item $I \subset \RR[x]$ is an ideal generated by linear polynomials.
\item There exists a linear pencil $\tilde{L} \in \RR^{\nu' \times \nu'}[x]$,
where $\nu' \leq \nu$, whose positivity set has nonempty interior such that
\[
\{x \in \RR^g \mid L(x) \succeq 0\} = \{x \in \RR^g \mid \tilde{L}(x) \succeq 0
\mbox{ and } \iota(x) = 0 \mbox{ for each } \iota \in I\}.
\]
\end{enumerate}
\end{thm}

The proof of Theorem \ref{thm:geom} is based on taking $I=\{0 \}$, $p=1$ and
will be given in \S\ref{sub:geomInt}. 

Geometrically,  Theorem \ref{thm:geom} implies  that given  a
linear pencil $L$ which defines a spectrahedron with empty interior,  either the
spectrahedron $\cD_L(1)$ is empty---that is, $L(x) \succeq 0$ is infeasible---or
it can be
viewed as a spectrahedron with non-empty interior lying inside a proper affine
subspace of $\RR^g$.
In \S\ref{sect:algorithms} we give an
algorithm for computing the ideal $I \subset \RR[x]$ and the linear pencil
$\tilde{L} \in \RR^{\nu' \times \nu'}[x]$ described in Theorem \ref{thm:geom}.
In particular, we will see that the algorithm discussed in Theorem \ref{thm:LraddAlg}
is a generalization of the process of finding the affine subspace on which a
spectrahedron lies, as given in \cite{KS}.

\subsubsection{Algorithms}
Applying Theorem \ref{thm:mainNMon} requires  one to compute the $(L,
\rC)$-real radical of a left module $I$.  In \S\ref{sect:CompRad} we will 
present an algorithm for doing so.
In addition, in \S\ref{sect:CompRad} we also give more refined algorithms
for the special cases of computing $\sqrt[\real]{I}$ and $\Lrad{L}{\{0\}}$.
This generalizes the algorithm for the special case
 $L =
1$
found in \cite{N}.
Here is a theorem listing the algorithms' desirable
properties.  We emphasize this algorithm works even for polynomials $L \in 
\RR^{\nu \times \nu}\axs$ which are not linear.
\begin{theorem}
 \label{thm:LraddAlg}
Let $L \in
\RR^{\nu \times
\nu}\axs_{\sigma}$ be a symmetric polynomial, let $\rC \subset \RR^{1
\times \ell}\axs$ be a
finite right chip space, and let $I\subset \RR^{1 \times \ell}\axs$ be a left
module.
The $L$-Real Radical algorithm for $\Lradd{L}{\rC}{I}$ in \S{\rm\ref{alg:Lraddd}}
has
the
following properties.
\begin{enumerate}[\rm(1)]
 \item The algorithm terminates in a finite number of steps.
 \item If $I$ is generated by polynomials in $\RR\axs_{\sigma}\rC$,
 then the algorithm involves computations on polynomials in
$\RR\axs_{\sigma}\rC$.
 \item The algorithm outputs a left Gr\"obner basis for
$\Lradd{L}{\rC}{I}$.
\end{enumerate}
\end{theorem}

\subsubsection{Completely Positive Maps}

In \S\ref{sec:cp} we apply our results on thin spectrahedra to 
 give  algebraic certificates for completely positive maps between
(nonunital) subspaces of matrix algebras. We shall see that complete positivity
of a map is equivalent to 
 LMI domination between a pair of associated linear pencils.

\subsection{Context and Reader's Guide}\label{sec:intro guide}

To give a broad perspective on the topic of this paper 
we point out that it fits in the area of Free Real Algebraic Geometry.
This in turn  lies within the 
 booming area called Free Analysis 
 the earliest and most developed branch of which is 
  Free Probability, see \cite{VDN92} for a survey.
  Also developing rapidly is Free Analytic Function Theory, see
 \cite{Voi04,Voi10,KVV+,MS11,Pope10,AM+,BB07}. 
We refer the reader to \href{http://math.ucsd.edu/~ncalg}{\tt NCAlgebra} \cite{HOSM} and \href{http://ncsostools.fis.unm.si}{\tt NCSOStools} \cite{CKP11} for computer algebra packages adapted to deal with free noncommuting variables.

While Free Null-Positivstellens\"atze as we develop in
this paper date back less than a decade, already
Free Positivstellens\"atze 
have found  physical applications. 
For instance,
applications to quantum physics are explained
by Pironio, Navascu\'es, Ac\'\i n \cite{PNA10} 
who also consider computational aspects related to 
noncommutative sum of squares.
Doherty, Liang, Toner, Wehner \cite{DLTW08}
employ free positivity and the Positivstellensatz \cite{HM} 
 to consider
the quantum moment problem and multi-prover games.

Turning from the general to the very specific
we describe the organization of the rest of this paper.
\S\ref{sect:radIdeals} proves some basic results about $L$-real radicals and
$(L, \rC)$-left modules.
\S\ref{sect:linFun} describes how to construct positive
linear functionals on spaces of square matrices of NC polynomials
for use in the proof of the main theorem.
\S\ref{sec:mainResults} proves the main result, Theorem \ref{thm:mainNMon}, and
many of the corollaries of this paper.
\S\ref{sec:thick} proves Theorem \ref{thm:detBdry2} and Proposition
\ref{prop:LradL}, which pertain to thick spectrahedra. 
\S\ref{sect:L0} characterizes the $L$-real radical of $\{0\}$, which pertains
to thin spectrahedra.
\S\ref{sect:CompRad} describes algorithms for computing
different real radicals appearing in our main results; many of these algorithms are
improvements on previously known algorithms.
\S\ref{sec:cp} gives nonlinear algebraic certificates for complete positivity of
maps between (nonunital) subspaces of matrix algebras.
Finally, \S\ref{sec:symmVars} gives direct analogs of the results of this paper
in
the case where all the variables $x_j$ are symmetric.

\subsubsection{Acknowledgments} 
The authors want to thank Scott McCullough, Mauricio de Oliveira,
Daniel Plaumann and Rainer Sinn for discussions and sharing their expertise.

\section{Properties of $(L, \rC)$-Real  Left Modules}
\label{sect:radIdeals}

In this section we prove some useful properties of $(L, \rC)$-real left
modules $I \subset \RR^{1 \times \ell}\axs$. Here $L\in
\RR^{\nu \times \nu}\axs$ is a matrix polynomial, and  $\rC \subset \RR^{1 \times
\ell}\axs$ is a right chip space.

\subsection{$L$-Real Left Modules}

One class of $L$-real left modules which arise naturally are left modules
$\cI\big(\{(X,v)\}\big)$, where $X$ is a tuple of matrices with $L(X) \succeq 0$;
cf.~Proposition \ref{prop:cIcCIsReal} in \S\ref{subsec:Lreal}.

\begin{prop}
\label{prop:idealOfVarIsLRad}
Let $L \in \RR^{\nu \times \nu}\axs$ be a symmetric polynomial, and let $X \in
(\RR^{n \times n})^g$ be such that $L(X) \succeq 0$.
For each vector $v \in \RR^{\ell n}$, the left module $\cI\big(\{(X,v)\}\big)$
is
$L$-real.
If also $L(X) \succ 0$, then $\cI\big(\{(X,v)\}\big)$ is strongly
$L$-real for each $v$.
\end{prop}

\begin{proof}
Suppose
\[\sum_{i}^{\finite} p_i^{\ast}p_i +
\sum_{j}^{\finite} q_j^{\ast}Lq_j \in
\RR^{\ell \times 1}\cI(\{(X,v)\}) + \big[\cI(\{(X,v)\})\big]^{\ast}\RR^{1
\times
\ell}\]
for some polynomials $p_i \in \RR^{1 \times \ell}\axs$ and $q_j \in
\RR^{\nu \times \ell}\axs$. For each $\iota \in
\cI(\{(X,v)\})$, we have 
\[v^*\iota(X)v = 0 \qquad\text{and}\qquad
v^{\ast}\iota(X)^{\ast}v = 0.
\]
Therefore
\[v^{\ast}\left( \sum_{i}^{\finite} p_i(X)^{\ast}p_i(X) +
\sum_{j}^{\finite} q_j(X)^{\ast}L(X)q_j(X) \right)v = 0. \]
For each $i$, 
and, since $L(X) \succeq 0$, for each $j$ we
have
\[v^{\ast}p_i(X)^{\ast}p_i(X)v \geq 0
\qquad\text{and}\qquad
v^{\ast}q_j(X)^*L(X)q_j(X)v \geq 0.
\]
Therefore, for each $i$,
\[ v^{\ast}p_i(X)^{\ast}p_i(X)v = \|p_i(X)v\|^2 = 0,\]
and for each $j$,
\[ v^{\ast}q_j(X)^{\ast}L(X)q_j(X)v = \|\sqrt{L(X)} q_j(X)v\|^2 = 0.\]
Hence each $p_i(X)v = 0$, or equivalently, $p_i \in \cI(\{(X,v)\})$.
Further,
 each $\sqrt{L(X)}q_j(X)v = 0$, so $L(X)q_j(X)v =
0$, which implies $Lq_j \in \RR^{\nu \times 1}\cI(\{(X,v)\})$.
If in addition $L(X) \succ 0$, then $L(X)$ is invertible, so
$L(X)q_j(X)v = 0$ if and only if $q_j(X)v = 0$, which implies $q_j \in
\RR^{\nu \times 1}\cI(\{(X,v)\})$.
\end{proof}

\subsection{Homogeneous Left Modules}

We next consider $\Lrad{L}{I}$ for a homogeneous
left module  $I$.
A left module $I \subset \RR^{\nu \times \ell}\axs$ is
\textbf{homogeneous}\index{homogeneous left module} if
it is generated by homogeneous polynomials.

\begin{prop}
 Let $I \subset \RR^{\nu \times \ell}\axs$ be a left module.  The following are
equivalent:
\begin{enumerate}[\rm(i)]
 \item $I$ is homogeneous;
 \item $p \in I$ if and only if $p$ is a sum of homogeneous polynomials in $I$;
 \item $p \in \RR^{\ell \times \nu} I + I^* \RR^{\nu \times \ell}$ if and only
if $p$ is a sum of homogeneous elements of $\RR^{\ell \times \nu} I + I^*
\RR^{\nu \times \ell}$;
\item $p \in \RR^{\ell \times \nu} I + I^* \RR^{\nu \times \ell}$ if and only
if $p$ is a sum of homogeneous polynomials which are each in $\RR^{\ell \times
\nu} I$ or $I^* \RR^{\nu \times \ell}$.
\end{enumerate}
\end{prop}

\begin{proof}
 Straightforward.
\end{proof}

Recall a linear pencil $L$ is monic if its constant term is the identity matrix.

\begin{prop}
\label{prop:qRadRealHom}
 Let $I \subset \RR^{1 \times \ell}\axs$ be a homogeneous left module, and
let $L \in \RR^{\nu \times \nu}\axs$ be a monic linear pencil.
The following are equivalent:
\begin{enumerate}[\rm(i)]
\item\label{item:real} $I$ is real;
 \item\label{item:qrad} $I$ is $L$-real;
 \item\label{item:qradstr} $I$ is strongly $L$-real.
\end{enumerate}
\end{prop}

\begin{proof}
By definition, \eqref{item:real} $\Leftarrow$ \eqref{item:qrad} $\Leftarrow$
\eqref{item:qradstr}.
Therefore suppose that $I$ is real.
Let
\begin{align}
\label{eq:sumSigma}
 \sum_{i}^{\finite} p_i^{\ast}p_i + \sum_{j}^{\finite} q_j^{\ast}Lq_j \in
\RR^{\ell \times 1} I + I^{\ast}\RR^{1 \times \ell},
\end{align}
where each $p_i \in \RR^{1 \times \ell}\axs$ and each $q_j \in \RR^{\nu \times
\ell}\axs$.
Let $\delta$ be the minimum degree such that at least one of the $p_i$ or
$q_j$ have
terms of degree $\delta$.  Let $\tilde{p}_i$ and $\tilde{q}_j$ be the
terms of
$p_i$ and $q_j$ respectively with degree $\delta$.  The terms
of \eqref{eq:sumSigma} with degree $2\delta$ are
\begin{align}
 \label{eq:sumSmall}
\sum_{i}^{\finite} \tilde{p}_i^{\ast}\tilde{p}_i + \sum_{j}^{\finite}
\tilde{q}_j^{\ast}\tilde{q}_j.
\end{align}
  Since $I$ is homogeneous, \eqref{eq:sumSmall} must be in
$\RR^{\ell \times 1} I + I^{\ast}\RR^{1 \times \ell}$.
Since $I$ is real, each $\tilde{p}_i \in I$ and $\tilde{q}_j \in
\RR^{\nu \times 1}I$.  Therefore,
\[  \sum_{i}^{\finite} (p_i - \tilde{p}_i)^{\ast}(p_i - \tilde{p}_i) +
\sum_{j}^{\finite} (q_j-\tilde{q}_j)^{\ast}L(q_j-\tilde{q}_j) \in
\RR^{\ell \times 1} I + I^{\ast}\RR^{1 \times \ell}.\]
We repeat this process to show that each homogeneous part of $p_i$ is in
$I$ and each homogeneous part of $q_j$ is in $\RR^{\nu \times 1}I$.  Hence
$I$ is strongly $L$-real.
\end{proof}

A special example of a homogeneous left module is $\{0\}$.
Proposition \ref{prop:qRadRealHom} implies that $\Lrad{L}{\{0\}} = \{0\}$ if $L$
is monic.
In the non-monic case---and in particular, if $\cD_L=\{X \mid L(X)
\succeq 0\}$ has empty interior---there is more to say
about $\Lrad{L}{\{0\}}$, as we will see in \S\ref{sect:L0}.

\subsection{$(L, \rC)$-Real Left Modules for Finite Right Chip Spaces}

For  a finite
right chip space
$\rC$ and a left module $I$,
the $(L,\rC)$-real radical  $\Lradd{L}{\rC}{I}$
is generated as a left module by polynomials in a restricted
vector subspace, as shown in the following proposition.

\begin{prop}
\label{prop:genOfQDRad}
Let $\rC \subset \RR^{1 \times \ell}\axs$ be a finite right chip space,
let $I \subset \RR^{1 \times \ell}\axs$ be a left module, and let $L \in
\RR^{\nu \times \nu}\axs$ be a symmetric polynomial of degree $\sigma$.  The
left
module
$\Lradd{L}{\rC}{I}$ is generated by $I$
together with some subset of polynomials in $\RR\axs_{\sigma}\rC$,
and $\LraddS{L}{\rC}{I}$ is generated by $I$
together with some subset of polynomials in $\rC$.
\end{prop}

\begin{proof}
We will construct an increasing chain of left modules $I^{(a)}$ such that
\[ I \subsetneq I^{(1)} \subsetneq \cdots \subsetneq I^{(k)} =
\Lradd{L}{\rC}{I}.\]
Suppose inductively that $I^{(a)} \subset \Lradd{L}{\rC}{I}$ is
generated by $I$
and by some
polynomials in $\RR\axs_{\sigma} \rC$.
 Consider a polynomial
\begin{align}
\label{eq:aSumNotIn}
\sum_{i}^{\finite} p_i^{\ast}p_i + \sum_{j}^{\finite}
q_j^{\ast}Lq_j \in \RR^{\ell \times 1} I^{(a)} +
\left(I^{(a)}\right)^{\ast}\RR^{1 \times
\ell}
\end{align}
for some $p_i \in \rC$ and $q_j \in
\RR^{\nu \times 1}\rC$.  Since $I^{(a)}
\subset \Lradd{L}{\rC}{I}$, we have that each $p_i \in
\Lradd{L}{\rC}{I}$ and
$Lq_j \in \RR^{\nu \times 1}\Lradd{L}{\rC}{I}$, which implies that
$e_k Lq_j
\in \Lradd{L}{\rC}{I}$ for each standard unit vector $e_k \in \RR^{1
\times
\nu}$.
If not all of the $p_i \in I^{(a)}$ and $Lq_j \in \RR^{\nu
\times 1} I^{(a)}$, then let $I^{(a + 1)}$ be generated by $I^{(a)}$ and by the
$p_i$ and $e_k L q_j$.  In this case, $I^{(a)} \subsetneq I^{(a+1)} \subset
\Lradd{L}{\rC}{I}$.  Furthermore, $I^{(a+1)}$ is generated by $I$ and some polynomials
in $\RR\axs_{\sigma} \rC$.

This process must terminate since $\RR\axs_{q}\rC$ is finite
dimensional.  Therefore we arrive at a point where \eqref{eq:aSumNotIn} holds
if and only if $p_i \in I^{(a)}$ and $Lq_j \in \RR^{\nu
\times 1} I^{(a)}$.  At this point, $I^{(a)}$ is $(L,\rC)$-real and
$I \subset
I^{(a)} \subset \Lradd{L}{\rC}{I}$.  Hence $I^{(a)} =
\Lradd{L}{\rC}{I}$.

The $\LraddS{L}{\rC}{I}$ case is similar and its proof is omitted.
\end{proof}

We will give algorithms for computing $\Lradd{L}{\rC}{I}$
in \S\ref{sect:CompRad}.

\section{Positive Linear Functionals on \texorpdfstring{$\RR^{\ell \times
\ell}\axs$}{[R ell times ell x xs]}}
\label{sect:linFun}

This section contains fundamental properties of positive linear functionals on 
$\Rll\axs$. 

A $\R$-linear functional $\lambda$ on $W \subset \RR^{\ell \times \ell}\axs$ is
\textbf{symmetric}\index{linear functional on $\RR^{\ell \times
\ell}$!symmetric}
if $\lambda(\omega^{\ast}) = \lambda(\omega)$ for each pair $\omega, \omega^*
\in W$.
A linear functional $\lambda$ on a subspace $W \subset \RR^{\ell \times
\ell}\axs$ is
\textbf{positive}\index{linear functional on $\RR^{\ell \times
\ell}\axs$!positive} if it is symmetric and
if $\lambda(\omega^{\ast}\omega) \geq 0$ for each  $\omega^*\omega \in W$.

\subsection{The GNS Construction}

Proposition \ref{prop:GNS} below describes a variant of the well-known
 Gelfand-Naimark-Segal (GNS) construction.

\begin{prop}
\label{prop:GNS}
Let $\lambda$ be a positive 
linear functional on
$\RR^{\ell \times \ell}\axs$,
and let \[I = \{ \vartheta \in \RR^{1 \times \ell}\axs \mid
L(\vartheta^*\vartheta) =
0\}.\]
There exists an inner product on the quotient space 
$\cH := \RR^{1 \times \ell}\axs / I$,
a tuple of operators $X$ on $\cH$, and a
vector $v \in \cH^n$ such that for each $p
\in \RR^{\ell \times \ell}\axs$ we have
\[\langle p(X)v, v \rangle = \lambda(p),\]
 and $\cH = \{q(X)v \mid q \in\RR^{1 \times \ell}\axs\}$.
\end{prop}

\begin{proof}
The proof follows the classical argument; alternately, see \cite[Proposition 5.3]{N} for a detailed proof.
\end{proof}

We shall apply Proposition \ref{prop:GNS} in the next subsection to 
``flat'' linear functionals, in which case the obtained quotient space $\cH$ is 
finite-dimensional, and $X$ is thus simply a tuple of matrices. We refer
to \cite{Pop10,HKMb} for more on flat linear functionals in a free algebra.

\subsection{Flat Extensions of Positive Linear Functionals}
We next turn to flat extensions of positive linear functionals on $\Rll\axs$. 
The reader is referred to \cite{CF96,CF98} for the classical theory of flatness on $\R[x]$.
The content of this subsection comes from \cite{N} and is summarized now for future use.

Let $W \subset \RR^{\ell \times \ell}\axs$ be a vector subspace and
let $\lambda$ be a positive linear functional on $W$.
Suppose
\[ \{\omega \in \RR^{1 \times \ell}\axs \mid \omega^*\omega \in W \} = J \oplus
T \]
where $J, T \subset \RR^{1 \times \ell}\axs$ are
vector
subspaces with
\[J := \{\vartheta \in \RR^{1 \times \ell}\axs \mid \vartheta^*\vartheta \in W
\mbox{ and } \lambda(\vartheta^*\vartheta) = 0\}.\]
 An extension $\bar{\lambda}$ of
$\lambda$ to a space $U \supset W$ is a \textbf{flat
extension}\index{flat extension} if $\bar{\lambda}$ is positive and if
\[ \{u \in \RR^{1 \times \ell}\axs \mid u^*u \in U \} = I \oplus
T \]
where
\[ I = \{ \iota \in \RR^{1 \times \ell}  \mid \iota^*\iota \in U \mbox{ and }
\bar{\lambda}(\iota^*\iota) = 0\}.\]

\begin{prop}
\label{prop:flatExtRC}
Let $\rC \subset \RR^{1 \times \ell}\axs$
be a finite right chip space, and 
let $\lambda$ be a positive linear functional on
$\rC^*\RR\axs_1 \rC$.
\begin{enumerate}[\rm(1)]
 \item There exists a positive extension of
$\lambda$ to the space  $\rC^*\RR\axs_2 \rC$ if and only if
whenever
 $\vartheta \in \rC$ satisfies
$\lambda(\vartheta^*\vartheta) = 0$, then $\lambda(b^*c\vartheta) =
0$ for each polynomial $b \in \RR\axs_1 \rC$ and each $c \in \RR\axs$
satisfying $c\vartheta \in \rC$.
\item If there exists a positive extension of
$\lambda$ to the space
$\rC^*\RR\axs_2 \rC$, then there exists a unique flat
extension
$\bar{\lambda}$  of $\lambda$
to
$\rC^*\RR\axs \rC$.
In this case, the space
\[\{\theta \in \RR\axs\rC \mid
\bar{\lambda}(\theta^*\theta) = 0\}\]
is generated as a left module by the set
\begin{equation}
\label{eq:genOfJRC}
  \{ \iota \in \RR\axs_1 \rC \mid \lambda(b^*\iota) = 0
\text{ for every } b \in \rC\}.
\end{equation}
\item Given the existence of a flat extension $\bar{\lambda}$ 
of $\lambda$
to
$\rC^*\RR\axs\rC$, there exists a flat extension
of $\bar{\lambda}$ to all of $\RR^{\ell \times \ell}\axs$.
\end{enumerate}

\end{prop}

\begin{proof}
 See \cite[Proposition 5.2]{N}.
\end{proof}

\begin{corollary}
\label{cor:GNS}
Let $\rC \subset \RR^{1 \times \ell}\axs$
be a full, finite right chip space.
Let $\lambda$ be a positive linear functional on
$\rC^*\RR\axs_1 \rC$, and let
$J$ be the set
\[J := \{ \vartheta \in \rC \mid \lambda(\vartheta^*\vartheta) = 0\}.\]
Suppose that 
if $\vartheta \in J$,
then $\lambda(b^*c\vartheta) =
0$ for each polynomial $b \in \RR\axs_1 \rC$ and each $c \in \RR\axs$
such that $c\vartheta \in \rC$.
Let $n := \dim(\rC) -
\dim(J \cap \rC)$,
and suppose $n > 0$.
Then there exists a $g$-tuple $X$ of  $n \times n$ real matrices, and a
vector $v \in \RR^{\ell n}$ such that for each $p
\in\rC^*\RR\axs_1\rC$ we have
\[v^{\ast}p(X)v = \lambda(p),\]
 and $\RR^{\ell n} = \{p(X)v \mid p \in\rC\}$.
\end{corollary}

\begin{proof}
By Proposition \ref{prop:flatExtRC}, there exists a flat extension
$\bar{\lambda}$ of
$\lambda$ to all of $\RR^{\ell \times \ell}\axs$.
Given this flat extension, apply Proposition \ref{prop:GNS} to produce the
desired
$X$ and $v$.
\end{proof}

\subsection{Truncated Test Modules}
\label{sect:truncTestMod}

Let $L \in \RR^{\nu \times \nu}\axs$ for some  $\nu\in\N$.
Let $T \subset \RR^{1 \times \ell}\axs$ and $U \subset \RR^{\nu \times
\ell}\axs$ be
vector spaces.
Define $M_{T,U}(L)$ as
\begin{equation}
\label{eq:MVWab}
  M_{T,U}(L) := \left\{ \sum_{i}^{\finite}
t_i^{\ast}t_i + \sum_{j}^{\finite} u_j^{\ast}Lu_j \mid t_i \in T,\
u_j \in U\right\}.
\end{equation}
We call $M_{T,U}(L)$ a \textbf{truncated (quadratic) module}\index{truncated module}.

Let $I \subset \RR^{1 \times \ell}\axs$ be a left module, $L
\in \RR^{\nu \times \nu}\axs$ a symmetric polynomial, and $\rC
\subset \RR^{1 \times \ell}\axs$ a right chip space.
Decompose
$\rC$ as
\[ \rC = (I \cap \rC) \oplus T,\]
for some space $T \subset \rC$.
Decompose $\RR^{\nu \times 1}T$ as
\[ \RR^{\nu \times 1} T = J \oplus K,\]
where $J$ is the subspace of $\RR^{\nu \times 1}T$ defined by
\[J = \{ \vartheta \in \RR^{\nu \times 1}T \mid L \vartheta \in \RR^{\nu \times
1}I\},\]
and $K \subset \rC$ is some complementary subspace.
Since $J \cap K = \{0\}$,  we have that $L\kappa \not\in \RR^{\nu \times 1}I$
for each $\kappa \in
K
\setminus \{0\}$.
The following is a \textbf{truncated test module}\index{truncated test module}
for $I$,
$L$ and $\rC$:
\begin{equation}
 \label{eq:Md}
\Md := M_{T, K}(L) = \left\{ \sum_{i}^{\finite}
\tau_i^{\ast}\tau_i + \sum_{j}^{\finite} \kappa_j^{\ast}L\kappa_j \mid \tau_i
\in T,\
\kappa_j \in K\right\}.
\end{equation}

\begin{lemma}
\label{prop:lowdegqradcond}
 Let $I \subset \RR^{1 \times \ell}\axs$ be a left module,
let $L \in \RR^{\nu
\times \nu}\axs$ be a
symmetric polynomial,
and let $\rC \subset \RR^{1 \times \ell}\axs$ be a
finite right chip space.
Let $\Md = M_{T,K}(L)$ be a truncated test module for $I$, $L$ and $\rC$, as in
\eqref{eq:Md}.
If $I$ is $(L, \rC)$-real, then
\[ \big(\RR^{\ell \times 1} I + I^{\ast}\RR^{1 \times \ell}\big) \cap \Md = \{0\}. \]
\end{lemma}

\begin{proof}
Suppose
\[  \sum_i^{\finite} \tau_{i}^{\ast}\tau_{i} + \sum_{j}^{\finite}
\kappa_{j}^{\ast}L\kappa_{j} \in \RR^{\ell
\times 1} I + I^{\ast}\RR^{1 \times \ell},\]
where each $\tau_{i} \in T$ and each $\kappa_j \in K$.
Since $I$ is
$(L,\rC)$-real, it must be that each $\tau_i \in I$ and each $L\kappa_j
\in \RR^{\nu \times 1}I$, which implies that each $\tau_i = 0$ and each
$\kappa_j =
0$.
\end{proof}

\subsection{Building Positive Linear Functionals via Matrices}

Recall that $\mathbb{S}^{k}$ is the set of $k \times k$ symmetric matrices
over $\RR$. Define $\langle A, B \rangle := \operatorname{Tr}(AB)$ to be the
inner product
on $\mathbb{S}^{k}$.

\begin{lemma}
\label{lem:pSd}
 Let $\cB \subset \mathbb{S}^{k}$ be a vector subspace.
Then exactly one of the following holds:
\begin{enumerate}[\rm(1)]
 \item\label{item:1} There exists $B \in \cB$ such that $B \succ
0$, and there exists no nonzero $A \in \cB^{\bot}$ with $A \succeq 0$.
\item\label{item:2} There exists $A \in \cB^{\bot}$ such that $A \succ 0$, and
there exists no nonzero $B \in \cB$ with $B \succeq 0$.
\item There exist nonzero $B \in \cB$ and $A \in \cB^{\bot}$ with
$A, B \succeq
0$, but there exist no $B \in \cB$ nor $A \in \cB^{\bot}$ with either $A \succ
0$ or $B \succ 0$.
\end{enumerate}
\end{lemma}

\begin{proof}
This is a consequence of the Bohnenblust 
\cite{Bon48} dichotomy; see \cite[Lemma 5.7]{N} for a detailed proof.
\end{proof}

Using Lemma \ref{lem:pSd} we can establish
the existence of certain positive linear functionals
with desirable properties.

\begin{lemma}
\label{lem:goodSepFun}
Let $L \in \RR^{\nu \times \nu}\axs$ be a linear pencil,
$\rC \subset \RR^{1
\times \ell}\axs$
be a finite right chip space, $I \subset \RR\axs^{1 \times \ell}$ be a left
module
generated by polynomials in $\RR\axs_1\rC$, and let $p \in
\rC^*\RR\axs_1\rC$ be a symmetric polynomial.
Let $\Md = M_{T,K}(L)$ be a truncated test
module for $\Lradd{L}{\rC}{I}$, $L$, and $\rC$.
If \[p
\not \in \Md + \RR^{\ell\times 1}\Lradd{L}{\rC}{I} +
\big(\Lradd{L}{\rC}{I}\big)^{\ast}\RR^{1 \times \ell},\] then there
exists a positive linear functional $\lambda$ on
$\rC^*\RR\axs_1\rC$ with the following properties:
\begin{enumerate}[\rm(1)]
 \item $\lambda(a) > 0$ for each $a \in \Md \setminus \{0\}$;
 \item $\lambda(\iota) = 0$ for each $\iota \in \big(\RR^{\ell \times
1}\Lradd{L}{\rC}{I} +
[\Lradd{L}{\rC}{I}]^{\ast}\RR^{1 \times \ell}\big) \cap \rC^*\RR\axs_1\rC$;
 \item $\lambda(p) < 0$.
\end{enumerate}
\end{lemma}

\begin{proof}
Without loss of generality we may assume $I = \Lradd{L}{\rC}{I}$ since
by
Proposition
\ref{prop:genOfQDRad}, $\Lradd{L}{\rC}{I}$ is also generated by
polynomials in $\RR\axs_1\rC$.

First, $T \neq \{0\}$ since otherwise  $I =
\RR^{1 \times \ell}\axs$.
Further, the case where
$K = \{0\}$ is similar to the case where $K \neq \{0\}$, so
without loss of generality assume that $K
\neq \{0\}$.

 If
\[\Md \cap \left( \RR p+ \RR^{\ell \times 1} I + I^* \RR^{1 \times \ell}
\right) = \{0\},\]
then let $W = \RR p$. Otherwise, let $W = \{0\}$.  In
either case, by Lemma \ref{prop:lowdegqradcond}, we have
\[\Md \cap \left( W+ \RR^{\ell \times 1} I + I^* \RR^{1 \times \ell}
\right) = \{0\}.\]

Let $\tau_1, \ldots, \tau_{\mu}$ be a basis for $T$ and let
$\kappa_1, \ldots, \kappa_{\sigma}$ be a basis for $K$. Define
column vectors $\tau :=
(\tau_i)_{1
\leq i \leq \mu}$, $\kappa:= (\kappa_j)_{1 \leq j \leq \sigma}$, and $L\kappa
:= (L \kappa_j)_{1 \leq j \leq \sigma}$.
The set $\Md$ is characterized as being
the set of polynomials of the
form $\tau^* A \tau + \kappa^* B (L\kappa)$,
where $A$ and $B$ are positive-semidefinite matrices.
By hypothesis, if
$\tau^* A \tau + \kappa^* B (L\kappa) \in W + \RR^{\ell \times 1}I + I^*\RR^{1
\times \ell}$
and $A,B \succeq 0$, then  $A, B
= 0$.

Let $\cZ \subset \mathbb{S}^{\mu} \times \mathbb{S}^{\sigma}$ be defined by
\[
\cZ := 
\left\{
(Z_{\tau}, Z_{\kappa}) \mid
\tau^*
Z_{\tau}
\tau
+
\kappa^*
Z_{\kappa}
(L\kappa) \in W+ \RR^{\ell \times 1}I + I^*\RR^{1 \times \ell}
\right\}.\]
By assumption, the space $\cZ$ contains no pairs $(Z_{\tau}, Z_{\kappa})$ with
$Z_{\tau}, Z_{\kappa} \succeq
0$ except $(0,0)$.
Therefore there is no nonzero positive-semidefinite matrix in the space
$\widehat{\cZ} \subset \mathbb{S}^{\mu + \sigma}$ defined by
\[
\widehat{\cZ} :=
\left\{
Z_{\tau} \oplus Z_{\kappa}
\mid
(Z_{\tau} , Z_{\kappa}) \in \cZ
\right\}.\]
By Lemma \ref{lem:pSd} there exists a positive definite matrix $C \in
\widehat{\cZ}^{\bot}$.  Let $C_{\tau} \in \mathbb{S}^{\sigma}$ and $C_{\kappa}
\in
\mathbb{S}^{\tau}$ be such that $C$ is of the form
\[
C = \left(
\begin{array}{cc}
 C_{\tau}&\tilde{C}\\
\tilde{C}^{\ast}&C_{\kappa}
\end{array}
\right)
\]
for some block matrix $\tilde{C}$.
Since $C \succ 0$, we have $C_{\tau}, C_{\kappa} \succ 0$, and if
$(Z_{\tau},
Z_{\kappa}) \in \cZ$, then since $C \in \cZ^{\bot}$,
\[
\langle C_{\tau}, Z_{\tau} \rangle
+
\langle C_{\kappa}, Z_{\kappa} \rangle
=
\left\langle
C, Z_{\tau} \oplus Z_{\kappa}
\right\rangle = 0.
\]

Decompose $\rC^*\RR\axs_1\rC$ as
\[ \rC^*\RR\axs_1\rC = \left(M + W + [(\RR^{\ell
\times
1}I + I^*\RR^{1 \times \ell}) \cap \rC^*\RR\axs_1\rC] \right)
\oplus S,
\]
for some space $S \subset \rC^*\RR\axs_1\rC$.
Define $\tilde{\lambda}$ on $\rC^*\RR\axs_1\rC$ as
follows:
\beq\label{eq:tlambda}
 \tilde{\lambda}\left(
\tau^*
A
\tau
+
\kappa^*
B
(L\kappa) + \iota + s
\right)
=
\operatorname{Tr}
(AC_{\tau}) + \operatorname{Tr}
(BC_{\kappa}),
\eeq
where $A \in \RR^{\mu \times \mu}$, $B \in \RR^{\sigma \times \sigma}$,
$\iota \in W +
\RR^{\ell \times 1}I + I^*\RR^{1 \times \ell}$
and $s \in S$.  We now verify that $\tilde{\lambda}$ is
well defined.

First,
if $\tau^*A \tau + \kappa^* B (L \kappa) \in W +
\RR^{\ell \times 1}I + I^*\RR^{1 \times \ell}$ for some $A \in \RR^{\mu \times
\mu}$ and $B \in \RR^{\sigma \times
\sigma}$,  not necessarily
symmetric,
 then $\tau^*A^{\ast} \tau + \kappa^* B^{\ast} (L \kappa) \in W +
\RR^{\ell \times 1}I + I^*\RR^{1 \times \ell}$, which implies that
\[\tau^*(A + A^{\ast}) \tau + \kappa^* (B + B^{\ast}) (L \kappa) \in W +
\RR^{\ell \times 1}I + I^*\RR^{1 \times \ell}.\]
Therefore
\[\operatorname{Tr}
(AC_{\tau}) + \operatorname{Tr}
(BC_{\kappa}) = \frac{1}{2} \big( \operatorname{Tr}
[(A+ A^{\ast})C_{\tau}] + \operatorname{Tr}
[(B + B^{\ast})C_{\kappa}]
\big) = 0.\]

Next, suppose $\tau^*A_1 \tau + \kappa^* B_1 (L \kappa) + \iota_1 + s_1 = 
\tau^*A_2
\tau +
\kappa^* B_2 (L \kappa) + \iota_2 + s_2$, where $A_1,A_2 \in \RR^{\mu \times 
\mu}$, 
$B_1, B_2 \in \RR^{\sigma \times \sigma}$, $\iota_1, \iota_2 \in W +
\RR^{\ell \times 1}I + I^*\RR^{1 \times \ell}$, and $s_1, s_2 \in S$.
Then
\[\tau^*(A_1-A_2) \tau + \kappa^* (B_1-B_2) (L\kappa) + (\iota_1-\iota_2) +
(s_1-s_2) =
0.\]
By construction, we must have $s_1 - s_2 = 0$.  Therefore
\[\tau^*(A_1-A_2) \tau + \kappa^* (B_1-B_2) (L \kappa) \in W +
\RR^{\ell \times 1}I + I^*\RR^{1 \times \ell}. \]
Hence
\begin{align}
\notag
 \tilde{\lambda}(\tau^*A_1 \tau + \kappa^* B_1 (L \kappa) + \iota_1 + s_1)
&= \operatorname{Tr}
(A_1C_{\tau}) + \operatorname{Tr}
(B_1C_{\kappa})\\
\notag
&=\operatorname{Tr}
(A_2C_{\tau}) + \operatorname{Tr}
(B_2C_{\kappa})\\
\notag
& \quad +
\operatorname{Tr}
([A_1-A_2]C_{\tau}) + \operatorname{Tr}
([B_1-B_2]C_{\kappa})
\\
\notag
&=\tilde{\lambda}(\tau^*A_2 \tau + \kappa^* B_2 (L \kappa) + \iota_2 + s_2).
\end{align}
Therefore $\tilde{\lambda}$ is well defined.

Next, if $W = \{0\}$, let $\lambda = \tilde{\lambda}$.
If $W = \RR p$, we define $\lambda$ as follows.
Choose a symmetric functional
$\xi$ on
$\rC^*\RR\axs_1\rC$ such that $\xi(p) < 0$ and $\xi(\RR^{\ell
\times 1} I + I^* \RR^{1 \times \ell}) = \{0\}$, which exists by the
Hahn-Banach Theorem.
Let $\tilde{C}_{\tau} = (\xi[\tau_i^*\tau_j])_{1 \leq i,j \leq \mu}$ and
$\tilde{C}_{\kappa}=(\xi[\kappa_i^*L\kappa_j])_{1 \leq i,j \leq \sigma}$.
Choose $\epsilon > 0$ such that $C_{\tau} + \epsilon \tilde{C}_{\tau} \succ 0$
and
$C_{\kappa}
+ \epsilon \tilde{C}_{\kappa} \succ 0$. Define $\lambda =
\tilde{\lambda} + \epsilon \xi$.

It follows immediately, by definition of $\tilde{\lambda}$ that
$\tilde{\lambda}(\iota) = 0$ for each $\iota \in W +
\RR^{\ell \times 1}I + I^*\RR^{1 \times \ell}$.
If $W = \{0\}$, then $\lambda(\RR^{\ell \times 1}I + I^*\RR^{1 \times \ell}) =
0$.  If $W = \RR p$, since $\xi(\RR^{\ell \times 1}I + I^*\RR^{1 \times
\ell}) = \{0\}$, we also have $\lambda(\RR^{\ell \times 1}I + I^*\RR^{1 \times
\ell}) =
0$.
It is also clear that $\lambda(b^*) = \lambda(b)$ for each $b\in
\rC^*\RR\axs_1 \rC$ in both cases.

Each nonzero element of $M$ is of the form $\tau^*A\tau + \kappa^*B (L\kappa)$,
where $A,B \succeq 0$, and at least one of $A$ and $B$ is nonzero.
If $W = \{0\}$, then
\[\lambda(\tau^*A\tau + \kappa^*B(L\kappa)) =
\operatorname{Tr}
(AC_{\tau}) + \operatorname{Tr}
(BC_{\kappa}) > 0,\]
since $C_{\tau}, C_{\kappa} \succ 0$.
If $W = \RR p$, and if $A=(A_{ij})_{1 \leq i,j \leq \mu}$,
and $B = (B_{ij})_{1 \leq i,j \leq \sigma}$, then
\begin{align}
 \notag
\lambda(\tau^*A\tau + \kappa^*B(L\kappa)) &=
\operatorname{Tr}
(AC_{\tau}) + \operatorname{Tr}
(BC_{\kappa})\\
\notag
& \quad
+ \epsilon \sum_{i=1}^{\mu} \sum_{j=1}^{\mu}
a_{ij} \xi(\tau_i^*\tau_j)
+ \epsilon \sum_{i=1}^{\sigma} \sum_{j=1}^{\sigma}
b_{ij} \xi(\kappa_i^*L\kappa_j)\\
\notag
&=\operatorname{Tr}
(A[C_{\tau} + \epsilon \tilde{C}_{\tau}]) + \operatorname{Tr}
(B[C_{\kappa} + \epsilon \tilde{C}_{\kappa}]) > 0,
\end{align}
since
$C_{\tau} + \epsilon \tilde{C}_{\tau}, C_{\kappa} + \epsilon \tilde{C}_{\kappa}
\succ 0$.
Further, if $q \in \rC$, then $q = \iota + \tau$ for some $\iota \in I$ and
$\tau \in T$.
We see that
\[\lambda(q^*q) = \lambda(\tau^*\tau) + \lambda(\iota^*\tau) +
\lambda(\tau^*\iota) + \lambda(\iota^*\iota) = \lambda(\tau^*\tau) \geq 0,\]
since $\iota \in I$ and $\tau^*\tau \in M$.  Therefore $\lambda$ is positive.

Finally, consider $\lambda(p)$.
If $W = \RR p$, then $\tilde{\lambda}(p) = 0$.
Hence
$\lambda(p) = \epsilon \xi(p) < 0$.
If $W = \{0\}$,
then
\[\RR p \cap (\Md + \RR^{\ell \times 1} I + I^* \RR^{1 \times \ell}) \neq
\{0\}.\]
Therefore
$\alpha p + \iota = m$ for some $\alpha \in \RR$ and some $m \in \Md \setminus
\{0\}$.
We cannot have $\alpha > 0$ since this would imply that $p \in \Md +
\RR^{\ell \times 1} I + I^* \RR^{1 \times \ell}$.  Similarly,
$\alpha \neq 0$ since otherwise $\iota \in \Md \setminus \{0\}$.  Hence
$\alpha < 0$,
so that $\lambda(p) = \alpha\lambda(m) < 0$.
\end{proof}

\begin{lemma}
 \label{lem:main}
 Let $L \in \RR^{\nu \times \nu}\axs$ be a linear pencil, $\rC \subset
\RR^{1 \times \ell}\axs$
be a full finite right chip space,
$I
\subset \RR^{1 \times \ell}\axs$ be a
left module generated by polynomials in $\RR\axs_1 \rC$, and
$p \in \rC^*\RR\axs_1\rC$ be a symmetric polynomial.
Let $\Md = M_{T,K}(L)$ be a truncated test module for $I$,$L$ and $\rC$.
Set $n = \dim(\rC) - \dim\big(\Lradd{L}{\rC}{I} \cap \rC\big)$.
\begin{enumerate}[\rm(1)]
 \item If $p
\not\in \Md + \RR^{\ell \times 1}\Lradd{L}{\rC}{I} +
(\Lradd{L}{\rC}{I})^*\RR^{1
\times \ell}$ then there exists $(X,v) \in V(I)^{(n)}$ such that
$v^*p(X)v < 0$ but $L(X) \succeq 0$.
\item If $p
\not\in \Md + \RR^{\ell \times 1}\LraddS{L}{\rC}{I} +
(\LraddS{L}{\rC}{I})^*\RR^{1
\times \ell}$, then there exists $(X,v) \in V(I)^{(n)}$ 
such that
$v^*p(X)v < 0$ and $L(X) \succ 0$.
\end{enumerate}
\end{lemma}

\begin{proof}
Without loss of generality, let $I$ be $(L,\rC)$-real
since, by Proposition
\ref{prop:genOfQDRad},
 the $(L, \rC)$-real radical of $I$ is also generated by
polynomials in $\RR\axs_1 \rC$.
Also, $p \not\in \RR^{\ell \times 1}I$ 
implies that $I \neq \RR\axs \rC$.  In particular, this implies that $n = 
\dim(\rC) - \dim(I \cap \rC) > 0$.
Let $\lambda$ be a linear functional with the properties described by
Lemma
\ref{lem:goodSepFun}.
By Corollary \ref{cor:GNS},
we produce a tuple of $n \times n$ matrices $X$, together with a
vector $v \in \RR^{\ell n}$
such that \[v^{\ast}a(X)v = \lambda(a)\] for each $a \in
\rC^*\RR\axs_1\rC$, and such that
 $\RR^{\ell n} =
\{ q(X)v \mid q \in \rC\}.$

If $\iota \in I \cap \RR\axs_1 \rC$, then for each $q \in
\rC$,
we
have \[(q(X)v)^{\ast}\iota(X)v = \lambda(q^{\ast}\iota) = 0\] since
$\lambda\big([\RR^{\ell \times
1}I +
I^{\ast}\RR^{1 \times \ell}] \cap \rC^*\RR\axs_1 \rC\big) =
\{0\}$.  Therefore $\iota(X)v = 0$.
Since
$I$ is generated by $I \cap \RR\axs_1\rC$,
this
implies that $(X,v) \in V(I)^{(n)}$.

Next, consider $L(X)$.  Let $q \in \RR^{\nu \times 1}\rC$ be decomposed as
$q = \varphi + \kappa$, where $L\varphi \in \RR^{\nu \times 1}\axs_1 \rC \cap
I$, and
$\kappa \in K$.
We then see that
\[ (q(X)v)^{\ast}L(X)(q(X)v) = v^{\ast}\kappa(X)^{\ast}L(X)\kappa(X)v =
\lambda(\kappa^{\ast}L\kappa) \geq
0.\]
Hence $L(X) \succeq 0$.
If $I$ is strongly $(L,\rC)$-real, then
$\varphi \in \RR^{\nu \times 1}I$.
Therefore
$q(X)v \neq
0$
if and only if $\kappa \neq 0$. In this case,
\[(q(X)v)^{\ast}L(X)(q(X)v) =
\lambda(\kappa^{\ast}L\kappa) > 0\]
 since $\kappa^{\ast} L \kappa \in M
\setminus
\{0\}$.
Hence $L(X) \succ 0$.

Finally, $v^{\ast}p(X)v = \lambda(p) < 0$.
\end{proof}

As a consequence of Lemma \ref{lem:main},
we can describe $\Lrad{L}{I} \cap \rC$  for linear pencils $L$.

\begin{corollary}
\label{cor:lradd}
Suppose $L \in \RR^{\nu \times \nu}\axs$ is a linear pencil.
Let $\rC \subset
\RR^{1 \times \ell}\axs$
be a full finite right chip space,
and let $I \subset
\RR^{1 \times \ell}\axs$ be a left module generated by polynomials in
$\RR\axs_1 \rC$. Then,
\[ \Lrad{L}{I} \cap \rC = \Lradd{L}{\rC}{I} \cap
\rC
\quad \mbox{and} \quad
\LradS{L}{I} \cap \rC = \LraddS{L}{\rC}{I} \cap
\rC.\]
\end{corollary}

\begin{proof}
If a left module is $L$-real, then by definition it is also $(L, 
\rC)$-real, so $\Lradd{L}{\rC}{I} \subset \Lrad{L}{I}$.  Conversely, assume 
there exists $p \in \left( \Lrad{L}{I} \cap \rC\right)
\setminus
\left(\Lradd{L}{\rC}{I} \cap \rC \right)$.
Let $\Md$ be a truncated test module for $I$, $L$ and $\rC$.
We claim that \[-p^*p \not\in
\Md + \RR^{\ell \times 1}\Lradd{L}{\rC}{I} +
(\Lradd{L}{\rC}{I})^*\RR^{1
\times \ell}.\]
Indeed, as otherwise there would exist $m \in \Md$
such that
\[ p^*p + m \in \RR^{\ell \times 1}\Lradd{L}{\rC}{I} +
(\Lradd{L}{\rC}{I})^*\RR^{1
\times \ell}, \]
which would imply that $p \in \Lradd{L}{\rC}{I}$.
Now by Lemma \ref{lem:main} there exists $(X,v) \in V(I)$ such that
\[
-v^*p(X)^*p(X)v < 0\qquad\text{and}\qquad L(X) \succeq 0.
\]
  Since $\cI\big(\{(X,v)\}\big)$ is $L$-real
by Proposition \ref{prop:idealOfVarIsLRad}, we see that
that $p \not\in \cI((X,v)) \supset \Lrad{L}{I}$, which is a contradiction.

The $\LradS{L}{I}$ case is similar.
\end{proof}

\section{Main Results}
\label{sec:mainResults}
 
 In this section we use the results of \S\ref{sect:linFun} to prove several
Positivstellens\"atze and 
 Nullstellens\"atze. We first prove the main theorem of this paper, Theorem \ref{thm:mainNMon}.

\subsection{Proof of The Main Theorem \ref{thm:mainNMon}}

 If $p$ is of the form of \eqref{eq:clasPosIdeal}, then Proposition
\ref{prop:idealOfVarIsLRad} implies that $v^*p(X)v \geq 0$ if $(X,v) \in V(I)$
and
$L(X) \succeq 0$.

Conversely, suppose $p$ is not of the form of \eqref{eq:clasPosIdeal}.
By Proposition \ref{prop:genOfQDRad}, $\Lradd{L}{\rC}{I}$ is generated by
polynomials in $\RR\axs_1 \rC$.
By \cite[Lemma 4.2]{N}, 
the set of symmetric elements of the set
\[(\RR^{\ell \times
1} \Lradd{L}{\rC}{I} + \Lradd{L}{\rC}{I}\RR^{1 \times \ell}) \cap \rC^*
\RR\axs_1 \rC\] is 
all elements of the form
\[\sum_k^{\finite} (r_k^*\iota_k + \iota_k^*r_k),\]
with each $\iota_k \in
\Lradd{L}{\rC}{I} \cap \RR\axs_1 \rC$ and each $r_k \in \rC$.
Therefore,
\[p
\not\in \Md + \RR^{\ell \times 1}\Lradd{L}{\rC}{I} +
(\Lradd{L}{\rC}{I})^*\RR^{1
\times \ell},
\]
 where $\Md$ is some truncated test module for $I$, $L$ and
$\rC$.
Now Lemma \ref{lem:main} implies that there exists $(X,v)
\in V(I)$ such that $v^*p(X)v < 0$ and $L(X) \succeq 0$.

The strongly $L$-real case is similar, so its proof is omitted.
\qed

In the rest of this section we state and prove a few corollaries of Theorem \ref{thm:mainNMon}.
These contain several of those listed in the introduction.

\subsection{Degree Bounds}

Using the machinery of right chip spaces we deduce degree bounds on the terms
appearing in the Positivstellensatz certificate \eqref{eq:clasPosIdeal}.

\begin{corollary}
\label{cor:mainWDegBds}
 Let $L \in \RR^{\nu \times \nu}\axs$ be a linear pencil.
Let $I
\subset \RR^{1 \times \ell}\axs$ be a
left module generated by polynomials with degree bounded by
$d$, degree in each variable $x_k$ bounded by $d_k$, and degree in each
variable $x_k^*$ bounded by $d_{k+g}$.
Let $p \in \RR^{\ell \times \ell}\axs$ be a symmetric polynomial of
degree $\delta$, degree $\delta_k$ in each variable $x_k$, and degree
$\delta_{k+g}$ in
each variable $x_k^*$.
\begin{enumerate}[\rm(1)]
 \item\label{item:bigdegreebds} $v^{\ast}p(X)v \geq 0$ whenever $(X,v) \in V(I)$
and $L(X) \succeq
0$ if and only if $p$ is of the form
\begin{equation}
\label{eq:clasPosIdealdegbds}
  p = \sum_{i}^{\finite}p_i^{\ast}p_i + \sum_{j}^{\finite} q_j^{\ast}Lq_j +
 \sum_k^{\finite} (r_k^{\ast}\iota_k + \iota_k^{\ast}r_k)
\end{equation}
where each $p_i, r_k \in \RR^{1 \times \ell}\axs$, each $q_j \in
\RR^{\ell \times \ell} \axs$ and $\iota_k \in \Lrad{L}{I}$, 
with the following degree bounds:
\begin{enumerate}[\rm(a)]
 \item each $p_i$, $q_j$, and $r_k$ has degree bounded by $\max\big\{d-1,
\lceil \frac{\delta-1}{2}, 0 \rceil\big\}$.
\item each $\iota_k$ has degree bounded by $\max\big\{d, \lceil \frac{\delta + 1}{2}
\rceil\big\}$.
\item each $p_i$, $q_j$, $r_k$ and $\iota_k$ has degree in each variable $x_k$
bounded by $\max\{d_k, \delta_k \}$,
\item each $p_i$, $q_j$, $r_k$ and $\iota_k$ has degree in each variable $x_k$
bounded by $\max\{d_{k+g}, \delta_{k+g} \}$,
\end{enumerate}
\item $v^{\ast}p(X)v \geq 0$ whenever $(X,v) \in V(I)$ and $L(X) \succ 0$,
if and only if $p$ is of the form
\eqref{eq:clasPosIdealdegbds}
where each $p_i, r_k \in \RR^{1 \times \ell}\axs$, each $q_j \in
\RR^{\ell \times \ell} \axs$ and $\iota_k \in \LradS{L}{I}$, 
and the same degree bounds as in
\eqref{item:bigdegreebds} hold.
\end{enumerate}
\end{corollary}

\begin{proof}
Let $\rC$ be spanned by all monomials in $\RR^{1 \times \ell}\axs$ with degree
bounded by $\max\big\{d-1, \lceil \frac{\delta -1}{2}, 0 \rceil \big\}$, degree in each
$x_i$ bounded by $\max\{d_i, \delta_i\}$, and degree in each $x_i^*$ bounded by
$\max\{d_{i+g}, \delta_{i+g}\}$.  Then $\rC$ is a full finite right chip space,
$I$ is generated  by some polynomials in $\RR\axs_1 \rC$, and $p \in
\rC^*\RR\axs_1 \rC$.
Also note that $\Lrad{L}{I} \cap \rC = \Lradd{L}{\rC}{I} \cap \rC$ and
$\LradS{L}{I} \cap \rC = \LraddS{L}{\rC}{I} \cap \rC$ by Corollary
\ref{cor:lradd}.
The result now
follows directly from Theorem \ref{thm:mainNMon}.
\end{proof}

\begin{remark}
Given a finitely-generated left module $I \subset \RR^{1 \times
\ell}\axs$
and a symmetric polynomial $p \in \RR^{\ell \times \ell}\axs$, in general one
can construct a right chip space $\rC$ satisfying the conditions of Theorem
\ref{thm:mainNMon} with dimension much smaller than the space of polynomials
with degree bounds given in Corollary \ref{cor:mainWDegBds}.
\end{remark}

\subsection{Convex Positivstellensatz for General Linear Pencils}

Theorem \ref{thm:mainNMon} and  Corollary \ref{cor:radZero} below are extensions of
the Convex
Positivstellensatz from
\cite{HKMb}.

\begin{corollary}
\label{cor:radZero}
Let $\rC \subset \RR^{1 \times \ell}\axs$ be a full, finite right chip space.
 Let $L \in \RR^{\nu \times \nu}\axs$ be a linear pencil.
Let $p \in \rC^*\RR\axs_1\rC$ be a symmetric polynomial.
Then $p(X) \succeq 0$ whenever $L(X) \succeq
0$ if and only if $p$ is of the form
\begin{equation*}
  p = \sum_{i}^{\finite}p_i^{\ast}p_i + \sum_{j}^{\finite} q_j^{\ast}Lq_j +
 \sum_k^{\finite} (r_k^{\ast}\iota_k + \iota_k^{\ast}r_k)
\end{equation*}
where $p_i, r_k \in \rC$, $q_j \in
\RR^{\nu \times 1}\rC$ and $\iota_k \in \Lradd{L}{\rC}{\{0\}} \cap \RR\axs_1
\rC$.
\end{corollary}

\begin{proof}
 Apply Theorem \ref{thm:mainNMon} with $I = \{0\}$.
\end{proof}

Here is the restriction of Corollary \ref{cor:radZero} to the monic
case (cf.~\cite[Theorem 1.1 (2)]{HKMb}).

\begin{corollary}
\label{cor:convPSS}
 Let $L \in \RR^{\nu \times \nu}\axs$ be a monic linear pencil, let $\rC
\subset \RR^{1 \times \ell}\axs$ be a finite chip space, and suppose
$p
\in \rC^*\RR\axs_1 \rC$ is symmetric.   Then $p(X) \succeq 0$ whenever
$L(X) \succeq 0$ if and only if $p$ is of the form
\begin{equation}
\label{eq:truncModForm}
 p = \sum_{i}^{\finite} p_i^*p_i + \sum_{j}^{\finite}
q_j^*Lq_j,
\end{equation}
where each $p_i \in \rC$ and each $q_j \in \RR^{\nu
\times \ell}\rC$.
\end{corollary}

\begin{proof}
 Since $L$ is monic,
by Proposition \ref{prop:qRadRealHom} we have $\sqrt[(L)]{\{0\}} = \{0\}$.
Now apply Corollary \ref{cor:radZero}.
\end{proof}

Our results on right chip spaces
yield tighter bounds on the polynomials $p_i$ and $q_j$ in
\eqref{eq:truncModForm}
than previous results in \cite{HKMa,HKMb}.

\subsection{Size Bounds}\label{subsec:sizeBd}

In this section we present size bounds; that is, given a linear pencil $L$ and a polynomial
$p\in\rC^* \RR\axs_1 \rC$, the positivity of $p$ on $\cD_L$ only needs to be tested on $n\times n$ matrices $X\in\cD_L$
for $n=\dim(\rC)$. More precisely, we have

\begin{corollary}
\label{cor:lmidomgen}
Let $\rC \subset \RR^{1 \times \ell}\axs$ be a right chip space.
 Let $L \in \RR^{\nu \times \nu}\axs$ be a linear pencil and let $p \in
\rC^* \RR\axs_1 \rC$. Set
\[
 n = \dim(\rC) - \dim(\Lradd{L}{\rC}{\{0\}} \cap \rC)
\quad \mbox{and} \quad
 n_{+} = \dim(\rC) - \dim(\LraddS{L}{\rC}{\{0\}} \cap \rC)
\]
Then:
\begin{enumerate}[\rm(1)]
 \item\label{it:strict} $p \big|_{\cD_L} \succ 0$ if and only if $p
\big|_{\cD_L(n_+)} \succ 0$;
 \item\label{it:nonstrict} $p \big|_{\cD_L} \succeq 0$ if and only if
$p \big|_{\cD_L(n)}
\succeq 0$.
\end{enumerate}
\end{corollary}

\begin{proof}
 The proof of \eqref{it:strict} essentially the same as the proof of
\eqref{it:nonstrict}, so we will only give the proof of \eqref{it:nonstrict}.

First, the implication $(\Rightarrow)$ is clear. Let $I = \Lradd{L}{\rC}{\{0\}}$. If $p
\big|_{\cD_L}
\not\succeq 0$, then $p$ is not of the form \eqref{eq:clasPosIdeal}, 
so Lemma \ref{lem:main} implies that there
exists $(X,v) \in \left( \RR^{n \times n} \right)^g \times \RR^n$ with
$v^*p(X)v < 0$ and $L(X) \succeq 0$.
\end{proof}

\begin{remark}
 If $\deg(p) \leq 2k+1$, then we have $p
\in \rC^*\RR\axs_1 \rC$ for $\rC = \RR^{1 \times \ell}\axs_{k}$.
\end{remark}

\begin{corollary}
\label{cor:lmidom}
 Let $L \in \RR^{\nu \times \nu}\axs$ and $\hat{L} \in \RR^{\ell \times
\ell}\axs$ be linear pencils. Let
\[
 n = \ell - \dim(\Lradd{L}{\RR^{1 \times \ell}}{\{0\}} \cap \RR^{1 \times \ell})
\quad \mbox{and} \quad
 n_{+} = \ell - \dim(\LraddS{L}{\RR^{1 \times \ell}}{\{0\}} \cap \RR^{1 \times
\ell})
\]
Then:
\begin{enumerate}[\rm(1)]
 \item\label{it:strict2} $\hat{L} \big|_{\cD_L} \succ 0$ if and only if $\hat{L}
\big|_{\cD_L(n_{+})} \succ 0$;
 \item\label{it:nonstrict2} $\hat{L} \big|_{\cD_L} \succeq 0$ if and only if
$\hat{L} \big|_{\cD_L(n)}
\succeq 0$.
\end{enumerate}
\end{corollary}

\begin{proof}
Let $\rC = \RR^{1 \times \ell}$ and apply Corollary \ref{cor:lmidomgen}.
\end{proof}

Note that
 Corollaries \ref{cor:lmidomgen} and \ref{cor:lmidom} do not assume that
$\cD_L$ is bounded nor do they
assume that it has an interior point.

\subsection{The Left Nullstellensatz}
\label{subsub:lnss}

In this section we prove Corollary \ref{thm:mainFromNotes}, which is the main result of \cite{N}
and is
a generalization of the 
Real Nullstellensatz from \cite{chmn}.

We begin with the following corollary of
Theorem \ref{thm:mainNMon}:

\begin{corollary}
   \label{thm:lnss}
If $I \subset \RR^{1 \times \ell}\axs$ is a finitely-generated left module,
then $\rr{I} = \sqrt{I}$.
\end{corollary}

\begin{proof}
 Let $L = 1$.  By definition, 
 $\rr{I} = \Lrad{L}{I}$.
By Proposition \ref{prop:cIcCIsReal}, we have $\rr{I}
\subset \sqrt{I}$.
Suppose $r \in \sqrt{I}$.
It follows that $-v^*r(X)^*r(X) v = 0$ for each  $(X,v)
\in V(I)$.  Since $L(X) \succeq 0$ for each $X$, Theorem \ref{thm:mainNMon}
implies that $-r^*r$ is of the form
\[ -r^*r = \sum_i^{\finite} p_i^*p_i + \sum_j^{\finite} q_j^*q_j + \iota +
\iota^*,\]
where $\iota \in \RR^{\ell \times 1} \rr{I}$.
Therefore
\[ r^*r + \sum_i^{\finite} p_i^*p_i + \sum_j^{\finite} q_j^*q_j \in \RR^{\ell
\times 1} \rr{I} + \left(
\rr{I}\right)^* \RR^{1 \times \ell},\]
which implies that $r \in
\rr{I}$.
\end{proof}

We now prove Corollary \ref{thm:mainFromNotes}.

\begin{proof}[Proof of Corollary {\rm\ref{thm:mainFromNotes}}]
 Note that $p_i(X)v = 0$ means each row of $p_i(X)v$ is $ 0$, i.e. $e_k^*p_i(X)v =
0$ for each $e_k \in \RR^{1 \times \nu_i}$.  Therefore
\[ V(I) = V\left( \sum_{i=1}^k
\RR^{1 \times \nu_i} \axs p_i \right).\]
The first part of the result now follows from Corollary \ref{thm:lnss}.

Next, if $q$ is an element
of the left module \eqref{eq:checkIReal}, then
\[q = \sum_i^{\finite} \sum_{j=1}^k a_{ij} b_{ij} p_j\]
for some $a_{ij} \in \RR^{\nu \times 1}$ and $b_{ij} \in \RR^{1 \times
\ell}\axs$.
Therefore,
\[q = \sum_{j=1}^k\left( \sum_{i}^{\finite} a_{ij} b_{ij} \right) p_j. \qedhere \]
\end{proof}

\subsection{Positivity on a Left Module}

We can characterize polynomials $p$ which are positive on the variety of a left module as follows:

\begin{corollary}
\label{thm:prob1}
Let $\rC \subset \RR^{1 \times \ell}\axs$ be a full, finite right chip space.
 Let $I \subset \RR^{1 \times \ell}\axs$ be a left module generated by
polynomials
 in $\RR\axs_1 \rC$,
and
let $p \in \rC^*\RR\axs_1 \rC$ be a symmetric polynomial.
Then $v^{\ast}p(X)v \geq 0$
for each $(X,v) \in V(I)$
if and only $p$ is of
the form
\begin{equation*}
 p = \sum_{i}^{\finite} q_i^*q_i + \sum_j^{\finite} (r_j^{\ast}\iota_j +
\iota_j^{\ast}r_j),
\end{equation*}
where each $q_i, r_j \in \rC$ and each $\iota_j \in \rr{I} \cap \RR\axs_1 \rC$.
\end{corollary}

\begin{proof}
If $L = 1$, then $\rr{I} = \sqrt[(L)]{I}$ by definition.
We see $L(X) \succ 0$ for all tuples of matrices $X$, and we see
for any $q \in \RR^{1 \times \ell}\axs$ that $q^*Lq = q^*q$.
Therefore Theorem \ref{thm:mainNMon} gives the result.
\end{proof}

\subsection{Zero on the Intersection of the Variety of a Left Module and the
Positivity Set of a Linear
Pencil}\label{subsec:cor410}

We return to polynomials $p$ which vanish
on the intersection of the variety of a left module with a spectrahedron.
We next prove Proposition \ref{prop:zeroOnModPosIntro} and its
strongly $L$-real radical analog:

\begin{corollary}
 \label{cor:zeroOnModPos}
 Let $L \in \RR^{\nu \times \nu}\axs$
be a linear pencil. 
Let $I \subset \RR^{1 \times \ell}\axs$ be a finitely-generated left module,
and let $p \in \RR^{1 \times \ell}\axs$.
\begin{enumerate}[\rm(1)]
 \item $p(X)v = 0$ whenever $(X,v) \in
V(I)$ and $L(X)
\succeq 0$ if and only if $p \in \Lrad{L}{I}$;
\item $p(X)v = 0$ whenever
$(X,v) \in V(I)$ and $L(X) \succ 0$ if and only if $p \in \LradS{L}{I}$.
\end{enumerate}
\end{corollary}

\begin{proof}
Let $(X,v) \in V(I)$ be such that $L(X) \succeq 0$.
 Proposition \ref{prop:idealOfVarIsLRad} implies
that
 $\cI(\{(X,v)\})$ is an $L$-real left module containing $I$. Therefore,
$\Lrad{L}{I} \subset \cI(\{(X,v)\})$.

Conversely, suppose $p(X)v = 0$ whenever $(X,v) \in V(I)$ and $L(X)
\succeq 0$.  Then \[v^*\big(-p(X)^*p(X)\big)v \geq 0\] whenever $(X,v) \in V(I)$ and $L(X)
\succeq 0$.  Theorem \ref{thm:mainNMon} implies that
\[ -p^*p = \sum_j^{\finite} q_j^*q_j + \sum_k^{\finite} r_k^*Lr_k + \iota +
\iota^*\] for some $q_j \in \RR^{1 \times \ell}\axs$, $r_k \in \RR^{\nu \times
\ell}\axs$, and $\iota \in \RR^{\ell \times 1}\Lrad{L}{I}$.  Therefore
\[ p^*p + \sum_j^{\finite} q_j^*q_j + \sum_k^{\finite} r_k^*Lr_k \in \RR^{\ell
\times 1}\Lrad{L}{I} + \left[\Lrad{L}{I}\right]^*\RR^{1 \times \ell},\]
which implies that $p \in \Lrad{L}{I}$.

The strongly $L$-real case is similar.
\end{proof}

\section{Thick Spectrahedra and Thick Linear Pencils}
\label{sec:thick}

This section proves a ``Randstellensatz" for $\cD_L$
and properties of $L$-real radicals for monic linear pencils $L$
satisfying the zero determining property (ZDP).
 These are Theorem \ref{thm:detBdry2} and Proposition
\ref{prop:LradL}  stated in the introduction. 
Then in Subsection \ref{subsec:ZDP} we exhibit big
classes of linear pencils having  ZDP.

\subsection{Randstellensatz}

\begin{definition}
 If $L \in \RR^{\ell \times \ell}\axs$ is a linear pencil, let $I_L = \RR^{1
\times \ell} \axs L \subset \RR^{1 \times \ell}\axs$ be the left module
generated by the rows of $L$.
\end{definition}

\begin{prop}
\label{prop:ILrealS}
 Let $L \in \RR^{\ell \times \ell}\axs$ be a monic linear pencil.  Then
$\sqrt[\real]{I_L} = I_L$.
\end{prop}

\begin{proof}
  Let $L$ be
    \beq\label{eq:pencNS}
   L = \operatorname{Id}_{\ell} - \Lambda \quad \mbox{where} \quad \Lambda =
\sum_{i=1}^{\rm
finite} \left( A_i \otimes
x_i + A_i^* \otimes x_i^* \right)
  \eeq
  and each $A_i \in \mathbb{R}^{\ell \times \ell}$.  Consider $\sqrt[\real]{I_L}$. By
Proposition \ref{prop:genOfQDRad}, $\sqrt[\real]{I_L} = \sqrt[1]{I_L}$ is
generated by $I_L$
together with possibly some
constant polynomials.  Let $c \in \sqrt[\real]{I_L}$ be constant.  To show
that $\sqrt[\real]{I_L} = I_L$ it suffices to show that $c \in I_L$.

By Corollary \ref{thm:lnss}, if $(X,v) \in \left(\RR^{n \times n} \right)^g
\times \RR^{\ell n}$, then $L(X)v
= 0$ implies that $cv = 0$.
Using the embedding $\CC \rightarrow \RR^{2 \times 2}$ given by
\[
 a + b \imath \mapsto \begin{pmatrix}
              a&b\\-b&a
             \end{pmatrix}
\]
we can consider evaluating $L$ at tuples of complex numbers.
Fix a variable $x_i$ and let $x_i = a_i + \imath b_i$, where $a_i$ and $b_i$
are real variables. If $v$ is an eigenvector
of $A_i + A_i^*$ with nonzero eigenvalue $\lambda$, then $\lambda$ must be real
and
\[
 L\Big(0, \ldots, 0, \frac1{\lambda}, 0, \ldots, 0\Big)v = \operatorname{Id}_{\ell}v -
\frac{1}{\lambda} (A_i + A_i^*)v = 0.
\]
Hence $cv = 0$.  Since $A_i + A_i^*$ is symmetric, there exists an
orthonormal
basis for $\RR^{\ell}$ consisting of eigenvectors of $A_i + A_i^*$.  Therefore,
since
$c^*$ is orthogonal to all eigenvectors with nonzero eigenvalues, $c^*$ must be
an eigenvector of $A_i + A_i^*$ with eigenvalue $0$.

Similarly, consider $\imath A_i - \imath A_i^*$.  If $v$ is an eigenvector
of $\imath A_i - \imath A_i^*$ with nonzero eigenvalue $\lambda$, then
$\lambda$ must be real
and
\[
 L\Big(0, \ldots, 0, \frac{\imath}{\lambda}, 0, \ldots, 0\Big)v = \operatorname{Id}_{\ell}v -
 \frac{1}{\lambda} (\imath A_i - \imath A_i^*)v = 0.
\]
Since $\imath A_i - \imath A_i^*$ is Hermitian, there exists an
orthonormal
basis for $\CC^{\ell}$ consisting of eigenvectors of $\imath A_i - \imath
A_i^*$. Therefore, since
$c^*$ is orthogonal to all eigenvectors with nonzero eigenvalues, $c^*$ must be
an eigenvector of $\imath A_i - \imath A_i^*$ with eigenvalue $0$.
This implies that $\Lambda c^* = 0$.

 After a change of basis,
if $c =
e_1$, then $\Lambda e_1 = 0$, which implies
\[
 L = \begin{pmatrix}
      1&0& \cdots & 0\\
      0&L_{22}& \cdots & L_{2\ell}\\
      \vdots&\vdots&\ddots&\vdots\\
      0&L_{\ell 2}& \vdots& L_{\ell \ell}
     \end{pmatrix}
\]
Therefore $c \in I_L$.
\end{proof}

\begin{proof}[Proof of Proposition {\rm\ref{prop:LradL}}]
By Proposition \ref{prop:ILrealS}, $\rr{I_L}=I_L$. The equality
$\sqrt{I_L}=\rr{I_L}$ is \cite[Theorem 1.3]{N}.
\end{proof}

\begin{prop}
\label{prop:pss}
Let $L \in \RR^{\ell \times \ell}\axs$ be a monic linear pencil 
satisfying the hypotheses of Theorem {\rm\ref{thm:detBdry2}}.
 Suppose $p \in \RR^{\ell \times \ell}\axs$ is of the form
 \begin{equation}
 \label{eq:sosL}
  p = \sum_{i}^{\rm finite} q_i^*q_i + \sum_{j}^{\rm finite} r_j^*Lr_j \in
\RR^{\ell \times 1} I_L
+ I_L^*\RR^{1 \times \ell}.
 \end{equation}
 Then each $q_i \in I_L$ and 
 for each $r_j$,
 \[
  r_j \in \RR^{\ell \times 1}I_L + \big\{C \in \RR^{\ell \times \ell} \mid LC =
CL\big\}.
 \]
\end{prop}

\begin{proof}
From \eqref{eq:sosL} it follows that $q_i$ and $Lr_j$ vanish on 
$\dbcD_L^\circ$, so by \eqref{eq:conj+} they vanish on $V(L)$,
i.e., $q_i\in I_L$ and $Lr_j\in \sqrt{I_L}=I_L$ by Proposition \ref{prop:LradL}.
 So consider some $Lr \in I_L$.  Let ${\mathcal 
N}$ be the vector subspace
\[{\mathcal N} =
\left( \bigcap_{i=1}^{g} \operatorname{Null}(A_i) \right) \cap \left( 
\bigcap_{i=1}^{g} \operatorname{Null}(A_i^*) \right).\]  
We consider two cases.

{\em Case $1$}: Suppose ${\mathcal N} = \{0\}$. Let $W$ be the set of all
monomials which are not the leading monomial of an element of $I_L$. Decompose
$r$ as $\theta + \tilde{r}$, where $\theta \in \RR^{\ell \times 1}I_L$ and
\[
\tilde{r} = \sum_{\omega \in W} R_{\omega} \otimes \omega. 
\]
  We see that $L \tilde{r} = Lr - L\theta \in \RR^{\ell \times 1}I_L$.
If $\deg(\tilde{r}) > 0$, then the leading degree terms of $L\tilde{r}$ are
\[
 \sum_{i=1}^g \sum_{|\omega| = \deg(\tilde{r})} (A_i R_{\omega} \otimes x_i
\omega + A_i^* R_{\omega} \otimes x_i^*\omega),
\]
which must be nonzero since ${\mathcal N} = \{0\}$.  Since $I_L$ is
generated by polynomials of degree at most $1$,
there exists a left Gr\"obner basis for $I_L$ consisting of polynomials with
degree bounded by $1$. We see, however, that the leading degree
terms of $L \tilde{r}$ are not divisible on the right by
the leading terms of the polynomials in the left Gr\"obner basis for $I_L$, since
their rightmost degree $\deg(\tilde{r})$ piece is in $W$. This is a
contradiction. Hence
$\tilde{r}$ is constant.

Suppose $L\tilde{r} = q L$ for some matrix polynomial $q$.  If $q$ is of the form
\[
 q = \sum_{m \in \ax} Q_m \otimes m
\]
then the leading degree terms of $qL$ are
\[
 \sum_{i=1}^g \sum_{|m| = \deg(q)} Q_{m}(A_i  \otimes mx_i + A_i^* \otimes 
mx_i^*),
\]
which are nonzero since ${\mathcal N} =
\{0\}$.
Because
\[\deg(qL) = 
\deg(L\tilde{r}) \leq 1,\]
we see that $q$ is constant.  Therefore
\[
 L\tilde{r} = \tilde{r} - \sum_{i=1}^g (A_i \tilde{r} \otimes x_i + A_i^* 
\tilde{r} \otimes x_i^*) = qL = q -
\sum_{i=1}^g (q A_i \otimes x_i + q A_i^* \otimes x_i^*).
\]
Matching up terms shows $q = \tilde{r}$ and thus $L \tilde{r} = \tilde{r} L$.

{\em Case $2$}: Suppose ${\mathcal N} \neq \{0\}$.  After applying an
orthonormal change of basis to $L$ we may assume that $L$ is of the form
\[
 L = \begin{pmatrix}
      \tilde{L}& 0\\0&\operatorname{Id}_{\nu}
     \end{pmatrix}
     \quad
     \mbox{and}
     \quad
     \tilde{L} = \operatorname{Id}_{\ell - \nu} - \sum_{i=1}^{g} (\tilde{A}_i
\otimes x_i + \tilde{A}_i \otimes x_i^*)
\]
for some $\nu$,
where 
\[\left( \bigcap_{i=1}^g \operatorname{Null}(\tilde{A}_i) \right) \cap \left( 
\bigcap_{i=1}^g \operatorname{Null}(\tilde{A}_i^*) \right)= \{0\}.\]  Next, 
express
$r$ as
\[
 r = \begin{pmatrix}
      r_{11}&r_{12}\\r_{21}&r_{22}
     \end{pmatrix}
\]
where $r_{12}$ and $r_{22}$ have column dimension $\nu$, then
$r_{12}, r_{22} \in \RR^{\ell \times 1}I_L$.  Further, there exists a $q \in
\RR^{\ell \times \ell}\axs$ such that
\[
 \begin{pmatrix}
      \tilde{L}& 0\\0&\Id_{\nu}
     \end{pmatrix}\begin{pmatrix}
      r_{11}&0\\r_{21}&0
     \end{pmatrix}
     =
     \begin{pmatrix}
      \tilde{L}r_{11}&0\\r_{21}&0
     \end{pmatrix}
     =
     \begin{pmatrix}
      q_{11}&q_{12}\\q_{21}&q_{22}
     \end{pmatrix}
      \begin{pmatrix}
      \tilde{L}& 0\\0&\Id_{\nu}
     \end{pmatrix}
     =
     \begin{pmatrix}
      q_{11}\tilde{L}&q_{12}\\q_{21}\tilde{L}&q_{22}
     \end{pmatrix}
\]
which shows $r_{21}, \tilde{L}r_{11} \in \RR^{\ell \times 1}I_L$.
By Case 1, $r_{11}$ may be decomposed as $r_{11} = s\tilde{L} + C$,
where $s \in \RR^{\ell \times \ell}\ax$ and $C$ is a constant matrix satisfying
$\tilde{L}C = C\tilde{L}$.  Then
\[
 \begin{pmatrix}
  r_{11}&0\\0&0
 \end{pmatrix}
 =
 \begin{pmatrix}
  s&0\\0&0
 \end{pmatrix}L
 + \begin{pmatrix}
    C&0\\0&0
   \end{pmatrix}
   \quad \mbox{and}
   \quad
   \begin{pmatrix}
    C&0\\0&0
   \end{pmatrix}L
=
 L  \begin{pmatrix}
    C&0\\0&0
   \end{pmatrix}.
   \qedhere
\]
\end{proof}

We are now ready to  prove Theorem \ref{thm:detBdry2}.

\begin{proof}[Proof of Theorem {\rm\ref{thm:detBdry2}}]
 First, $L(X) \succeq 0$ implies that $p(X) \succeq 0$ is equivalent to $p$
being of the form \eqref{eq:sosL} by \cite[Theorem 1.1 (2)]{HKMb}.
Further, $L(X)v = 0$ implies that $p(X)v = 0$ is equivalent to $p \in
\RR^{\ell \times 1}\sqrt[\real]{I_L} = \RR^{\ell \times 1}I_L$ by Corollary
\ref{thm:lnss}. Therefore Proposition \ref{prop:pss} gives the result.
\end{proof}

\subsection{The zero determining property, ZDP}
\label{subsec:ZDP}

In this subsection we shall describe a rich class of pencils $L$ with the
ZDP.
We do not know of any examples of minimal pencils which fail to satisfy it.

\def\tp{{\tilde p}}
\def\del{{\delta}}

\def\mi{\deg}
\def\mddf{{minimum degree defining polynomial for}}
\def\mdd{{minimum degree defining polynomial}}

Let $p$ be a classical  commutative polynomial
with $p(0)>0$.
The closed set 
${\cC_p}$ is defined to be the closure of
the connected component of 0 of
$$\{x \in {\mathbb R}^g: p(x) > 0\}.$$
%
%
We call $\tp$
a \df{\mdd} for  $\cC_p$ if 
$\tp$ is the lowest degree polynomial for which $\cC_p=\cC_\tp$.
Recall from \cite{HV07} (see Lemma \ref{lem:mindeg} for details) 
there is only one \mddf \ $\cC_p$, i.e., $\tp$ is unique up to multiplication by a positive scalar.
 Denote by $\mi(\cC_p)$ the degree of such a minimal $\tp$.
Observe that this definition also applies to spectrahedra $\cD_L(n) \subseteq \matng$
associated to a monic linear pencil $L$.

Given a linear pencil $L$, let  
\beq
\delta_n(X):= \det \big(L(X)\big) \quad\text{ for }  X \in \matng.
\eeq
For example, consider the \df{free ball}
\[
L(x)= \begin{pmatrix}
1 & x_1^* & x_2^* & \cdots & x_g^* \\
x_1 & 1 & 0 & \cdots & 0\\
x_2 & 0 & 1 & \ddots & \vdots \\
\vdots & \vdots & \ddots &  \ddots & 0 \\
x_g & 0 & \cdots & 0 & 1
\end{pmatrix}.
\]
Then 
$\cD_L(n)=\{X\in\matng \mid \|X\|\leq1\}$, 
and 
\[
\delta_n(X) = \det \Big(I - \sum_{j=1}^g X_j^*X_j\Big)
\]
by way of Schur complements.
Note  $\del_n$ is a degree $2n$ polynomial in the entries of the $X_j$s.

\begin{theorem}
\label{thm:minRules}
Suppose $L$ is  a monic linear pencil.
\ben[\rm(1)]
\item
$\cZ(\dbcD_L^\circ) \subset V(L)$.
\item
Suppose $\delta_n$ is a minimal degree defining polynomial 
for $\cD_L(n)$ for every $n$,
 then   $\cZ(\dbcD_L^\circ) = V(L)$.
 \item
 The conclusion of  $(2)$ holds  even if $\delta_n =\mu_n^m$ with $\mu_n$
the minimal degree defining polynomial for $ \cD_L(n)$.
\een
\end{theorem}

\begin{cor}
\label{cor:mindegyes}
Suppose $L$ is an $\ell \times \ell$
 monic pencil.
\ben[\rm(1)]
\item $\deg \del_n \geq\mi(\cD_L(n)) \geq n \deg  \del_1$.
\item 
Suppose $\del_1$ is a \mddf{} $\cD_L(1)$.
If $ \deg  \del_1  =\ell$, then $\mi(\cD_L(n)) = n \deg  \del_1$ and
 $\delta_n$ is a minimal degree defining polynomial 
for $\cD_L(n)$ for every $n$, so $\cZ(\dbcD_L^\circ) = V(L)$.
\item
 If $ \cD_L$ is the free ball, then $ \mi(\cD_L(n)) = n \deg  \del_1$ and
 $\delta_n$ is a minimal degree defining polynomial 
for $\cD_L(n)$ for every $n$, so $\cZ(\dbcD_L^\circ) = V(L)$.
\een
\end{cor}

\begin{cor}\label{cor:genZDP}
A generic $\ell\times\ell$ monic linear pencil $L$ in $g>2$ variables has ZDP.
\end{cor}

\begin{proof}
It is clear that $\delta_1=\det L(x)$ is of degree $\leq\ell$.
Furthermore, the determinant of a generic symmetric
matrix is irreducible (see e.g..~\cite[\S 61, p.~177]{Boc07}).
Then the zero set $V$ of $\delta_1$ is a generic linear section of this irreducible
variety and $\delta_1$ is thus irreducible by Bertini's theorem \cite[Theorem II.3.\S1.6, p.~249]{Saf99}.
In particular, the Zariski closure of $\cD_L(1)$ is $V$.
Now if $r$ is a \mddf{} $\cD_L(1)$, then $r$ vanishes on $V$ and is thus a multiple of $\delta_1$ 
by irreducibility.
Hence $\delta_1$ is a degree $\ell$ \mddf{} $\cD_L(1)$ and the desired conclusion follows
from Corollary \ref{cor:mindegyes}.
\end{proof}

\begin{remark}
\label{rem:mindegyes}
\mbox{}\par
\ben[\rm(1)]
\item
For a given $L$ the minimality of $\delta_1$ can be easily checked with computer algebra. 
It suffices to establish that the ideal in $\R[x]$ generated by $\det L$ is real radical.
We refer the reader to \cite{BN93,Neu98} for algorithmic aspects of real radicals in commutative polynomial rings.
\item From 
Corollary \ref{cor:genZDP}
we infer that there are numerous examples of ZDP pencils.
\item
Also for perspective,  any bivariate RZ polynomial $p(x_1 ,x_2)$ of degree $\ell$,
has a determinantal representation $p(x_1 ,x_2)=\det L(x_1 ,x_2)$ for 
some $\ell\times\ell$ monic linear pencil $L$, and $\cC_p=\cD_L(1)$ \cite{HV07}.
(However, in more than 2 variables $p$ may only admit $\ell\times\ell$ 
determinantal representations for $\ell> \deg p$, cf.~\cite{Vin12}.)
\een
\end{remark}

\begin{exa}
Let 
\[L(x_1,x_2)=\begin{pmatrix} 1+x_1 & x_2 \\ x_2 & 1-x_1
                  \end{pmatrix}.\]
Then $\cD_L(1)=\{(x_1 ,x_2)\in\R^2\mid 1-x_1 ^2-x_2^2\geq0\}$,
so $\det L(x_1 ,x_2) = 1-x_1 ^2 - x_2^2$ is the 
minimum degree defining polynomial for
$\cC_p=\cD_L(1)$. Hence $L$ has ZDP.
\end{exa}

To prove the above theorem  we need some lemmas and
we set about to prepare them.

\subsubsection{Minimal degree defining polynomials}

\def\Zar{{\rm Zar}}
\def\Var{{\rm Var}}

Here we give some basic facts about minimal degree defining polynomials.
We require background from the proof of Lemma 2.1 of \cite{HV07},
so the proof is reproduced  in an online 
 Appendix \ref{app:HV}. 

Given a commutative polynomial $p$ let $\Var(p)$ denote its zero set.
If $S\subseteq\R^m$, then $\Zar(S)\subseteq\R^m$ denotes the Zariski closure of $S$,
i.e., the set of common zeros of all polynomials vanishing on $S$.

\def\bcD{{\partial \cD}}

Beware the polynomials in the lemma are commutative. 
\begin{lemma}
\label{lem:mindeg}
A  \mdd{} $p$ for $\cC=\cC_p$ is unique up to a constant. Moreover,
\ben[\rm(1)]
\item
 any other polynomial $q$ with $ \cC_q = \cC_p$
is given by $q=ph$ where $h$ is an arbitrary polynomial
which is positive on a dense connected subset of $\cC_p$. 
\item  
 any other polynomial $q$ which vanishes on $ \partial \cC_p$
is given by $q=ph$ where $h$ is an arbitrary polynomial. 
\item
$\Zar(\partial \cC_p) = \Var(p)$.
\een
\end{lemma}

\begin{proof}
(1)\  This is Lemma \ref{lem:HV07}.

(2) \ Let  $V$ be the Zariski closure of $\partial {\mathcal C}\subseteq\R^m$,
and let $V = V_1 \cup \cdots \cup V_k$ be the decomposition of $V$ into irreducible components satisfying  $\dim V_i = m-1$ for each $i$
established in the proof of   Lemma \ref{lem:HV07}.
Write $p=p_1\cdots p_k$, where $p_i$ 
is an irreducible polynomial vanishing on $V_i$.

Since $q$ vanishes on $\partial\cC$, it vanishes on each $V$ and thus on each $V_i$.
By the real Nullstellensatz for principal ideals \cite[Theorem 4.5.1]{BCR98},
$q=p_1 r_1$ for some $r_1$. Since $p_1$ does not identically vanish on $V_2$,
$r_1$ vanishes on $V_2$. Thus there is $r_2$ with $r_1=p_2 r_2$. Repeating this
$k$ times leads to $q= p_1 p_2 \cdots p_k h$ for some polynomial $h$.

(3) \ This is basically a restatement of (2).
\end{proof}

\def\foot{{\rm foot}}
Let $\foot(S)$ denote the \df{footprint} 
$$\foot(S) := \{X \mid  (X,v) \in S \}$$
of the  set $S$ in $\matng \times \RR^n$.

\begin{lem}
\label{lem:null1dim}
Let $L$ be a monic linear pencil, and suppose that
$\delta_n$ is a minimal degree defining polynomial 
for $\cD_L(n)$ for some $n$.
\ben[\rm (1)]
\item Then $\partial^1 \cD_L(1)$ defined by
\[\partial^1 \cD_L(n) = \{X\in\matng \mid    \dim\ker \big(L(X)\big)=1 \;\& \;
L(X) \succeq 0 \}\]
is nonempty and dense in $\partial\cD_L(n)$.
\item $ V^1(L)$ defined by 
\[ V^1(L)(n) = \{X\in\matng \mid \dim\ker \big(L(X)\big)=1\}\]
is nonempty and dense in $\foot[V(L)(n)]$.
\een
\end{lem}

\begin{proof}

(1) \ Consider Renegar's directional derivative $\delta_n'$  of $\delta_n$. Like $\delta_n$, it is 
an RZ polynomial, and the corresponding algebraic interior $\cC'$ contains $\cC$, 
cf.~\cite[\S 4]{Ren06} or \cite[\S2]{Vin12}. By minimality of $\delta_n$, $\cC'\supsetneq\cC$, i.e.,
there is $X\in\partial\cD_L(n)$ with $\delta_n'(X)\neq0$. That is, $X$ is a simple 
root of $\delta_n$ and thus $\dim\ker L(X)=1$.

Having established that $\partial^1\cD_L(n)\neq\varnothing$, the density
 follows from \cite[Theorem 6]{Ren06}.
 Note that in this case $\partial^1\cD_L(n)$ are exactly smooth points of $\partial\cD_L(n)$ \cite[Lemma 7]{Ren06}.
 
 (2)  We use the decomposition 
 $$ \foot[V(L)(n)]= V_1 \cup \cdots \cup V_k$$
 into irreducible components described in the proof of Lemma \ref{lem:mindeg}.
 Each $V_j$ has relatively open intersection  with $ \partial^1\cD_L(n)$,
 so there is an $X^j$ in $V_j$ for which $\dim \ker L(X^j)=1$.
 Thus $\dim \ker L(X)=1$ for $X$ in an open dense subset of $V_j$.  
\end{proof}

\begin{proof}[Proof of Theorem {\rm\ref{thm:minRules}}]
 (1) is obvious.

(2) Since $\del_n$ is the \mddf \  $\bcD_L(n)=\foot[\widehat \bcD_L^\circ(n)]$, 
by Lemma \ref{lem:mindeg}(3)
 Zariski closures satisfy $\Zar( \bcD_L(n))= \Var(\del_n )$
 for each $n$. So
  \beq
  \label{eq:ZisV}
  \foot[\cZ( \widehat \bcD_L^\circ)(n)] = \Zar(\bcD_L(n))=  \Var (\del_n )= \foot[V(L)(n)] 
  \eeq
  Intuitively, this says the ``footprint" of what we are trying to prove is as claimed.
  By Lemma \ref{lem:null1dim}, we get $1 = \dim \ker L(X) $
  for an open dense set $\cU$ of  $X$ in $\foot[V(L)(n)]$,
  so also for such  $X \in \foot[\cZ(\widehat  \bcD_L^\circ)(n)]$.
 This with \eqref{eq:ZisV} and  (1) says  
$\ker  L(X)(n) = 
\{u \in   \RR^n \mid \ (X,u)\in \cZ(\dbcD_L^\circ(n)) \} $,
which proves (2).
 
(3) follows from the same arguments as (2).
\end{proof}

\begin{proof}[Proof of Corollary {\rm\ref{cor:mindegyes}}]

(1) \
To get a lower bound on the minimal degree required to define 
$\cD_L(n)$
take $X=X^1\oplus\cdots\oplus X^n$ diagonal matrices in $\matng$, with 
$X^j\in\R^g$, and evaluate
\beq
\label{eq:dnd1}
\delta_n(X)=\det L(X) =  \prod_{j=1}^n\det L(X^j)= \prod_{j=1}^n \del_1( X^j).
\eeq
Note it has (on commutative $X$) degree $n \deg \del_1$.
No lower degree polynomial will vanish on all of 
the diagonal $X \in \bcD_L(n)$,  since $\del_1$
is minimal, so $\mi(\cD_L(n)) \geq n \deg \del_1$.
As $\del_n$ 
is a defining polynomial for  $\cD_L(n)$,
it follows that
 $\deg \del_n \geq \mi(\cD_L(n)).$ This proves (1).

An implication of (1) is that when $\del_1$ is minimal,
 $\del_n$ must be  minimal   for those $L$
 with 
 $\deg \del_n \leq n \deg \del_1.$
 This forces 
 $$\mi(\cD_L(n))= \deg  \del_n=  n \deg  \del_1.$$  From 
 this we get (2) and (3) immediately:
 \\
 (2) \
 By the definition of the determinant we have $\deg \del_n \leq  n d$,
 where $d$ is the size of $L$. 
 By hypothesis $d=\deg \del_1$, so $\deg \del_n \leq  n \deg  \del_1$,
 proving (2).
\\
 (3) Note for the ball that $n \deg \del_1 =2n$
 and we already saw  that $\del_n$ has degree $2n$.
\end{proof}

\subsubsection{Necessary side}

We present a necessary condition for a pencil to satisfy ZDP.

\def\SOS{{\rm SOS}}

\begin{lemma}
Suppose $L$ is a monic linear pencil, and let $\mu$ denote the \mddf \   $\cD_L(1)$.
Necessary for the zero determining property of $L$ is
that 
$$\delta_1= \mu^{m} \rho$$
 for some $m$
and polynomial $\rho$ which has zeros only on $\Var(\mu)$
but does not vanish everywhere on $\Var(\mu)$.  
Moreover, by the real Nullstellensatz,
the existence of such a $\rho$ is equivalent to
$$f \delta_1= 
 \mu^{m}(\mu^{2s} +\SOS)$$
for some polynomial $f$ and some $s$.
\end{lemma}

\begin{proof} From  
Lemma \ref{lem:mindeg}(2) we get
 $\delta_1= \mu^{m} \rho$
 for some $m$
and $\rho$ which is not zero 
everywhere on $\Var(\mu)$.  
If $\rho $ is 0 at some $X \not \in \Var(\mu)$,
then  $X \not \in \Zar(\cD_L)$ contradicting ZDP
even at the ``footprint level".

We just proved  $\rho=0$ implies $\mu=0$.
The real Nullstellensatz says equivalent to this  is
\beq
\label{eq:nssmu}
\mu^{2s} +\SOS = f \rho
\eeq
for some polynomial $f$. So
$f \delta_1= \mu^{m} f \rho=
 \mu^{m}(\mu^{2s} +\SOS)$.
Conversely, if such $f$ exists, then  $\delta_1= \mu^{m} \rho$ implies
that $\rho $ satisfies \eqref{eq:nssmu}.
\end{proof}

\section{Decomposition of Thin Linear Pencils}
\label{sect:L0}

The main concern of this section is a linear pencil $L$ for which the spectrahedron $\cD_L(1)$ has no 
interior, i.e., a thin linear pencil.
The special case of Theorem \ref{thm:mainNMon} where $I = \{0\}$ and $p = 0$
gives a characterization of the space $\Lrad{L}{\{0\}}$, which in turn
gives a nice algebro-geometric interpretation of spectrahedra $\cD_L$ having no interior points.

\subsection{Characterization of \texorpdfstring{$\Lrad{L}{\{0\}}$}{[sqrt L 0]}}

Recall that a spectrahedron with empty interior lies on an affine hyperplane
\cite{Ba}.
We now give a matricial version of this result.

\begin{prop}
\label{prop:charOfL}
  Let $L \in \RR^{\nu \times \nu}\axs$ be a linear pencil.
The space $\Lrad{L}{\{0\}} \subset \RR^{1 \times \ell}\axs$ is characterized by
\[\Lrad{L}{\{0\}} = \{ p \in \RR^{1 \times \ell}\axs \mid p(X) = 0\
\mbox{whenever}\ L(X)
\succeq 0\}.\]
\end{prop}

\begin{proof}
Every pair $(X,v)$ evaluated at $0$ is $0$.
The result follows from Corollary \ref{cor:zeroOnModPos}.
\end{proof}

Interestingly, there is a strong relation between a free spectrahedron $\cD_L$ 
and its scalar counterpart $\cD_L(1)$.

\begin{prop}
 \label{cor:infeas}
Let $L$ be a linear pencil, and let $\{0\}\subset\RR\axs$ be the trivial ideal.
Then the following are equivalent:
\begin{enumerate}[\rm(i)]
 \item\label{it:L0is0} $\Lrad{L}{\{0\}} = \RR\axs$.
 \item\label{it:infeasScal} The linear matrix inequality
$L(x) \succeq 0$ is infeasible over $x \in \RR^g$, i.e., $\cD_L(1)=\varnothing$.
\item\label{it:infeasGen} The linear matrix inequality
$L(X) \succeq 0$ is infeasible over matrix tuples $X \in (\RR^{n \times n})^g$
for each $n \in \NN$, i.e., $\cD_L=\varnothing$.
\end{enumerate}
\end{prop}

The equivalence between (ii) and (iii) of this proposition is easy---for example, see \cite[Corollary 4.1.4]{KS11}---but
we present a new proof here using the real radical theory of this paper.

\begin{proof}
That \eqref{it:infeasGen} implies \eqref{it:infeasScal} is clear.

Next, assume \eqref{it:infeasScal} holds.
Lemma \ref{lem:main} implies that if $-1 \not\in M + \Lradd{L}{\RR}{\{0\}}
+ (\Lradd{L}{\RR}{\{0\}})^*$,
where $M$ is a truncated test module for $\{0\}$, $L$ and $\RR$,
then there exists a tuple $X$ of $n \times n$ matrices
such that $L(X) \succeq 0$, where $n = \dim(\RR) - \dim(\{0\}) = 1$.
Since $n=1$, such an $X$ is actually in $\RR^g$, which is a contradiction.
Therefore, $-1 \in M + \Lradd{L}{\RR}{\{0\}}
+ (\Lradd{L}{\RR}{\{0\}})^*$.  Hence $1 + M \in \Lradd{L}{\RR}{\{0\}}
+ (\Lradd{L}{\RR}{\{0\}})^*$, which implies that $1 \in \Lradd{L}{\RR}{\{0\}}$,
which implies \eqref{it:L0is0}.

  Finally, suppose \eqref{it:L0is0} holds. 
  The condition $1(X) = 0$ is infeasible for all matrix tuples $X$.
By Proposition \ref{prop:charOfL},
and since $1 \in \Lrad{L}{\{0\}}$, it must be that $L(X) \succeq 0$ is
infeasible over all matrix tuples, which gives \eqref{it:infeasGen}.
\end{proof}

In $\RR^{1 \times \ell}\axs$, the $L$-real radical of $\{(0, \ldots, 0)\}$ can be derived easily from $\Lrad{L}{\{0\}} \subset \RR\axs$, as the following corollary shows.

\begin{corollary}
\label{cor:dimOfLrad0}
 Let $L \in \RR^{\nu \times \nu}\axs$ be a linear pencil.
Let $\{0\}$ denote the ideal of $ \RR\axs$ and let $\{(0, \ldots,
0)\}$ denote the left $\RR\axs$-module generated by $(0, \ldots, 0) \in \RR^{1 \times
\ell}\axs$.
Then $\Lrad{L}{\{(0, \ldots, 0)\}} = \RR^{1 \times \ell}
\otimes
\Lrad{L}{\{0\}}$.
\end{corollary}

\begin{proof}
By Proposition \ref{prop:charOfL},
we have
\[ \Lrad{L}{\{(0, \ldots, 0)\}} = \{p \in \RR^{1 \times \ell} \mid p(X) = 0\
\mbox{whenever}\ L(X) \succeq 0\}.\]
If $p = \sum_{i=1}^{\ell} e_i \otimes p_i \in \RR^{1 \times \ell}\axs$,
then $p(X) = 0$ means that each $p_i(X) = 0$.  Therefore $p \in \Lrad{L}{\{(0,
\ldots, 0)\}}$ if and only if each $p_i \in \Lrad{L}{\{0\}}$.
\end{proof}

\begin{definition}
A set $S \subset \RR^{\nu \times\nu}\axs$ is said to be \textbf{$\ast$-closed}
\index{star-closed set@$\ast$-closed set}
if $S^* = S$. A {\bf $\ast$-ideal}\index{$\ast$-ideal} is a two-sided ideal $I \subset \RR\axs$
which is $\ast$-closed, that is, $I = I^*$.
If $U \subset \RR\axs$, then the $\ast$-ideal generated by $U$
is
the two-sided ideal in $\RR\axs$ generated by $U + U^*$.
\end{definition}

\begin{corollary}
\label{cor:LradIsAClTwo}
 Let $L \in \RR^{\nu \times \nu}\axs$ be a linear pencil.
Then $\Lrad{L}{\{0\}} \subset \RR\axs$ is a  $\ast$-closed
real
ideal.
\end{corollary}

\begin{proof}
 Let $p \in \Lrad{L}{\{0\}}$.  If $q \in \RR\axs$, then
\[p(X)q(X) = 0q(X) = q(X)0 =
q(X)p(X) = 0.\]
Therefore $pq,qp \in
\Lrad{L}{\{0\}}$.
Next, by definition, $\Lrad{L}{\{0\}}$ is real.  Further, if $p \in \Lrad{L}{\{0\}}$,
then $pp^* \in
\Lrad{L}{\{0\}}$, which implies, since $\Lrad{L}{\{0\}}$ is real, that $p^* \in
\Lrad{L}{\{0\}}$.
\end{proof}

\begin{prop}
\label{prop:L0genByLinear}
 If $L \in \RR^{\nu \times \nu}\axs$ is a linear pencil, then
$\Lrad{L}{\{0\}}$ is the $\ast$-ideal in $\RR\axs$ generated
by  $(\Lradd{L}{\RR}{\{0\}})_1$.
Further, if $L(0) \succeq 0$, then $(\Lradd{L}{\RR}{\{0\}})_1$ is
spanned by a set of homogeneous linear forms.
\end{prop}

\begin{proof}
First, if $\Lrad{L}{\{0\}} = \RR\axs$, then Corollary \ref{cor:lradd} implies
that $1 \in \Lradd{L}{\RR}{\{0\}}$ so that $(\Lradd{L}{\RR}{\{0\}})_1 = \RR\axs_1$, 
which gives the result.  Therefore, by Corollary \ref{cor:infeas},
the only case that remains is the case where $L(x) \succeq 0$ is feasible
over $\RR^g$.
Without loss of generality, we can apply an affine change of variables to $x$ so
that $L(0) \succeq 0$.  Further, if $\iota \in \Lrad{L}{\{0\}}$, then Proposition 
\ref{prop:charOfL} implies that $\iota(0) =0$.  In particular, this implies that
$(\Lradd{L}{\RR}{\{0\}})_1$ is spanned by linear forms.

Next,  let $I \subset \Lrad{L}{\{0\}}$ be the 
$\ast$-ideal generated by $(\Lradd{L}{\RR}{\{0\}})_1$.
Suppose
\begin{equation}
\label{eq:1testI}
  \sum_i^{\finite} p_i^*p_i + \sum_j^{\finite} q_j^*Lq_j \in I.
\end{equation}
Let $\prec$ be a total
order on the letters $x_1, \ldots, x_g, x_1^*, \ldots, x_g^*$.
Since $I$ is generated by linear forms,
it is straightforward to 
reduce each $p_i$ and $q_j$ in \eqref{eq:1testI} to have monomials with no
letters which are the leading letter of an element of $I$.
Further, we can express $L$ as $L = \tilde{L} + L_I$, where $L_I \in
\RR^{\ell
\times \ell} \otimes I$ and $\tilde{L}$ has terms which contain no letters which are the leading
letter of an element of $I_1$.  Under
this condition
\begin{equation}
\label{eq:2test0}
 \sum_i^{\finite} p_i^*p_i + \sum_j^{\finite} q_j^*\tilde{L}q_j = 0.
\end{equation}
If at least some of the $p_i$ or $Lq_j$ in \eqref{eq:2test0} are nonzero, let
$\delta$ be the smallest degree such that at least some of the degree $\delta$ terms  $\tilde{p}_i$ of $p_i$
or $\tilde{q}_j$ of $q_j$ satisfy $\tilde{p} \neq 0$ or $\tilde{L}\tilde{q}_j \neq 0$.  For each $m \in \axs_{\delta}$, let $A_{m,p_i} \in \RR$ be
the coefficient of $p_i$ in $m$, and let $B_{m,q_j} \in \RR^{\nu \times 1}$
be the coefficient of $q_j$ in $m$.  Then the coefficient of $m^*m$ in
\eqref{eq:2test0} is
\begin{equation}
\label{eq:maximalI}
\sum_i^{\finite} A_{m,p_i}^*A_{m,p_i} + \sum_j^{\finite}
B_{m,q_j}^*L(0)B_{m,q_j} = 0.
\end{equation}
Since $L(0) \succeq
0$, each $A_{m,p_i} = 0$ and each $L(0)B_{m,q_j} = 0$.
Further, the terms of \eqref{eq:2test0} in $m^*\RR\axs_1^{\hom}m$ are
\[ m^*\left(\sum_j^{\finite}
B_{m,q_j}^*(\tilde{L}-L(0))B_{m,q_j}\right)m = 0.\]
Hence
\begin{equation*}
\sum_j^{\finite}B_{m,q_j}^*L B_{m,q_j} =
\sum_j^{\finite}B_{m,q_j}^*L_I B_{m,q_j} \in \RR^{\ell \times
1}\Lradd{L}{\RR}{\{0\}} +
\left(\Lradd{L}{\RR}{\{0\}}\right)^*\RR^{1 \times \ell}.
\end{equation*}
This implies that $L B_{m,q_j} \in \RR^{\nu
\times 1} (\Lradd{L}{\RR}{\{0\}})_1$ for each $j$.
Therefore, if \eqref{eq:2test0} holds, then each $LB_{m,q_j} \in I$, which
implies that $\tilde{L}B_{m,q_j} = 0$.  
Since $m$ was arbitrary, this yields $\tilde{p}_i = 0$ and
$\tilde{L}\tilde{q}_j = 0$, which is a contradiction.
 Hence each $p_i = 0$ and each $\tilde{L}q_j = 0$.
Therefore, \eqref{eq:1testI} holds if and only if
each
$p_i \in I$ and each $Lq_j \in \RR^{\nu \times 1}I$, which implies that $I$ is
$L$-real.  Since $I \subset \Lrad{L}{\{0\}}$, we have $I =
\Lrad{L}{\{0\}}$.
\end{proof}

\subsection{Decomposition of Linear Pencils}

In this subsection we express a thin spectrahedron as the intersection
of a thick spectrahedron with an affine subspace.

\begin{prop}
\label{prop:decompOfL}
Let $L \in \RR^{\nu \times \nu}\axs$ be a linear pencil such that the linear
matrix inequality $L(X) \succeq 0$ is feasible.
Decompose $\RR\axs_1$ as $(\Lrad{L}{\{0\}})_1 \oplus T$ for some space $T$.
Let $L_T$ be the projection
of $L$ onto $\RR^{\nu \times \nu} \otimes T$, so that $L - L_T \in \RR^{\nu 
\times \nu} \otimes (\Lrad{L}{\{0\}})_1$.
Let \[\cN = \{n \in \RR^{\nu} \mid L_Tn = 0\}\] and suppose $\dim(\cN) < \nu$. Let
$\tilde{L} = C^* L_T C$, where the columns of $C$ form a basis for $\cN^{\bot}$.
\begin{enumerate}[\rm(1)]
 \item \label{item:LLtildeIota}Given a tuple of matrices $X$, $L(X) \succeq
0$ if and only if
$\tilde{L}(X) \succeq 0$ and $\iota(X) = 0$ for each $\iota \in
\Lrad{L}{\{0\}}$.
\item \label{item:TildeLPD}There exists $x \in \RR^g$ such that
$\tilde{L}(x) \succ 0$.
\end{enumerate}
\end{prop}

\begin{proof}
\eqref{item:LLtildeIota}
Let $X \in (\RR^{n \times n})^g$ for some $n \in \NN$.

First, suppose $L(X) \succeq 0$. Proposition \ref{prop:charOfL} implies that
$\iota(X) = 0$ for each
$\iota \in \Lrad{L}{\{0\}}$.
Since $L-L_T \in \RR^{\nu \times \nu} \otimes \Lrad{L}{\{0\}}$, this implies
that
\[ C^*L(X)C = C^*\big([L-L_T](X) + L_T(X)\big)C = \tilde{L}(X) \succeq 0.\]

Conversely, suppose $\tilde{L}(X) \succeq 0$ and $\iota(X) = 0$ for each
$\iota \in \Lrad{L}{\{0\}}$.  We have that $L - L_T \in \RR^{\nu \times \nu}
\otimes \Lrad{L}{\{0\}}$, hence $L(X) = L_T(X)$.  Each $v \in
\RR^{\nu}$ can be expressed as $n_v + C c_v$, where $n_v \in \cN$ and $C c_v \in
\cN^{\bot}$.  Therefore,
\[ v^*L(X)v = (C c_v)^*L_T(X) C c_v =  c_v^*\tilde{L}(X)
c_v \geq 0.\]
Hence $L(X) \succeq 0$.

\eqref{item:TildeLPD}
Let $\nu' = \dim(\cN^{\bot})$.
Suppose
\[\sum_i^{\finite} p_i^*p_i + \sum_j^{\finite} q_j^*\tilde{L}q_j = 0\]
for some $p_i \in \RR$ and $q_j \in \RR^{\nu' \times 1}$.
Then
\[ \sum_i^{\finite} p_i^*p_i + \sum_j^{\finite} q_j^*C^* L C q_j \in \RR^{\nu
\times 1} \Lrad{L}{\{0\}} + (\Lrad{L}{\{0\}})^*\RR^{1 \times \nu}. \]
Therefore $L Cq_j \in \RR^{\nu \times 1}\Lrad{L}{\{0\}}$ for each $j$.
This implies that 
each $\tilde{L} q_j \in  \RR^{\nu' \times
1}\Lrad{L}{\{0\}}$.  Since $\tilde{L} \in \RR^{\nu' \times
\nu'} \otimes T$, and $T \cap \Lrad{L}{\{0\}} = \{0\}$, we have
$\tilde{L}q_j =
0$ for all $j$.
By construction, however, this implies that each $q_j = 0$, which also implies
that each $p_i = 0$.  So $\{0\}$ is strongly $(\tilde{L},
\RR)$-real.

Let $M$ be a truncated test module for $\{0\}$, $\tilde{L}$, and $\RR$.
We see that $-1 \not\in M$ since otherwise $1 + m = 0$ for some $m
\in M$, which would imply that $1 \in \{0\} =
\LraddS{\tilde{L}}{\RR}{\{0\}}$.
By Lemma \ref{lem:main}, there
exists $(X,v) \in V(\{0\})^{n}$ such that $\tilde{L}(X) \succ 0$ and $-v^*1v
<
0$, where $n = \dim(\RR\axs_0) - \dim(\{0\}) = 1$.  Therefore $X \in \RR^g$ and
$\tilde{L}(X) \succ 0$.
\end{proof}

\subsection{Geometric Interpretation of $\Lrad{L}{\{0\}}$}
\label{sub:geomInt}
Given a linear pencil $L \in \RR^{\nu \times \nu}[x]$ of the form
\[
L = A_0 + A_1 x_1 + \cdots + A_g x_g
\]
we can easily construct a linear pencil $L_0 \in \RR^{\nu \times \nu}\axs$
such that $L_0(x) = L(x)$ for each $x \in \RR^g$, namely
\[
L = A_0 + \frac{1}{2} A_1 x_1 + \cdots + \frac{1}{2} A_g x_g + \frac{1}{2} A_1 x_1^* + \cdots + \frac{1}{2} A_g x_g^*.
\]
Using this, we now can prove Theorem \ref{thm:geom}.

\begin{proof}[Proof of Theorem {\rm\ref{thm:geom}}]
Let $L_0$ be such that $L_0(x) = L(x)$ for each $x \in \RR^g$.
Let $\tilde{L}$ be the pencil obtained from $L_0$ using
Proposition \ref{prop:decompOfL}  
Then Proposition \ref{prop:decompOfL} \eqref{item:LLtildeIota} implies that
\[
\{x \in \RR^g \mid L(x) \succeq 0\} = \{x \in \RR^g \mid \tilde{L}(x) \succeq 0 \mbox{ and } \iota(x) = 0 \mbox{ for each } \iota \in I\},
\]
since $x \in \RR^g$ can be viewed as a tuple of $1 \times 1$ matrix variables.
Further, Proposition \ref{prop:decompOfL} \eqref{item:TildeLPD} implies that the spectrahedron $\cD_{\tilde L}(1)$ has nonempty interior.
\end{proof}

Here is the geometric interpretation of Theorem \ref{thm:geom}:  if $L$ is a linear pencil which defines a spectrahedron with empty interior, then either the spectrahedron is empty or it can be viewed as a spectrahedron inside a proper affine linear subspace of $\RR^g$, with the new spectrahedron having nonempty interior. The affine linear subspace is defined by
\[
\{x \in \RR^g \mid \iota(x) = 0 \mbox{ for each } \iota \in \Lrad{L}{\{0\}}\}
\]
and the new spectrahedron is defined by $\tilde{L}(x) \succeq 0$.
Hence 
 the commutative collapse of $\Lrad{L}{\{0\}}$ defines the affine subspace 
containing the thin spectrahedron $\cD_L(1)$ found in \cite{KS}.

\section{Computation of Real Radicals of Left Modules}
\label{sect:CompRad}

Given a linear pencil $L \in
\RR^{\nu \times \nu}\axs$, a left module $I \subset \RR^{1 \times \ell}\axs$, 
and a
right chip space $\rC \subset \RR^{1 \times
\ell}\axs$, 
this section describes algorithms for computing the real radicals
$\rr{I}$, $\Lrad{L}{\{0\}}$ and $\Lradd{L}{\rC}{I}$.
Computing these radicals helps one verify
whether or not 
a polynomial $p \in \RR^{\ell
\times \ell}\axs$ is positive where $L$ is positive and each $\iota\in I$ vanishes, i.e., whether
\eqref{eq:clasPosIdeal} holds for  $p$, which we describe in detail in \S\ref{sub:verp}.

\subsection{Left Gr\"obner Bases}
\label{sect:LGB}

Left monomial orders on $\axs$ are used to compute left Gr\"obner bases
for left ideals $I \subset \RR\axs$.
There is a general theory of one-sided Gr\"obner bases for one-sided modules
with coherent bases over algebras with ordered multiplicative basis
\cite{Gre00}.
In \cite{N} there is a version of this theory specific to our case.
 Left Gr\"obner bases
are easily
computable and are used to algorithmically determine membership in a left
module.  In this section we recall the highlights of the left Gr\"obner basis
theory found in \cite{N}.

Given a total order $\prec$ on $\R^{1 \times \ell}\axs$, we say the {\bf leading
monomial} of a polynomial $p$ is the highest monomial, according to $\prec$,
appearing in $p$.  We denote this leading monomial as $\lead{p}$.  Given a
subset $S \subset \RR^{1 \times \ell}\axs$, let $\lead{S}$ denote the set of
leading monomials of elements of $S$.

A {\bf left admissible order} $\prec$ on the monomials in $\RR^{1 \times 
\ell} \axs$ is a well-order such
that $a \prec b$ for some monomials $a, b \in \RR^{1 \times \ell}\axs$ 
implies that for each $c \in \axs$ we
have $ca \prec cb$.
Given a left module $I \subset \RR^{1 \times
\ell}\axs$, a subset $\cG \subset I$
is a {\bf left Gr\"obner basis of $I$ with respect to $\prec$}\index{left
Grobner basis} if the left module generated by $\lead{\cG}$ equals the left
module generated by $\lead{I}$.
We say a polynomial $p$ is {\bf monic} if the coefficient of $\lead{p}$ in $p$
is $1$.
We say a left Gr\"obner basis $\cG$ is {\bf reduced} if the following hold:
\begin{enumerate}[\rm(1)]
\item Every element of $\cG$ is monic.
 \item If $\iota_1, \iota_2 \in \cG$, then $\lead{\iota_1}$ does not divide any
of the terms of $\iota_2$ on the right.
\end{enumerate}

\begin{prop}
 Let $I \subset \RR^{1 \times \ell}\axs$ be a left module and let $\prec$ be a
left admissible order. Then
 \begin{enumerate}[\rm(1)]
  \item There is a left Gr\"obner basis for $I$ with respect to $\prec$.
  \item There is a unique reduced left Gr\"obner basis for $I$ with respect to
$\prec$.
  \item If $\cG$ is a left Gr\"obner basis for $I$ with respect to $\prec$,
then $\cG$ generates $I$ as a left module.
 \item $\RR^{1 \times \ell} = I \oplus \operatorname{Span} \left( \nonlead{I}
\right)$.
 \end{enumerate}
\end{prop}

\begin{proof}
 See \cite[Propositions 4.2, 4.4]{Gre00}.
\end{proof}

\begin{prop}[\protect{\cite[Lemma 8.2]{N}}]
\label{prop:lGBVerMemOG}
Let $I \subset \RR^{1 \times \ell}\axs$ be a left module and let
$\{\iota_i\}_{i\in \alpha}$ be a left Gr\"obner basis for $I$.
Every element $p \in I$ can be expressed
uniquely as
\begin{equation}
p = \sum_{i}^{\finite} q_i \iota_i,
\end{equation}
for some $q_i \in \RR\axs$.
In particular,
the leading monomial of $p$ is divisible on the
right by the leading monomial of one of the left Gr\"obner basis elements
$\iota_i$.
\end{prop}

\subsubsection{Algorithm for Computing Reduced Left Gr\"obner Bases}

Let $\prec$ be a left monomial order on $\RR^{1 \times \ell}\axs$.
Let
$I$ be the left module generated by polynomials $\iota_1, \ldots, \iota_{\mu}
\in
\RR^{1 \times \ell}\axs$.
It is easy to show that inputting $\iota_1, \ldots, \iota_{\mu}$ into the
following algorithm computes a
reduced left
Gr\"obner basis for $I$.

\begin{enumerate}[\rm(1)]
 \item \label{step:begin} Given: $\cG = \{ \iota_1, \ldots, \iota_{\mu}\}$.
\item\label{step:begin2} If $0 \in \cG$, remove it.
Further, perform scalar multiplication so that each element of $\cG$ is monic.
\item
For each $\iota_i, \iota_j \in \cG$,
compare $\lead{\iota_i}$ with the terms of $\iota_j$.
\begin{enumerate}
 \item If $\lead{\iota_i}$ divides a term of $\iota_j$ on the right,
let
$q \in \axs$ and $\xi \in \RR$ be such that $\xi q\lead{\iota_i}$ is a
term in $\iota_j$. 
Replace $\iota_j$ with $\iota_j - \xi q \iota_i$.
Return to \eqref{step:begin2}.
\item If $\lead{\iota_i}$ does not divide any terms of any $\iota_j$
on the right for any $i
\neq j$,
stop and output $\cG$.
\end{enumerate}
\end{enumerate}

\subsection{The \texorpdfstring{$L$}{[L]}-Real Radical Algorithm}
\label{sect:algorithms}

We now turn our attention to computing $\Lradd{L}{\rC}{I}$.
A special case of $\Lradd{L}{\rC}{I}$ is $\Lradd{L}{\RR}{\{0\}}$.
Proposition \ref{prop:L0genByLinear} implies that $\Lradd{L}{\RR}{\{0\}}$
is a generating set for $\Lrad{L}{\{0\}}$.
For each linear pencil $L$, each left module $I \subset \RR\axs$, and
each full right chip space $\rC \subset \RR\axs$, we always have
$\Lrad{L}{\{0\}}
\subset \Lradd{L}{\rC}{I}$ since $0 \in I$ and $\RR \subset \rC$.
Further, Corollary \ref{cor:dimOfLrad0} implies that
computing
$\Lrad{L}{\{0\}}$ automatically gives $
\Lrad{L}{\{(0, \ldots, 0)\}} \subset \RR^{1 \times \ell}\axs$.
As previously noted, the commutative collapse of $\Lrad{L}{\{0\}}$ is generated by a set of linear polynomials, which  is given by \cite{KS}.

When a linear pencil $L \in \RR^{\nu \times \nu}\axs$ is inputted into the
following algorithm, the algorithm outputs a
generating set of linear polynomials for $\Lrad{L}{\{0\}}$.

\subsubsection{The \texorpdfstring{$L$}{[L]}-Real Radical Algorithm for
\texorpdfstring{$\Lrad{L}{\{0\}}$}{[sqrt L 0]}}
\label{alg:Lradd0}
\begin{enumerate}[\rm(1)]
 \item Let $I^{(0)} = \{0\}$, $T^{(0)} = \{1, x_1, \ldots, x_g, x_1^*, \ldots,
x_g^*\}$.  Fix a total
order $\prec$ on the letters $x_1, \ldots, x_g, x_1^*, \ldots, x_g^*$.
\item\label{item:cN0} Compute the space $\cN^{(0)}
\subset \RR^{\nu^{(i)}}$
defined as
\[ \cN^{(0)} = \{ n \in \RR^{\nu} \mid L n = 0\}.\]
Define $\nu^{(0)}$ to be the dimension of
$(\cN^{(0)})^{\bot}$.  If $\nu^{(0)} = 0$, then stop and output $I^{(0)} =
\{0\}$ and
$\tilde{L} = 1 \in \RR^{1 \times 1}\axs$.
Otherwise, let $\{\xi_1^{(0)}, \ldots, \xi_{\nu^{(0)}}^{(0)}\} \subset
\RR^{\nu}$
be an orthonormal basis for $\cN^{\bot}$, and
compress $L$ onto $([\cN^{(0)}]^{\bot})^*([\cN^{(0)}]^{\bot})$ as
\[L^{(0)} :=
\begin{pmatrix}
\xi_1^{(0)}& \cdots & \xi_{\nu^{(0)}}^{(0)}
\end{pmatrix}^*
L
\begin{pmatrix}
\xi_1^{(0)}& \cdots & \xi_{\nu^{(0)}}^{(0)}
\end{pmatrix}.
\]
\item Let $i = 0$.
 \item \label{step:Agen} Consider the problem
\begin{equation}
  \label{eq:toSolveInAlg}
\operatorname{Tr}(L^{(i)} A^{(i)}) + c^{(i)} = 0 \qquad \qquad A^{(i)} \succeq
0,\ c^{(i)} > 0.
\end{equation}
\begin{enumerate}
\item If \eqref{eq:toSolveInAlg} has a solution with $c^{(i)} > 0$, output
$\{1\}$, $\tilde{L} = 1 \in \RR^{1 \times 1}\axs$, and stop.
 \item If \eqref{eq:toSolveInAlg} has a solution $0 \neq A^{(i)} \succeq 0$ and
$c^{(i)}=0$,
then
let $\iota_{(1,1)}^{(i)}, \ldots, \iota_{(\nu^{(i)},\nu^{(i)})}^{(i)}$ be
defined by
\begin{equation}
 \label{eq:defiota}
\left(
\begin{array}{cccc}
 \iota_{(1,1)}^{(i)}&\iota_{(1,2)}^{(i)}&\cdots&\iota_{(1,\nu^{(i)})}^{(i)}\\
 \iota_{(2,1)}^{(i)}&\iota_{(2,2)}^{(i)}&\cdots&\iota_{(2,\nu^{(i)})}^{(i)}\\
\vdots&\vdots&\ddots&\vdots\\
 \iota_{(\nu^{(i)},1)}^{(i)}&\iota_{(\nu^{(i)},2)}^{(i)}&\cdots&\iota_{(\nu^{(i)
},\nu^{ (i)})}^{ (i)}\\
\end{array}
\right)
:= L^{(i)}\sqrt{A^{(i)}}.
\end{equation}
\begin{enumerate}
\item Define $I^{(i+1)}$ to be a reduced left Gr\"obner basis for the set
\begin{equation}
\label{eq:defIiplus1}
  I^{(i)} \cup \left\{ \iota_{(j,k)}^{(i)} \right\}_{j,k=1}^{\nu^{(i)}}.
\end{equation}
\item If $I^{(i+1)} = \{1\}$, stop and output $I^{(i+1)}$ and $\tilde{L} = 1$.
\item Let $T^{(i+1)}$ be the set containing $1$ and all letters $x_i$ and
$x_i^*$
such that neither $x_i$ nor $x_i^*$ is a
leading letter of an element of $I^{(i+1)}$ or $(I^{(i+1)})^*$.
\item Perform division in order of $\prec$ on the entries of $L^{(i)}$ using
$I^{(i+1)}$ and $(I^{(i+1)})^*$ to get
\[L^{(i)} = L_I^{(i)} + L_T^{(i)},\]
where the entries of $L_I^{(i)}$ are spanned by
 $I^{(i+1)} + (I^{(i+1)})^*$ and the entries of $L_T^{(i)}$ are in $T^{(i+1)}$.
\item Compute the space $\cN^{(i+1)}
\subset \RR^{\nu^{(i)}}$
defined as
\[ \cN^{(i+1)} = \{ n \in \RR^{\nu^{(i)}} \mid L_T^{(i)} n = 0\}.\]
Define $\nu^{(i+1)}$ to be the dimension of
$(\cN^{(i+1)})^{\bot}$.
\item If $\nu^{(i+1)} = 0$ then stop and output $I^{(i+1)}$ and $\tilde{L} = 1
\in
\RR^{1 \times 1}\axs$.
\item Otherwise, let $\{\xi_1^{(i+1)}, \ldots, \xi_{\nu^{(i+1)}}^{(i+1)}\}
\subset \RR^{\nu^{(i)}}$
be a basis for $(\cN^{(i+1)})^{\bot}$.  Let $L^{(i+1)}$ be defined by
\[L^{(i+1)} :=
\begin{pmatrix}
\xi_1^{(i+1)}& \cdots & \xi_{\nu^{(i+1)}}^{(i+1)}
\end{pmatrix}^*
L_T^{(i)}
\begin{pmatrix}
\xi_1^{(i+1)}& \cdots & \xi_{\nu^{(i+1)}}^{(i+1)}
\end{pmatrix}.
\]
\item Let $i:= i+1$ and go to \eqref{step:Agen}.
\end{enumerate}
\item If \eqref{eq:toSolveInAlg} has no nonzero solution $A^{(i)} \succeq 0$
and $c^{(i)} \geq 0$,
then stop and output $I^{(i)}$ and $L^{(i)}$.
\end{enumerate}
\end{enumerate}

\subsubsection{Properties of the \texorpdfstring{$L$}{[L]}-Real Radical Algorithm
for
\texorpdfstring{$\Lrad{L}{\{0\}}$}{[sqrt L 0]}}
\begin{prop}
Let $L \in \RR^{\nu \times \nu}\axs$ be a linear pencil.
 The $L$-Real Radical algorithm for $\Lrad{L}{\{0\}}$ in {\rm\S\ref{alg:Lradd0}} has
the
following properties.
\begin{enumerate}[\rm(1)]
 \item The algorithm terminates in a finite number of steps.
 \item The only polynomials involved in the algorithm have degree $\leq1$.
 \item The algorithm outputs a space of linear polynomials
which generate
$\Lradd{L}{\RR}{\{0\}}$ and, consequentially, $\Lrad{L}{\{0\}}$.
 \item \label{it:outputTildeL}The algorithm also outputs a linear pencil
$\tilde{L}$ such that
$L(X) \succeq 0$ if and only if
$\tilde{L}(X)
\succeq 0$ and $\iota(X) = 0$ for each $\iota \in \Lrad{L}{\{0\}}$.
Further, there exists a real scalar solution $x\in\R^g$ to the linear matrix equality
$\tilde{L}(x) \succ 0$.
\end{enumerate}
\end{prop}

\begin{proof}
First, it is clear that the algorithm
only involves linear polynomials.
Next, the algorithm stops at step
\eqref{item:cN0} if and only if $L = 0$.  In this case, it is trivial to see
that $\Lradd{0}{\RR}{\{0\}} = \{0\}$,
and so the algorithm outputs $\{0\}$ which generates $\Lradd{0}{\RR}{\{0\}}$.
Also note that $\tilde{L} = 1$ is always positive definite.

Given an index $i$, assume inductively that
$I^{(i)} \subset (\Lradd{L}{\RR}{\{0\}})_1$, that
\begin{equation}
\label{eq:inductionC1}
  \RR\axs_1 = \operatorname{span}(I^{(i)} + [I^{(i)}]^{\ast}) \oplus
\operatorname{span}T^{(i)},
\end{equation}
and that for $i > 0$,
the set $I^{(i)}$ is a reduced left Gr\"obner basis
for $\Lradd{L}{\RR}{\{0\}}$.
Also assume $L^{(i)}$ is of the form
\begin{equation}
 \label{eq:inductionC2}
 L^{(i)} = (C^{(i)}) \big(L - L_I^{(i)}\big) (C^{(i)})^*,
\end{equation}
where $L_I^{(i)} \in \RR^{\nu \times \nu} \otimes (I^{(i)}+ [I^{(i)}]^{\ast})$
so that $L -
L_I^{(i)} \in \RR^{\nu \times \nu} \otimes T^{(i)}$, and $C^{(i)}$ is
a matrix whose columns are a basis for the space $(\cZ^{(i)})^{\bot}$
defined by
\[ \cZ^{(i)} := \{ n \in \RR^{\nu} \mid  (L - L_I^{(i)})n = 0 \}.\]

Suppose there is a nonzero solution to \eqref{eq:toSolveInAlg}.
Decompose $A^{(i)}$ as $A^{(i)} = (U^{(i)})^{\ast}\Lambda (U^{(i)})$, where
$\Lambda$
is a diagonal matrix with entries $\psi_1^{(i)}, \ldots, \psi_{\nu^{(i)}}^{(i)}$
and $U^{(i)} \in \RR^{\nu^{(i)} \times \nu^{(i)}}$ is an orthogonal matrix.
Then
\[A^{(i)} = \sum_{j=1}^{\nu^{(i)}}
\left(a_j^{(i)}\right)\left(a_j^{(i)}\right)^*,\]
where $a_j^{(i)} = \sqrt{\xi_j^{(i)}} U^{\ast} e_j^* \in \RR^{\nu{(i)}}$.
Hence \eqref{eq:toSolveInAlg} implies that
\[c^{(i)} + \sum_{j=1}^{\nu^{(i)}} \left(a_j^{(i)}\right)^* L^{(i)}
\left(a_j^{(i)}\right)  = 0.\]
Therefore,
\[c^{(i)} + \sum_{j=1}^{\nu^{(i)}} \left(\left[ C^{(i)}
\right]^*a_j^{(i)}\right)^* L
\left(\left[ C^{(i)}
\right]^*a_j^{(i)}\right) \in  \Lradd{L}{\RR}{\{0\}} +
(\Lradd{L}{\RR}{\{0\}})^*,
\]
since $I^{(i)} \subset \Lradd{L}{\RR}{\{0\}}$.  This implies that
each $L^{(i)} [C^{(i)}]^* a_j^{(i)} \in \Lradd{L}{\RR}{\{0\}}$ and
$\sqrt{c^{(i)}} \in \Lradd{L}{\RR}{\{0\}}$.  If $c^{(i)} > 0$, then this
implies that $\Lradd{L}{\RR}{\{0\}} = \RR\axs$ so that the algorithm outputs
$\{1\}$, a
generating set for $\Lradd{L}{\RR}{\{0\}}$, and $\tilde{L} = 1$,
and the condition that
$\iota(X) =
0$ for
each $\iota \in \Lradd{L}{\RR}{\{0\}}$ is infeasible.  If $c^{(i)} = 0$ but
$A^{(i)} \neq 0$,
then since not all of the $a_j^{(i)}$ are $0$, it follows that each
nonzero $L^{(i)}
a_j^{(i)} \in \Lradd{L}{\RR}{\{0\}}$.
Therefore,
\[L^{(i)}\sqrt{A^{(i)}} = L^{(i)} \begin{pmatrix}
                                   a_1^{(i)}& \cdots &a_{\nu^{(i)}}^{(i)}
                                  \end{pmatrix}U
\]
has entries in $ \Lradd{L}{\RR}{\{0\}} \setminus I^{(i)}$.

Given $I^{(i+1)}$, we find $T^{(i+1)}$ which satisfies the necessary assumptions
given above.  Further, if $L_T^{(i)} = 0$, then we see that if $\iota(X) = 0$
for each $\iota \in I^{(i+1)}$, then $L(X) = 0$.  In this case, $L(X) \succeq
0$ if and
only if $\iota(X) = 0$ for each $\iota \in I^{(i+1)}$.  Therefore
$\Lradd{L}{\RR}{\{0\}}$ is generated by $I^{(i+1)}$, which we output, and we
choose $\tilde{L} = 1$, which is always positive definite.
If $L_T^{(i)} \neq 0$, then we have $L^{(i+1)}$ of the form $(C^{(i+1)})^*(L -
L_{I}^{(i+1)})(C^{(i+1)})$, which satisfies the assumptions given above.

This algorithm must terminate in a finite number of iterations since at each
iteration we add some linear polynomials to $I^{(i)}$ to get $\RR\axs
I^{(i)}\subsetneq
\RR\axs I^{(i+1)} \subset \Lradd{L}{\RR}{\{0\}}$.
At the end,
there is no nonzero solution to \eqref{eq:toSolveInAlg}.  Therefore, if
\[\sum_j^{\finite} p_j^*p_j + \sum_k^{\finite} q_k^*Lq_k \in (\RR\axs I^{(i)}) +
\left(\RR\axs I^{(i)}\right)^*,\]
since the entries of $L^{(i)}$ are in $T^{(i)}$,
\[\sum_j^{\finite} p_j^*p_j + \sum_k^{\finite}
(C^{(i)}q_k)^*L^{(i)}(C^{(i)}q_k) =0.\]
Hence
\[\operatorname{Tr}\left(L^{(i)}\left[\sum_k^{\finite}(C^{(i)}q_k)(C^
{ (i) }q_k)^*\right] \right) + \left(\sum_j^{\finite} p_j^*p_j \right) =
0,\]
which implies, since there is no nonzero solution to \eqref{eq:toSolveInAlg},
that each $C^{(i)}q_k = 0$ and each $p_j = 0$.  Therefore, each $Lq_k =
L^{(i)}q_k + L_I^{(i)} C^{(i)}q_k \in \RR^{\nu^{(i)} \times 1}I^{(i)}$.
This implies that $I^{(i)}$ is $(L,\RR)$-real.  Since $I^{(i)} \subset
\Lradd{L}{\RR}{\{0\}}$, this implies that $I^{(i)} = \Lradd{L}{\RR}{\{0\}}$.

Finally, Proposition \ref{prop:decompOfL}
implies, given the construction of $L^{(i)}$,
that the outputted
$\tilde{L} = L^{(i)}$ satisfies all the properties given in
 \eqref{it:outputTildeL}.
\end{proof}

\subsection{Examples}

Here are some examples of linear pencils $L$ 
and their real radicals
$\Lradd{L}{\RR}{\{0\}}$.

\begin{exa}
 Let $L$ be the pencil
\[
L = \left(
\begin{array}{cc}
 1&x_1\\x_1^*&1
\end{array}
\right).
\]
Note that 
\[
\cD_L =\{ X_1 \mid \|X_1\|\leq1\},
\]
and $L(X) \succ 0$ iff $\|X_1\| < 1$. Proposition
\ref{prop:charOfL} implies that $\Lrad{L}{\{0\}} = \{0\}$.  Proposition
\ref{prop:qRadRealHom} also implies that $\Lrad{L}{\{0\}} = \{0\}$.  Therefore
we expect the $L$-Real Radical algorithm to output $\{0\}$ and $L$.

We now run the algorithm.  First, we see that $Ln \neq 0$ for any $n\in \RR^2
\setminus \{0\}$ since $L(0) \succ 0$, whence $L^{(0)} = L$.  Next, we see that if
$A^{(0)} = (a_{jk})_{1 \leq j,k \leq 2}$, then
\[ \tr(L^{(0)} A^{(0)}) = a_{11} + a_{22} + a_{12} x_1 +a_{21}x_1^*.\]
For this to be constant, we need $a_{12} = a_{21} = 0$.  Next, if $A^{(0)}
\succeq 0$, then $a_{11}, a_{22} \geq 0$. Hence there is no nonzero
solution to \eqref{eq:toSolveInAlg}.  We therefore stop and output $\{0\}$ and
$\tilde{L} = L$.

\end{exa}

\begin{exa}
Let $L$ be the pencil
\[
L = \left(
\begin{array}{cc}
 1&x_1\\x_1^*&0
\end{array}
\right).
\]
Note that since there is a $0$ on the diagonal,   we have
\[
\cD_L=\{X_1 \mid X_1=0\}.
\]

For the algorithm, we first see that $Ln = 0$ has no nonzero solution in
$\RR^2$.  Therefore $L^{(0)} = L$.
Next, we see that if $A^{(0)} = E_{22} \succeq 0$, then
\[ \tr(L^{(0)}A^{(0)}) = 0.\]
Since $\sqrt{A^{(0)}} = E_{22}$, we see
\[ L^{(0)}\sqrt{A^{(0)}} = \left(
\begin{array}{cc}
 0&x_1\\
0&0
\end{array}
\right),\]
so $I^{(1)} = \RR x_1$.
We decompose $\RR\axs_1$ as
\[ \RR\axs_1 = (I^{(1)} + [I^{(1)}]^*) \oplus T^{(1)} \quad \mbox{with}
\quad T^{(1)} := \RR 
. \]
When we project $L$ onto $\RR^{2 \times 2} \otimes T^{(1)}$ we get
\[\left(
\begin{array}{cc}
 1&0\\
0&0
\end{array}
\right) = E_{11}. \]
The nullspace of $E_{11}$ is $\cN^{(1)} = \RR e_2$.  The
compression of $E_{11}$ onto the
space $([\cN^{(1)}]^{\bot})^*([\cN^{(1)}]^{\bot})$ is $L^{(1)} = (1) \in \RR^{1
\times 1}\axs$.

There is no nonzero solution to
\[ \tr(L^{(1)} A^{(1)}) + c = 0,\]
 so we output $\Lradd{L}{\RR}{\{0\}} = \{x_1\}$ and $\tilde{L} = (1)$.
\end{exa}

\begin{exa}
 Let $L$ be the pencil
\[
L = \left(
\begin{array}{cc}
 x_1+x_1^*&1\\1&0
\end{array}
\right).
\]
This pencil $L$ is clearly infeasible, i.e., $\cD_L=\varnothing$.

Applying the algorithm, we get $L^{(0)} = L$, and we see that
\[ \tr(L^{(0)}E_{22}) = 0.\]
Next,
\[
L^{(0)}\sqrt{E_{22}} =
\left(
\begin{array}{cc}
 0&1\\
0&0
\end{array}
\right),
\]
so we add $1$ to $I^{(0)}$ to get $I^{(1)} = \RR\axs_1$.
Output now $I^{(1)}$ and $\tilde{L} = 1$.
\end{exa}

\begin{exa}
This is a version of \cite[Example 4.6.3]{KS} presented in free non-symmetric variables.
 Let $L$ be the pencil
 \[
  L = \left(
  \begin{array}{ccc}
   0&x_1&0\\
   x_1^*&x_2+x_2^*&1\\
   0 & 1 & x_1+x_1^*
  \end{array}
  \right)
 \]

Applying the algorithm, we get $L^{(0)} = L$ and thus
\[
 \operatorname{Tr}(L^{(0)} E_{11}) = 0.
\]
Next,
\[
 L^{(0)} \sqrt{E_{11}} =
 \left( \begin{array}{ccc}
         0&0&0\\ x_1^*&0&0\\ 0&0&0
        \end{array}
 \right).
\]
Therefore we add $x_1^*$ to $I^{(0)}$ to get $I^{(1)} = \RR\axs x_1^*$.
This leads to 
\[
 L^{(1)} = \left( \begin{array}{cc}
        x_2+x_2^*&1\\ 1&0
        \end{array}
 \right).
\]
Then,
\[
  \operatorname{Tr}(L^{(1)} E_{22}) = 0,
\]
so we see
\[
 L^{(1)} \sqrt{E_{22}} =
 \left( \begin{array}{cc}
         0&1\\ 0&0
        \end{array}
 \right).
\]
We thus add $1$ to $I^{(1)}$ to obtain $I^{(2)} = \RR\axs$.
Hence $L$ is infeasible.
\end{exa}

\begin{exa}
Let $L$ be the pencil
 \[
L = \left(
\begin{array}{cccc}
1&x_1&x_2&x_3\\
x_1^*&1&0&0\\
x_2^*&0&1&0\\
x_3^*&0&0&0
\end{array}
\right).
\]
The pencil $L$ has no nullspace, so $L^{(0)} = L$.
We see that $\operatorname{Tr}(L^{(0)}E_{44}) = 0$ and 
\[
L^{(0)}\sqrt{E_{44}} = \left(
\begin{array}{cccc}
0&0&0&x_3\\
0&0&0&0\\
0&0&0&0\\
0&0&0&0
\end{array}
\right)
\]
Therefore $I^{(1)} = \RR x_3$.  We then get 
\[
\RR\axs_1 = \left( I^{(1)} + \left[I^{(1)}\right]^* \right) \oplus T^{(1)} \quad \mbox{with} \quad
T^{(1)} := \RR + \sum_{j=1}^2 \RR x_j + \sum_{j=1}^2 \RR x_j^*.
\]
When we project $L^{(0)}$ onto $T^{(1)}$ we get
\[
\left(
\begin{array}{cccc}
1&x_1&x_2&0\\
x_1^*&1&0&0\\
x_2^*&0&1&0\\
0&0&0&0
\end{array}
\right).
\]
This matrix has a null space $\cN^{(1)} = \RR e_4$, so we obtain
\[
L^{(1)} = \left(
\begin{array}{ccc}
1&x_1&x_2\\
x_1^*&1&0\\
x_2^*&0&1
\end{array}
\right).
\]
This pencil has non-empty interior, so $\tilde{L} = L^{(1)}$.  Geometrically, we see that the set $L(x) \succeq 0$ is the two-dimensional spectrahedron defined by $\tilde{L}$, which is the closed ball $$\{(x_1, x_2, 0) \in \RR^3 \mid x_1^2 + x_2^2 \leq 1\}.$$
\end{exa}

\def\al{\alpha}
\begin{exa} 
As our final example we present
a classical example of a spectrahedron used in mathematical optimization
to construct a semidefinite program (SDP) with nonzero duality gap, cf.~\cite[Example 4.6.4]{KS}.
 Let $L$ be the pencil
 \[
  L = \left(
  \begin{array}{ccc}
   \al+x_2+x_2^* & 0 & 0\\
   0 & x_1+x_1^*&x_2 \\
   0 & x_2^* & 0
  \end{array}
  \right)
 \]
for some $\al>0$.

Applying the algorithm, we get $L^{(0)} = L$ and thus
\[
 \operatorname{Tr}(L^{(0)} E_{33}) = 0.
\]
Next,
\[
 L^{(0)} \sqrt{E_{33}} =
 \left( \begin{array}{ccc}
         0&0&0\\0&0&x_2\\ 0&0&0
        \end{array}
 \right).
\]
Therefore we add $x_2$ to $I^{(0)}$ to get $I^{(1)} = \RR\axs x_2$. 
We decompose $\RR\axs_1$ as
\[ \RR\axs_1 = (I^{(1)} + [I^{(1)}]^*) \oplus T^{(1)} \quad \mbox{with}
\quad T^{(1)} := \RR + \RR x_1 +  \RR x_1^*
. \]
Projecting $L$ onto $\RR^{2 \times 2} \otimes T^{(1)}$ yields
\[
 L^{(1)} = \left( \begin{array}{cc}
        \al & 0 \\ 0 & x_1+x_1^* 
        \end{array}
 \right).
\]
There is no nonzero
solution to \eqref{eq:toSolveInAlg}.  We therefore stop and output $I^{(1)}$ and
$\tilde{L} = L^{(1)}$.
\end{exa}

\subsection{$\rC$-Bases}

For right chip spaces $\rC\subset\R^{1 \times \ell}\axs$, we would ideally like to find an order on $\R^{1 \times \ell}\axs$
 satisfying $a \prec b$ whenever $a \in \rC$ and $b
\not\in \rC$.  However, as it turns out 
 right chip space $\rC$ rarely admit 
left admissible orders (as defined previously in \S\ref{sect:LGB}) with this property. Therefore, we
discuss $\rC$-orders, which were introduced in \cite{N}.

Let $\rC \subset \RR^{1 \times \ell}\axs$ be a right chip space.
Let $\prec_0$ be a degree order on $\axs$, that is,  $\prec_0$ is a total
order on $\axs$ satisfying $a \prec_0 b$ whenever $|a| < |b|$.
We say that $\prec_{\rC}$ is a \textbf{$\rC$-order
(induced by $\prec_0$)}\index{C-order@$\rC$-order}
if $\prec_{\rC}$ is a total order on $\R^{1 \times \ell}\axs$ such that if $a,b \in
\R^{1 \times \ell}\axs$, then $a \prec_{\rC} b$ if
one of the following hold:
\begin{enumerate}
 \item $a \in \rC$ and $b \not\in \rC$;
 \item $a \in \RR\axs \rC$ and $b \not\in \RR\axs \rC$;
 \item $a = a_1a_2$, $b = b_1b_2$, where $a_2,b_2 \in \RR\axs_1 \rC \setminus
\rC$, $a_1, b_1 \in \axs$, and $a_1 \prec_0 b_1$;
\item $a = wa_2$, $b = wb_2$, where $a_2,b_2 \in \RR\axs_1 \rC \setminus
\rC$, $w \in \axs$, and $a_2 \prec_{\rC} b_2$.
\end{enumerate}
The above conditions in and of themselves only define a partial
order. By definition, a $\rC$ order $\prec_{\rC}$ is defined in some way
among
the elements
of $\rC$, $\RR\axs_1 \rC \setminus \rC$, and $\RR^{1 \times \ell} \setminus
\RR\axs \rC$ respectively to make it a total order.

Further, let $I \subset \RR^{1 \times
\ell}\axs$ be a left module generated by polynomials in $\RR\axs_1 \rC$.
We say that a pair of sets $(\{ \iota_i\}_{i \in A},
\{\vartheta_j\}_{j \in B})$ is a
\textbf{$\rC$-basis}\index{C-basis@$\rC$-basis} for $I$ if $\{ \iota_i\}_{i \in
A}$ is a maximal set of monic polynomials in
$I \cap
(\RR\axs_1 \rC \setminus \rC)$ with distinct leading monomials
and if $\{\vartheta_j\}_{j \in B}$
is a maximal (possibly empty) set of monic polynomials in $I \cap \rC$ with
distinct
leading monomials.

Here is a useful property of $\rC$-bases.

  \begin{lemma}[{\cite[Lemma 3.4]{N}}]
Let $\rC \subset \RR^{1 \times \ell}\axs$ be a finite right chip space
and let $\prec_{\rC}$
be a $\rC$-order induced by some degree order.
 Let $I \subset \RR^{1 \times \ell}\axs$ be a left module
generated by polynomials in $\RR\axs_1 \rC$, and
let $(\{\iota_i\}_{i=1}^{\mu},
\{\vartheta_j\}_{j=1}^{\sigma})$ be a $\rC$-basis for $I$.
Then each element of $I$ can be represented uniquely as
\begin{equation}
 \label{eq:formOfAllInI}
\sum_{i=1}^{\mu} p_i \iota_i + \sum_{j=1}^{\sigma} \alpha_j
\vartheta_j,
\end{equation}
where each $p_i \in \RR\axs$ and $\alpha_j \in \RR$.

Conversely, any pair of sets of monic polynomials $(\{\iota_i\}_{i=1}^{\mu},
\{\vartheta_j\}_{j=1}^{\sigma})$ with distinct leading monomials such that any
element of $I$
can be expressed
in the form \eqref{eq:formOfAllInI} is a $\rC$-basis for $I$.
\end{lemma}

\subsection{The $L$-Real Radical Algorithm for
\texorpdfstring{$\Lradd{L}{\rC}{I}$}{[sqrt L C I]}}

Here is an algorithm for the more general real radical $\Lradd{L}{\rC}{I}$, where $L \in
\RR^{\nu
\times \nu}\axs_{\sigma}$ is any symmetric polynomial, $\rC\subset \RR^{1
\times \ell}\axs$ is some
finite right chip space, and $I \subset \RR^{1 \times \ell}\axs$ is a left
module.
When a generating set $\{p_1, \ldots, p_{\mu}\}$ for $I$ is
inputted into the following algorithm, it outputs a $\rC$-basis
for $\Lradd{L}{\rC}{I}$.

\label{alg:Lraddd}
\begin{enumerate}[\rm(1)]
\item Fix some $\rC$-order on $\RR^{1 \times \ell}\axs$.
Compute a $\rC$-basis for $I$, 
and let
 $I^{(0)}$ be the outputted pair of sets.
 \item Let $i=0$.
\item
Let $T^{(i)} \subset
\rC$ be set of all monomials in $\rC$
which are not the leading monomial of an element in
$I^{(i)}$.
\item For the polynomial $\theta^{(i)}$,
\[ \theta^{(i)} = L \left(\sum_{m \in \rC} \alpha_m m \right), \]
where the $\alpha_m$ are $\nu$-dimensional column vector variables, use the
$\rC$-basis to solve for the space of $\alpha$ such that $\theta^{(i)}$ is
in the left module generated by $I^{(i)}$.
Using this solution, let $J^{(i)}$ be a basis for the following space:
\[ \left\{ \theta \in \RR^{\nu \times 1}T^{(i)} \mid L \vartheta
\in I^{(i)}
 \right\}.\]
Let $K^{(i)} \subset \RR^{\nu \times 1}T^{(i)}$ be a maximal set of linearly
independent
polynomials not in $J^{(i)}$.
\item Let $\tau^{(i)} = (\tau_j^{(i)})_{1 \leq j \leq \pi^{(i)}}$ be a vector
whose
entries are the elements of $T^{(i)}$, and let $\kappa^{(i)} =
(\kappa_j^{(i)})_{1
\leq j \leq \rho^{(i)}}$ be a vector whose entries are the elements of
$K^{(i)}$.
Define $L\kappa^{(i)} = (L\kappa_j^{(i)})_{1 \leq j \leq \rho}$.
 \item \label{step:A}
Let $I^{(i)} = \big(\{\iota_j^{(i)}\}_{j=1}^{\mu^{(i)}},
\{\vartheta_j^{(i)}\}_{j=1}^{\sigma^{(i)}}\big)$.
For $1 \leq j \leq \mu^{(i)}$, let $s_j^{(i)} \in \rC$ be defined as
\[s_j^{(i)} := \sum_{c \in \rC} \gamma_{c,j}^{(i)} c,\]
where the $\gamma_{c,j}^{(i)}$ are real-valued variables.
For $1 \leq j \leq \sigma^{(i)}$, let $\alpha_j^{(i)} \in \rC$ be
\[\alpha_j^{(i)} = \sum_{k \in \Gamma(\rC)} \xi_{k,j}^{(i)} e_k \otimes 1,\]
where the $\xi_{k,j}^{(i)}$ are real-valued variables.
Consider the problem of finding $A^{(i)}, B^{(i)} \succeq
0$ such that
\begin{align}
  \label{eq:toSolveInAlgD}
(\tau^{(i)})^*A^{(i)} (\tau^{(i)}) + (\kappa^{(i)})^*B^{(i)}(L\kappa^{(i)})
&= \sum_{j=1}^{\mu^{(i)}} \big([s_j^{(i)}]^*\iota_j^{(i)} +
[\iota_j^{(i)}]^*[s_j^{(i)}]\big)\\
\notag
&= \sum_{j=1}^{\sigma^{(i)}} \big([\alpha_j^{(i)}]^*\vartheta_j^{(i)} +
[\vartheta_j^{(i)}]^*\alpha_j^{(i)}\big)
\end{align}
for some values of $\gamma_{c,j}^{(i)}, \xi_{k,j}^{(i)} \in \RR$.
This can be solved by an algorithm
similar to the SOS algorithm in \cite[\S9.1]{N}; see \S\ref{subsec:modSOS} below.

 \item If \eqref{eq:toSolveInAlgD} has a nonzero solution,
then let $\iota_j^{(i)}$ and $\zeta_{(j,k)}^{(i)}$ be defined by
\begin{align}
 \label{eq:defiotaD}
\begin{pmatrix}
 \iota_1^{(i)}\\ \vdots \\ \iota_{\sigma^{(i)}}
\end{pmatrix}
&=
\sqrt{A^{(i)}} \tau^{(i)}\\
\notag
\left(
\begin{array}{cccc}
 \zeta_{(1,1)}^{(i)}& \zeta_{(1,2)}^{(i)}&\cdots &
\zeta_{(1,\rho^{(i)})}^{(i)}\\
 \zeta_{(2,1)}^{(i)}& \zeta_{(2,2)}^{(i)}&\cdots &
\zeta_{(2,\rho^{(i)})}^{(i)}\\
\vdots&\vdots&\ddots&\vdots\\
 \zeta_{(\rho^{(i)},1)}^{(i)}& \zeta_{(\rho^{(i)},2)}^{(i)}&\cdots &
\zeta_{(\rho^{(i)},\rho^{(i)})}^{(i)}
\end{array}
\right)  &= \sqrt{B^{(i)}}L\kappa^{(i)}.
\end{align}
Let $I^{(i+1)}$ be the $\rC$-basis generated
by the set
\[I^{(i)} \cup \{ \iota_j \}_{j}^{\sigma^{(i)}} \cup \{\zeta_{(j,k)}^{(i)}
\}_{j,k=1}^{\rho^{(i)}}.\]
Let $T^{(i+1)}$ be the space spanned by all monomials in $T^{(i)}$
which are not the leading monomial of
 an element of $I^{(i+1)}$.
Set $i:= i+1$ and go to \eqref{step:A}.
\item If \eqref{eq:toSolveInAlgD} has no nonzero solution,
then stop and output $I^{(i)}$.
\end{enumerate}

\subsubsection{Modified SOS Algorithm}\label{subsec:modSOS}
We now explain in detail how to solve the problem given in Step
\eqref{step:A} of the above algorithm.

\begin{enumerate}[(a)]
\item
Let $\cZ^{(i)}$ be the space
\[
\cZ^{(i)} = \{ (Z_{\tau}, Z_{\kappa}) \in \mathbb{S}^{\pi^{(i)}} \times
\mathbb{S}^{\rho^{(i)}} \mid (\tau^{(i)})^* Z_{\tau} (\tau^{(i)})
+  (\kappa^{(i)})^* Z_{\kappa} (\kappa^{(i)}) = 0 \},
\]
and let $(Z_{i,1,\tau}, Z_{i,1,\kappa}), \ldots, (Z_{i,n^{(i)},\tau},
Z_{i,n^{(i)},\kappa})$ be a basis for $\cZ^{(i)}$.
\item Express the right hand side of \eqref{eq:toSolveInAlgD} as
\begin{multline*}
 \sum_{c \in \rC} \sum_{j=1}^{\mu^{(i)}} \gamma_{c,j}^{(i)} \big(
[\tau^{(i)}]^* F_{c,i,j, \tau} [\tau^{(i)}] + [\kappa^{(i)}]^* F_{c,i,j, \kappa}
[\kappa^{(i)}]\big)
\\
+ \sum_{k \in \Gamma(\rC)} \sum_{j=1}^{\sigma^{(i)}} \xi_{k,j}^{(i)}
\big([\tau^{(i)}]^* H_{k,i,j, \tau}[\tau^{(i)}] + [\kappa^{(i)}]^* H_{k,i,j,
\kappa}[\kappa^{(i)}]\big),
\end{multline*}
for some symmetric matrices $F_{c,i,j, \tau}, F_{c,i,j, \kappa}, H_{k,i,j,\tau},
H_{k,i,j,\kappa}$.
\item
If the linear pencil
\begin{align}
\label{eq:linPenAlgMod}
  L_i(\alpha^{(i)}, \gamma^{(i)}, \kappa^{(i)}) &=
\sum_{j=1}^{n^{(i)}} \alpha^{(i)} (Z_{i,j, \tau} \oplus Z_{i,j,\kappa})
+ \sum_{c,j} \gamma_{c,j} (F_{c,i,j,\tau} \oplus F_{c,i,j,\kappa})\\
\notag
&\quad + \sum_{k,j} \xi_{k,j} ( H_{k,i,j,\tau} \oplus H_{k,i,j,\kappa})
\end{align}
contains any 0 on its diagonal, set all
entries in the row and column corresponding to the $0$ diagonal entry
to be $0$.  Use the resulting
linear equations to reduce the number of variables.
Repeat this step until there are no diagonal entries equal to $0$.
\item If we eventually get $L_i = 0$, stop and output
that there is no nonzero solution.
\item Solve the linear matrix inequality
\[
L_i(\alpha^{(i)}, \gamma^{(i)}, \kappa^{(i)}) \succeq 0
\]
to see if there is a nonzero solution $(\alpha^{(i)}, \gamma^{(i)},
\kappa^{(i)})$.
\item If there is not, stop and output 
that there is no nonzero solution.
\item Otherwise, output the obtained  solution.
\end{enumerate}

\subsubsection{Properties of the \texorpdfstring{$L$}{[L]}-Real Radical Algorithm
for
\texorpdfstring{$\Lradd{L}{\rC}{I}$}{[sqrt L C I]}}
\begin{prop}
Let $L \in
\RR^{\nu \times
\nu}\axs_{\sigma}$ be a symmetric polynomial, let $\rC \subset \RR^{1
\times \ell}\axs$ be a
finite right chip space, and let $I\subset \RR^{1 \times \ell}\axs$ be a left
module.
The $L$-Real Radical algorithm for $\Lradd{L}{\rC}{I}$ in {\rm \S\ref{alg:Lraddd}}
has
the
following properties.
\begin{enumerate}[\rm(1)]
 \item The algorithm terminates in a finite number of steps.
 \item If $I$ is generated by polynomials in $\RR\axs_{\sigma}\rC$,
 then the algorithm involves only computations on polynomials in
$\RR\axs_{\sigma}\rC$.
 \item The algorithm outputs a $\rC$-basis for
$\Lradd{L}{\rC}{I}$.
\end{enumerate}
\end{prop}

\begin{proof}

Given an index $i$, assume inductively that
$I \subset \RR\axs I^{(i)} \subset \Lradd{L}{\rC}{I}$, that $I^{(i)}$ is a
$\rC$-basis, and that
\begin{equation}
\label{eq:inductionC1again}
\rC = \left( \operatorname{span}I^{(i)} \cap \rC \right) \oplus
\operatorname{span}T^{(i)}.
\end{equation}
For $i=0$, this is true by construction
since every monomial in $\rC$ is either the leading monomial of
an element in $I^{(0)}$---in which case it is the leading monomial
of an element of $I^{(0)} \cap \rC$---or it is in $T^{(0)}$.
We compute $J^{(i)}$ and $K^{(i)}$ so that,
by Lemma \ref{prop:lowdegqradcond},  the left module generated
by $I^{(i)}$ is $(L,\rC)$-real if and only if whenever
\begin{equation}
 \label{eq:theorForm}
\sum_r^{\finite} \tau_r^*\tau_r + \sum_{j}^{\finite} \kappa_j^*L\kappa_j \in
\RR^{\ell \times 1}\axs I^{(i)} +
(I^{(i)})^*\RR^{1 \times \ell}\axs
\end{equation}
for $\tau_r \in \operatorname{span} T^{(i)}$ and $\kappa_j \in
\operatorname{span} K^{(i)}$,
then each $\tau_r  = 0 $ and each $\kappa_j = 0$.

If \eqref{eq:toSolveInAlgD} has a nonzero solution, then we see that
\[\sum_{j=1}^{\tau(i)} (\iota_j^{(i)})^*\iota_j^{(i)} + \sum_{r=1}^k
\begin{pmatrix}
 \zeta_{(1,r)}^{(i)}\\ \vdots \\
 \zeta_{(\rho^{(i)},r)}^{(i)}
\end{pmatrix}^*L \begin{pmatrix}
 \zeta_{(1,r)}^{(i)}\\ \vdots \\
 \zeta_{(\rho^{(i)},r)}^{(i)}
\end{pmatrix}  \]
is in the space $ \RR^{\ell \times 1}\axs I^{(i)} +
(I^{(i)})^*\RR^{1 \times \ell}\axs,$
which implies that the $\iota_j^{(i)}, \zeta_{(j,k)}^{(i)} \in
\Lradd{L}{\rC}{I}$.

If \eqref{eq:toSolveInAlgD} has no nonzero solution, then it follows that
\eqref{eq:theorForm} holds if and only if each $\tau_r, \kappa_j = 0$.
Therefore
the left module generated by $I^{(i)}$ is
$(L,\rC)$-real, and since $I \subset \RR\axs I^{(i)} \subset
\Lradd{L}{\rC}{I}$, we have that
 $I^{(i)}$ is a $\rC$-basis for $\Lradd{L}{\rC}{I}$.

At each iteration of the algorithm, we either stop or we add at
least one new
polynomial in $\RR\axs_{\sigma} \rC$ to
$I^{(i)}$.
Therefore the
algorithm  takes at most $\dim\big(\RR\axs_{\sigma}\rC\big)$
iterations.
Also, if I is generated by polynomials in $\RR\axs_{\sigma}\rC$, then
the generating set of each $I^{(i)}$ is made up of polynomials
in $\RR\axs_{\sigma}\rC$.
\end{proof}

\subsection{Verifying if a Polynomial is Positive on a Spectrahedron}
\label{sub:verp}

Given a left module $I \subset \RR^{1 \times \ell}\axs$, a finite right chip
space $\rC \subset \RR^{1 \times \ell}\axs$ and a linear pencil $L \subset
\RR^{\nu \times \nu}\axs$, we now show how to algorithmically verify whether a 
symmetric
polynomial $p \in \rC^*\RR\axs_1 \rC$ is of the form \eqref{eq:clasPosIdeal}.
In particular, by Theorem \ref{thm:mainNMon}, this tells us whether
$p$ is positive where $L$ is positive and each $\iota\in I$ vanishes.

\subsubsection{Algorithm}

\begin{enumerate}
 \item Compute a $\rC$-basis for $\Lradd{L}{\rC}{I}$.  Let $\tilde{\iota}$ be a
vector whose entries are all polynomials in the $\rC$-basis.
 \item Let $c$ be a vector whose entries are all monomials in $\rC$.
 \item Given a $\rC$-order, let $\tau$ be a vector whose entries are all
monomials in $\rC$ which are not divisible on the right by the leading monomial
of an element of the $\rC$-basis for $\Lradd{L}{\rC}{I}$.
 \item Let $\tilde{\kappa}$ be a vector whose entries are all polynomials of the
form $L e_i^* \tau_j$ for some entry $\tau_j$ of $\tau$.
 \item Consider the equation
 \begin{equation}
  \label{eq:inCoeffs}
  p = \tau^*A \tau + \tilde{\kappa}^* B \tilde{\kappa} + \tilde{\iota}^* C
c + \tilde{\iota}^* C^* c
 \end{equation}
 where $A$, $B$, $C$ are unknowns.
 This equation amounts to a series of linear equations in the entries of $A$,
$B$, $C$.
 \item The polynomial  $p$ is of the form \eqref{eq:clasPosIdeal} if and only if
the following linear matrix inequality
is feasible:
 \[
  A, B \succeq 0 \quad \mbox{ such that } \quad \eqref{eq:inCoeffs} \mbox{
holds}
 \]
\end{enumerate}

\section{Completely Positive Maps in the Absence of Invertible Positive Elements}\label{sec:cp}
\def\mbK{\mathbb K}

The theory we have developed can be used to strengthen the theory of complete
positivity (CP). The theorem at the core of the subject represents
a CP map $\tau$ between unital subspaces of matrix algebras
as $\tau(A) = V^*\phi(A) V$, where $\phi$ is an isometric isomorphism.
This can be thought of as an algebraic certificate  for CP,
and it is gotten by combining the Arveson extension theorem
with the Stinespring representation theorem.
In this section we give  algebraic certificates for CP
maps between nonunital subspaces of matrix algebras.

\subsection{Completely Positive Maps}
A subspace $\rA\subseteq \Rnn$ closed under the transpose will be called a 
\df{(nonunital) operator system}.  We write $\mbS(\rA)$ for the set of all symmetric
elements $A=A^*\in\rA$, 
and $\mbK(\rA)=\{A\in\rA \mid A^*=-A\}$ denotes the skew-symmetric elements of $\rA$.
Furthermore, 
\[\rA_{\succeq0}=\{A\in\mbS(\rA)\mid A\succeq0\}.\]

\begin{lem}\label{lem:decomp}
We have
\[
\rA= \mbS(\rA)\oplus\mbK(\rA).
\]
\end{lem}

\begin{proof}
Observe that
\[
A = \frac{A+A^*}2+ \frac{A-A^*}2. \qedhere
\]
\end{proof}

Let $\rB\subseteq\Rll$ be another operator system.
A linear $*$-map $\tau:\rA\to\rB$ is called \df{completely positive (CP)} if it is positive
(i.e., $0\preceq A\in\rA$ implies $\tau(A)\succeq0$) and all its \df{ampliations}
$\tau\otimes \id_k:M_k(\rA)\to M_k(\rB)$, $k\in\N$, are positive.

\begin{lem}\label{lem:CP1}
Suppose $\rA\subseteq\Rnn$ and $\rB\subseteq\Rll$ are operator systems,
and let $\tau:\rA\to\rB$ be a linear $*$-map. Suppose $\rA_{\succeq0}=\{0\}$. 
Then:
\ben[\rm(1)]
\item
$M_k(\rA)_{\succeq0}=\{0\}$ for all $k\in\N$;
\item
$\tau$ is CP.
\een
\end{lem}

\begin{proof}
If $0\neq A\in M_k(\rA)_{\succeq0}$, then at least one of the diagonal  $\nu\times\nu$ blocks $A_{jj}$ of $A$ is 
positive semidefinite and nonzero, violating $\rA_{\succeq0}=\{0\}$.  Item (2) now follows easily.
\end{proof}

\begin{remark}\label{rem:CP1}
For this reason we may restrict our attention the the case of operator systems $\rA$ with 
nonzero $\rA_{\succeq0}$. We point out that detecting whether $\rA_{\succeq0}=\{0\}$ 
is easily done using the machinery developed above. Indeed, 
choose a basis  $\{A_1,\ldots,A_s\}$ of $\mbS(\rA)$,  a basis
$\{A_{s+1},\dots,A_g\}$ of $\mbK(\rA)$, and form the linear pencil
\begin{multline*}
L(x)= A_1(x_1+x_1^*) + \cdots + A_s(x_s+x_s^*)   +  
A_{s+1}(x_{s+1}-x_{s+1}^*)+\cdots+A_g (x_g-x_g^*).
\end{multline*}
Consider the expanded pencil 
\beq\label{eq:Lhat}
\begin{split}
\hat L(x) & = L(x) \oplus \Big( \tr\big(L(x)\big)-1 \Big) \\ 
& = L(x) \oplus  \Big(
\tr(A_1)(x_1+x_1^*) + \cdots + \tr(A_g)(x_s+x_s^*) -1 \Big).
\end{split}
\eeq
Note that 
$\rA_{\succeq0}=\{0\}$ iff
$L(\cD_L)=\{0\}$ iff  $\cD_{\hat L}=\varnothing$ iff 
$\Lrad{\hat L}{\{0\}} = \RR\axs$ (by Proposition \ref{cor:infeas}), 
and the last condition is 
easily detected by the algorithms presented in \S\ref{sect:CompRad}.
\end{remark}

\begin{exa}
In general not every CP map $\tau:\rA\to\rB$ extends to a CP map 
$\hat\tau:\Rnn\to\Rll$.
\ben[\rm(1)]
\item
Suppose $\rA_{\succeq0}=\{0\}$, and let $\tau:\rA\to \Rll$ be any map
\emph{not} of the Arveson-Stinespring form (i.e., 
of the form $X\mapsto \sum_j V_j^*XV_j$). Then $\tau$ is CP by Lemma \ref{lem:CP1},
but cannot be extended to the full matrix algebra.
\item
For a slightly less trivial example, let $\rA=\begin{pmatrix} 0 & \R \\ \R & \R \end{pmatrix}\subseteq \R^{2\times 2}$, and consider
$\tau:\rA\to\R$ given by $\begin{pmatrix} a_{11} & a_{12} \\ a_{21} & a_{22} \end{pmatrix}
\mapsto a_{12}+a_{21}$. Then $\tau(\rA_{\succeq0})=\{0\}$, so $\tau$ is positive. Since it maps
into $\R$, an easy exercise (or see \cite{Pau02}) now shows $\tau$ is CP. But $\tau$ does not admit
an extension to a positive map on $\R^{2\times2}$.
\een
\end{exa}

\subsection{Pencils Associated with Operator Systems}

Retain the notation of Lemma \ref{lem:CP1}. We assume $\rA_{\succeq0}\neq\{0\}$.
As in Remark \ref{rem:CP1}
 we can select a basis $A_0,\ldots,A_{s-1},A_{s},\ldots,A_g$ of $\rA$ consisting
solely of symmetric and skew symmetric elements. Here, $A_0,\ldots,A_{s-1}\in\mbS(\rA)$ and
$A_{s},\ldots,A_g\in\mbK(\rA)$ for $s\in\N$. Let us call this a \df{symmetric basis} for $\rA$.
To such a basis we associate the linear pencil 
\begin{multline}\label{eq:pencCP}
L_{\rA}(x) = A_0 + A_1 (x_1 +x_1^*) + \cdots +  A_{s-1} (x_{s-1} +x_{s-1}^*)   \\
+ 
A_s (x_s -x_s^*)  + \cdots +A_g (x_g -x_g^*)  .
\end{multline}

\begin{prop}\label{prop:CPbounded}
Assume that $A_0\neq0$ 
is a maximum rank positive semidefinite matrix of $\rA$, 
and that $A_0,\ldots,A_{s-1}$ are pairwise orthogonal, i.e., $\tr(A_i^*A_j)=0$.
Then $\cD_{L_\rA}(1)$ is bounded.
\end{prop}

\def\la{\lambda} 

\begin{proof}
Without loss of generality we may assume
\[
A_0= \begin{bmatrix} \Id_r & 0 \\ 0 & 0_{\nu-r} \end{bmatrix}
\]
for some $1\leq r\leq \nu$.

\smallskip
{\bf Claim.} If for some $A_1= \begin{bmatrix} A_{11} & A_{12} \\ A_{12}^* & A_{22}\end{bmatrix}\in\rA$ with $A_{11}\in \mbS^r$ we have
$\langle A_0, A_1\rangle=0$ and
\beq\label{eq:block2}
A_0 + \la A_1 = 
\begin{bmatrix} \Id_r+ \la A_{11} & \la A_{12} \\ \la A_{12}^* & \la A_{22}
\end{bmatrix}
\succeq 0 \quad\text{ for all } \quad \la\in\R_{\geq0},
\eeq
then 
$A_1=0$.

\emph{Explanation.}
Since $\langle A_0,A_1\rangle=0$, 
$\tr(A_{11})=0$. This means that either $A_{11}=0$ or $A_{11}$ has both positive and negative eigenvalues.
In the latter case, fix an eigenvalue $\mu<0$ of $A_{11}$.
Then for every $\la\in\R$ with $\la>-\mu^{-1}>0$, 
we have that $\Id_r+\la A_{11} \not\succeq0$, contradicting \eqref{eq:block2}.
So $A_{11}=0$. If $r=\nu$ we are done. Hence assume $r<\mu$.

Now
\beq\label{eq:preschur}
A_0+ \la A_1 = \begin{bmatrix} \Id_r  & \la A_{12} \\ \la A_{12}^* & \la A_{22}
\end{bmatrix}
\succeq 0
\eeq
for all $\la\in\R_{\geq0}$.
Using Schur complements, \eqref{eq:preschur}
is equivalent to
\[
\la A_{22} - \la^2 A_{12}^*A_{12} \succeq0.
\]
Hence $A_{22} - \la A_{12}^*A_{12} \succeq0$ for all $\la\in\R_{\geq0}$. 
Equivalently,
$A_{12}=0$ and $A_{22}\succeq0$.
If $A_{22}\neq0$, then $0 \preceq A_0+A_1 \in\rA$, and
$$r=\rank(A_0) < \rank(A_0+A_1),$$ contradicting the maximality of the rank of $A_0$.
\hfill$\square$

\smallskip
We now show that $\cD_{L_{\rA}}(1)$ is bounded. Assume otherwise. 
Then there exists a sequence $(x^{(k)})_k$
in $\R^{s-1}$
such
  that $\|x^{(k)}\|=1$ for all $k$, and an increasing sequence
  $t_k\in\R_{>0}$ tending to $\infty$ such that
  $L_\rA(t_k x^{(k)})\succeq0$.  By convexity this implies $t_kx^{(j)}\in \cD_{L_\rA}$ for all $j\ge k$.
Without loss of generality we assume the sequence
   $(x^{(k)})_k$ converges to a vector $x=(x_1,\ldots,x_{s-1})\in\R^{s-1}$.
Clearly, $\|x\|=1$.
For any $t\in\R_{\ge0}$, $tx^{(k)}\to tx$, and for $k$ big enough,
$tx^{(k)}\in \cD_{L_\rA}$ by convexity.
So $x$ satisfies $L_\rA(t x)\succeq 0$ for
all $t\in\R_{\geq0}$.
In other words,
\[ A_0 + 2t ( A_1 x_1 + \cdots +A_{s-1} x_{s-1} ) \geq0
\]
for all $t\in\R_{\geq0}$. But now the claim implies
$A_1x_1+\cdots+A_{s-1}x_{s-1}=0$, contradicting the linear independence of the $A_j$s.
\end{proof}

\def\La{\Lambda}
\def\Rdd{\R^{d\times d}}

\begin{lemma}\label{lem:bounded34}
Let $L(x)$ be as in \eqref{eq:pencCP}, and assume the $A_j$ satisfy 
the assumptions of Proposition {\rm\ref{prop:CPbounded}}.
Then:
\begin{enumerate}[\rm (1)]
\item
  if
 $\Lambda\in \Rdd$ and $Z\in (\Rdd)^g,$
  and if
\begin{multline}\label{eq:T1}
S:= \La\otimes A_0 + (Z_1 +Z_1^*)\otimes A_1 + \cdots +   (Z_{s-1} +Z_{s-1}^*)   \otimes A_{s-1} \\
+ 
 (Z_s -Z_s^*)\otimes A_s  + \cdots + (Z_g -Z_g^*) \otimes A_g 
\end{multline}
  is symmetric, then  $\Lambda =\Lambda^*$;
\item if $S\succeq0$, then
$\Lambda\succeq 0$.
\end{enumerate}
\end{lemma}

\begin{proof}
 To prove item (1), suppose $S$ is symmetric. Then
$
0=S-S^* = (\Lambda-\Lambda^*)\otimes A_0.
$
(Here we have used that $A_0,\ldots,A_{s-1}$ are symmetric and $A_s,\ldots,A_g $ are skew-symmetric.)
Since $A_0\neq0$, we get $\Lambda=\Lambda^*$.

For (2),
  if $\Lambda \not \succeq 0$, then there is a unit vector $v$ such that
$v^*\La v < 0$.
 Consider the orthogonal projection $P\in S\R^{d\nu\times d\nu}$ from $\R^d\otimes\R^\nu$ onto $\R v \otimes \R^\nu$, and
let 
$Y=((v^*Z_iv)P_v)_{i=1}^g\in(\mbS^d)^g$. Here $P_v\in \mbS^{d}$ is the orthogonal projection from $\R^d$ onto $\R v$. Note that
$P=P_v\otimes \Id_\nu$. Then the compression
\[
\begin{split}
PSP& = P\Big(
\La\otimes A_0 + (Z_1 +Z_1^*)\otimes A_1 + \cdots +   (Z_{s-1} +Z_{s-1}^*)   \otimes A_{s-1}
\\
& \qquad \; + 
 (Z_s -Z_s^*)\otimes A_s  + \cdots + (Z_g -Z_g^*) \otimes A_g \Big) P
\\ 
&
     = (v^*\La v)P_v\otimes A_0  +\sum_{i=1}^{s-1} 2 Y_i\otimes A_i \succeq 0,
\end{split}
\]
which implies that  $0\neq\sum_{i=1}^{s-1}Y_i\otimes A_i\succeq 0$ since $0\neq A_0\succeq0$ and $v^*\La v<0$.
This implies $0\neq t Y\in \cD_L$ for all $t>0$. In particular, the spectrahedron
$\cD_{\hat L}$ of the commutative collapse $\hat L$ of $L$ is unbounded. 
Hence  $\cD_L(1)$ is unbounded (cf.~\cite[\S4.1]{KS11}) contradicting Proposition \ref{prop:CPbounded}.
\end{proof}

\subsection{Characterizing Completely Positive Maps}\label{subsec:CP}

Suppose $\rA\subseteq\Rnn$ and $\rB\subseteq\Rll$ are operator systems, and
$\tau:\rA\to\rB$ is a linear $*$-map. Assume $\rA_{\succeq0}\neq\{0\}$ and select
a basis $A_0,\ldots,A_{s-1},A_s,\ldots,A_g$ of $\rA$ consisting of symmetric and
skew-symmetric elements and satisfying the assumptions of Proposition \ref{prop:CPbounded}.
Consider the linear pencil $L_\rA(x)$ given by \eqref{eq:pencCP}, and let
\begin{multline*}\label{eq:pencCP}
L_{\rB}(x) = \tau(A_0) + \tau(A_1) (x_1 +x_1^*) + \cdots +  \tau(A_{s-1}) (x_{s-1} +x_{s-1}^*)   \\
+ 
\tau(A_s) (x_s -x_s^*)  + \cdots + \tau(A_g) (x_g -x_g^*)  .
\end{multline*}

\begin{theorem}\label{thm:dominate}
The following are equivalent:
\ben[\rm (i)]
\item
$\tau$ is CP;
\item
$\cD_{L_\rA}\subseteq\cD_{L_\rB}$;
\item
$\cD_{L_\rA}(\ell)\subseteq\cD_{L_\rB}(\ell)$.
\een
\end{theorem}

\cite[Theorem 3.5]{HKMa} obtained this result for $A_0\succ0$, and \cite[\S4]{KS11} 
considered (complete) positivity of linear functionals $\tau:\rA\to\R$.

\begin{proof}
The implication (i) $\Rightarrow$ (ii) is
obvious. We next prove its converse.

  Fix $d\in\N$. Suppose $S\in M_d(\rA)$ is positive semidefinite.
Then it is of the form \eqref{eq:T1} for some
$\Lambda\in\R^{d\times d}$ and $Z\in(\Rdd)^g$.
By Lemma \ref{lem:bounded34}, $\La\succeq0$.
If we replace $\Lambda$ by $\Lambda+\epsilon I$ for some $\epsilon>0$, the resulting
$S=S_\epsilon$ is still in $M_d(\rA)$, so without loss of generality
we may assume $\Lambda\succ0$.
   Hence,
   \begin{multline*}
     ( \Lambda^{-\frac12}\otimes I) S( \Lambda^{-\frac12}\otimes I) 
    =  
    I \otimes A_0 + \sum_{i=1}^{s-1} \Lambda^{-\frac12}
(Z_i +Z_i^*) \Lambda^{-\frac12} \otimes A_i + 
\sum_{i=s}^g 
\Lambda^{-\frac12}
(Z_i -Z_i^*) \Lambda^{-\frac12} \otimes A_i \succeq0.
\end{multline*}
  Since $\cD_{L_\rA}\subseteq\cD_{L_\rB}$, this implies
 \begin{multline*}
  I \otimes \tau(A_0) + \sum_{i=1}^{s-1} \Lambda^{-\frac12}
(Z_i +Z_i^*) \Lambda^{-\frac12} \otimes \tau(A_i)   + 
\sum_{i=s}^g 
\Lambda^{-\frac12}
(Z_i -Z_i^*) \Lambda^{-\frac12} \otimes \tau(A_i) \succeq0.
 \end{multline*}
   Multiplying on the left and right by $ \Lambda^{\frac12}\otimes I$
   shows
 \begin{multline*}
\tau(S_\epsilon)=
 \La\otimes \tau(A_0) + (Z_1 +Z_1^*)\otimes \tau(A_1) + \cdots +   (Z_{s-1} +Z_{s-1}^*)   \otimes \tau(A_{s-1}) \\
+ 
 (Z_s -Z_s^*)\otimes \tau(A_s)  + \cdots + (Z_g -Z_g^*) \otimes \tau(A_g)\succeq0. 
 \end{multline*}
  A straightforward approximation argument now
  implies that if $S\succeq 0$, then $\tau(S)\succeq 0$ and
  hence $\tau$ is CP. This proves (ii) $\Rightarrow$ (i).
  
  The implication (ii) $\Rightarrow$ (iii) is obvious, and (iii) $\Rightarrow$ (ii) is
  given in Corollary \ref{cor:lmidom}.
\end{proof}

\subsection{Algorithm for Determining Complete Positivity}

Given are operator systems $\rA\subseteq\Rnn$,  $\rB\subseteq\Rll$, and 
  a linear $*$-map $\tau:\rA\to\rB$.

\ben[\rm(1)]
\item\label{it:autoCP}
If $\rA_{\succeq0}=\{0\}$, then $\tau$ is CP. Stop.
\item\label{it:A0}
Find the maximum rank positive semidefinite $A_0\in\rA$.
\item
Compute a basis 
$A_0,\ldots,A_{s-1},A_s,\ldots,A_g$ of $\rA$ consisting only of symmetric and
skew-symmetric elements, satisfying the assumptions of Proposition \ref{prop:CPbounded}.
\item 
Form $L_\rA$ and $L_\rB$ as in \S\ref{subsec:CP}.
\item\label{it:dominate}
Is $L_\rB|_{\cD_{L_\rA}}\succeq0$? If yes, then $\tau$ is CP. If not, $\tau$ is not CP.
\een

The correctness of this algorithm follows from Theorem \ref{thm:dominate}.

\begin{remark}
\mbox{}\par
\ben[\rm(1)]
\item
How to implement \eqref{it:autoCP} 
is explained in Remark \ref{rem:CP1}.
\item
To find a matrix $A_0$ with maximum possible rank one solves the
strictly feasible SDP 
\[\lambda^* = \min  \big\{ \lambda :
\hat{L}(x) + \lambda I \succeq 0,\, \lambda\geq0,\, x\in\R^g\big\},\]
where $\hat L$ is as in \eqref{eq:Lhat},  using a path-following interior-point
method. As shown in \cite[Theorem 5.6.1]{dKTR00} (see also \cite[\S3]{dKl02}), the limit of
the central path is maximally complementary, therefore when $\lambda^* =
0$, the solution of this problem will produce $A_0$ with maximal rank. Note
that if $\lambda^* > 0$, then no feasible $A_0$ exists.
\item
The algorithmic interpretation \S\ref{sub:verp} 
of Theorem \ref{thm:mainNMon} enables us to compute a certificate for $L_\rB|_{\cD_{L_\rA}}\succeq0$, yielding at the same time a certificate for complete positivity of $\tau$.
\een
\end{remark}

\section{Adapting the Theory to Symmetric Variables}
\label{sec:symmVars}

In some contexts, it is desirable to work with NC polynomials in
symmetric variables.  Define $\ax$ to be the monoid freely generated by $x$ 
with identity the empty word, and
let $\RR\ax$ denote the $\RR$-algebra freely generated by $\ax$. Define the
involution $\ast$ on $\RR\ax$ to be linear such that $x_i^* = x_i$ and
such that $(pq)^* = q^*p^*$ for each $p, q \in \RR\ax$.  We say that
elements of $\RR\ax$ are NC polynomials in {\bf symmetric variables}.
We henceforth  refer to polynomials in $\RR\axs$ as polynomials in
{\bf non-symmetric variables}.

There are direct analogs of the results of this paper to the
case of symmetric variables.  It turns out that essentially the same proofs
given throughout this paper work for symmetric variables.
Alternately, some results for symmetric variables can be proved directly from
our existing results on non-symmetric variables.  In this section, we will
prove the analog of Theorem \ref{thm:mainNMon} for symmetric variables.

\begin{lemma}
 \label{lem:symm1}
 Let $I \subset \RR^{1 \times \ell}\ax$ be a left module and let $L \in
\RR^{\nu \times \nu}\ax$ be a linear pencil.
 Then
 \[
  J = \{ p \in \RR^{1 \times \ell}\ax \mid p(X)v = 0 \mbox{ whenever } (X,v)
\in V(I) \mbox{ and } L(X) \succeq 0\}
 \]
 is an $L$-real left module containing $I$.
\end{lemma}

\begin{proof}
Suppose
 \[
  \sum_i^{\rm finite} p_i^*p_i + \sum_j^{\rm \finite} q_j^* L q_j  \in
\RR^{\ell \times 1} J + J^* \RR^{1 \times \ell}.
 \]
Now if $(X,v) \in V(I)$ is such that $L(X) \succeq 0$, then
 \[
  v^*\Big( \sum_i^{\rm finite} p_i(X)^*p_i(X) + \sum_j^{\rm \finite} q_j(X)^*
L(X) q_j(X) \Big) v = \sum_{i}^{\rm finite} \|p_i(X)v \|^2 + \sum_{j}^{\rm
finite} \|\sqrt{L}(X) q_j(X)\| ^2 = 0
 \]
 which implies that each $p_i \in J$ and each $L q_j \in \RR^{\nu \times 1}J$.
\end{proof}

For $p \in \RR^{a \times b}\axs$ define $\sym(p) \in \RR^{a \times b}\ax$ to be
the polynomial produced by setting each $x_i^*$ equal to $x_i$.  If $q \in
\RR^{a \times b} \ax$, define $\free(q) \in \RR^{a \times b}\axs$ to be
\[
\free(q)=
q\Big(\frac{1}{2}(x + x^*)\Big).
\]

Here is the symmetric analog of Theorem
\ref{thm:mainNMon}.

\begin{theorem}
 \label{thm:mainSymmAnalog}

Suppose $L \in \RR^{\nu \times \nu}\ax$ is a linear pencil.
Let $\rC \subset
\RR^{1 \times \ell}\ax$
be a finite chip space,
let $I
\subset \RR^{1 \times \ell}\ax$ be a
left $\RR\ax$-module generated by polynomials in $\RR\ax_1 \rC$,
and let $p \in \rC^*\RR\ax_1 \rC$ be a symmetric polynomial.
\begin{enumerate}[\rm(1)]
 \item\label{it:weakSym} $v^{\ast}p(X)v \geq 0$ whenever $(X,v) \in V(I)$ and
$L(X) \succeq
0$ if and only if $p$ is of the form
\begin{equation}
\label{eq:clasPosIdealSymm}
  p = \sum_{i}^{\finite}p_i^{\ast}p_i + \sum_{j}^{\finite} q_j^{\ast}Lq_j +
  \sum_k^{\finite} (r_k^{\ast}\iota_k + \iota_k^{\ast}r_k)
\end{equation}
where each $p_i, r_k \in \rC$, each $q_j \in
\RR^{\ell \times 1} \rC$ and each $\iota_k \in \Lradd{L}{\rC}{I} \cap \RR\axs_1
\rC$.
\item \label{it:strSym} $v^{\ast}p(X)v \geq 0$ whenever $(X,v) \in V(I)$ and
$L(X) \succ 0$,
if and only if $p$ is of the form
\eqref{eq:clasPosIdealSymm}
where each $p_i, r_k \in \rC$, each $q_j \in
\RR^{\nu \times 1}\rC$ and
each $\iota_k \in \LraddS{L}{\rC}{I} \cap \RR\axs_1 \rC$
\end{enumerate}
\end{theorem}

\begin{proof}
 We will only prove  \eqref{it:weakSym}.  The proof of \eqref{it:strSym} is
similar.

If $p$ is of the form \eqref{eq:clasPosIdealSymm}, then Lemma \ref{lem:symm1}
implies that $v^*p(X)v \geq 0$ whenever $(X,v) \in I$ and $L(X) \succeq 0$.

Conversely, suppose $v^*p(X)v \geq 0$ whenever $(X,v) \in I$ and $L(X) \succeq
0$.  Let $\tilde{\rC} \subset \RR^{1 \times \ell}\axs$ be the right chip
space spanned by all monomials in $\sym^{-1}(\rC)$.  That is, the monomials in
$\tilde{\rC}$ have the property that when all of the $\ast$ are removed, one is
left with a monomial in $\rC$. Therefore, it is easy to see that $\tilde{\rC}$
is finite. 

Let $J \subset \RR^{1 \times \ell}\axs$ be the left module generated
by $\sym^{-1}(I) \cap \RR\axs_1 \tilde{\rC}$.  Note that $\sym(J) = I$. Further,
if $(Y,v) \in V(J)$, and $X = \frac{1}{2}(Y + Y^*)$, then for each $\iota \in I
\cap \RR\ax_1 \rC$ we see $\iota(X)v = \free(\iota)(Y)v = 0$. Also, $\free(L)(Y)
= L(X)$ and $\free(p)(Y) = p(X)$. Therefore $v^*\free(p)(Y)v \geq 0$ whenever
$(Y,v) \in V(J)$ and $\free(L)(Y) \succeq 0$. Theorem
\ref{thm:mainNMon} implies that
\[
 \free(p) = \sum_i^{\rm finite} a_i^*a_i + \sum_j^{\rm finite} b_i^*\free(L)b_i
+ \sum_{k}^{\rm finite} (c_k^*\theta_k + \theta_k^*c_k)
\]
where each $a_i, c_k \in \tilde{\rC}$, $b_i \in \RR^{\nu \times 1} \tilde{\rC}$
and each $\theta_k \in \Lradd{\free(L)}{\tilde{\rC}}{J}$.  Therefore
\begin{multline*}
  p = \sum_i^{\rm finite} \sym(a_i)^*\sym(a_i) + \sum_j^{\rm finite}
\sym(b_i)^*L\sym(b_i) \\
+ \sum_{k}^{\rm finite} \big(\sym(c_k)^*\sym(\theta_k) + \sym(\theta_k)^*\sym(c_k)\big)
\end{multline*}
Hence it  suffices to show that $\sym(\theta_k) \in \Lradd{L}{\rC}{I}$.

Let $K \subset \RR^{1 \times \ell}\ax$ be the left module generated by
$\sym^{-1}(\Lradd{L}{\rC}{I}) \cap \RR\ax_1 \tilde{\rC}$, i.e., the set of  
polynomials in $\RR\ax_1 \tilde{\rC}$ which map into $\Lradd{L}{\rC}{I}$ under 
$\sym$. 
Suppose
\[
 \sum_i^{\rm finite} f_i^*f_i + \sum_{j}^{\rm finite} g_j^*\free(L) g_j \in
\RR^{\ell \times 1} K + K^* \RR^{1 \times \ell},
\]
where each $f_i \in \tilde{\rC}$ and each $g_j \in \RR^{\nu \times
1}\tilde{\rC}$.
Then
\[
  \sum_i^{\rm finite} \sym(f_i)^*\sym(f_i) + \sum_{j}^{\rm finite}
\sym(g_j)^*L \sym(g_j) \in
\RR^{\ell \times 1} \Lradd{L}{\rC}{I} + \Lradd{L}{\rC}{I}^{\,*} \RR^{1 \times \ell},
\]
each
$\sym(f_i) \in \rC$, and each $\sym(g_j) \in \RR^{\nu \times
1}\rC$.  By definition, each $\sym(f_i) \in \Lradd{L}{\rC}{I}$ and each
$\sym(g_j) \in \RR^{\nu \times 1} \Lradd{L}{\rC}{I}$.  This
implies that each $f_i \in K$ and each $g_j \in \RR^{\nu \times 1}
K$. Further, it is clear by definition that $J \subset K$.  Therefore
\[\sym(\Lradd{\free(L)}{\tilde{\rC}}{J}) \subset \sym(K) = \Lradd{L}{\rC}{I}
. \qedhere
\]
\end{proof}

\newpage
\linespread{1.1}

\newpage

\centerline{NOT FOR PUBLICATION}

\appendix
\section{\\ Background from \cite{HV07} 
on Algebraic Interiors and their Degree}\label{app:HV}

This lemma and its proof are copied directly from \cite{HV07}.
To mesh perfectly with this the current paper
take $x^0 =0$ and ${\mathcal C}_p(x^0) = {\mathcal C}_p(0)= \cC$.

\begin{lemma}[\protect{\cite[Lemma 2.1]{HV07}}]
\label{lem:HV07}
A polynomial $p$ of the lowest degree for which ${\mathcal C} = {\mathcal C}_p(x^0)$
is unique $($up to a multiplication by a positive constant$)$, and any other polynomial $q$
such that ${\mathcal C} = {\mathcal C}_q(x^0)$
is given by $q=ph$ where $h$ is an arbitrary polynomial which
is strictly positive on a dense connected subset of ${\mathcal C}$.
\end{lemma}

\begin{proof}
We shall be using some properties of algebraic and
semi-algebraic sets in ${\mathbb R}^m$, so many readers may
want to skip over it and go to our main results which are much
more widely understandable; our reference is
\cite{BCR98}.  We notice first that ${\mathcal C}$ is a semi-algebraic
set (since it is the closure of a connected component of a
semi-algebraic set, see \cite[Proposition 2.2.2 and Theorem
2.4.5]{BCR98}). Therefore the interior $\operatorname{int} {\mathcal
C}$ of ${\mathcal C}$ is also semi-algebraic, and so is the boundary
$\partial {\mathcal C} = {\mathcal C} \setminus \operatorname{int} {\mathcal C}$.
Notice also that ${\mathcal C}$ equals the closure of its interior.

We claim next that {\it for each $x \in \partial {\mathcal C}$, the
local dimension
$\dim \partial {\mathcal C}_x$ equals} $m-1$. On the one hand, we
have
$$
\dim \partial {\mathcal C}_x \leq \dim \partial {\mathcal C} <
\dim \operatorname{int} {\mathcal C} = m;
$$
here we have used \cite[Proposition 2.8.13 and Proposition 2.8.4]{BCR98},
and the fact that $\partial {\mathcal C} = \operatorname{clos} \operatorname{int} {\mathcal C}
\setminus \operatorname{int} {\mathcal C}$,
since ${\mathcal C}$ equals the closure of its interior.
On the other hand, let $B$ be an open ball in ${\mathbb R}^m$ around $x$;
then
$$
B \cap \partial {\mathcal C} = B \setminus \big[
\bigl(B \cap ({\mathbb R}^m \setminus {\mathcal C})\bigr) \cup
\bigl(B \cap \operatorname{int} {\mathcal C}\bigr)\big].
$$
Since ${\mathcal C}$ equals the closure of its interior,
every point of $\partial {\mathcal C}$ is an accumulation point of both
${\mathbb R}^m \setminus {\mathcal C}$ and $\operatorname{int} {\mathcal C}$;
therefore $B \cap ({\mathbb R}^m \setminus {\mathcal C})$
and $B \cap \operatorname{int} {\mathcal C}$ are disjoint
open nonempty semi-algebraic subsets of $B$. Using \cite[Lemma
4.5.2]{BCR98} we conclude that
$\dim B \cap \partial {\mathcal C} \geq m-1$,
hence $\dim \partial {\mathcal C}_x \geq m-1$.

Let now $V$ be the Zariski closure of $\partial {\mathcal C}$,
and let $V = V_1 \cup \cdots \cup V_k$ be the decomposition of $V$ into irreducible components.
We claim that {\it $\dim V_i = m-1$ for each $i$}.
Assume by contradiction that $V_1,\ldots,V_l$ have dimension $m-1$ while
$V_{l+1},\ldots,V_k$ have smaller dimension.
Then there exists $x \in \partial {\mathcal C}$
such that $x \not\in V_1,\ldots,V_l$, and consequently there exists an open ball $B$
in ${\mathbb R}^m$ around $x$ such that
$$
B \cap \partial {\mathcal C} = (B \cap \partial {\mathcal C} \cap V_{l+1}) \cup \cdots \cup
(B \cap \partial {\mathcal C} \cap V_k).
$$
By assumption each set in the union on the right hand side has dimension smaller than $m-1$,
hence it follows (by \cite[Proposition 2.8.5, I]{BCR98}) that
$\dim B \cap \partial {\mathcal C} < m-1$, a contradiction
with $\dim \partial {\mathcal C}_x = m-1$.

Suppose now that $p$ is a polynomial
of the lowest degree with ${\mathcal C} = {\mathcal C}_p(x^0)$.
Lowest degree implies that $p$ can have no
multiple irreducible factors, i.e., $p=p_1 \cdots p_s$, where $p_1$, \dots, $p_s$ are distinct
irreducible polynomials; we may assume without loss of generality that every $p_i$
is non-negative on ${\mathcal C}$. Since $p$ vanishes on $\partial {\mathcal C}$
it also vanishes on $V = V_1 \cup \cdots \cup V_k$.
We claim that for every $V_i$ there exists a $p_j$ so that $p_j$ vanishes on $V_i$:
otherwise ${\mathcal Z}(p_j) \cap V_i$ is a proper algebraic subset
of $V_i$ for every $j = 1,\ldots,s$, therefore (since $V_i$ is irreducible)
$\dim {\mathcal Z}(p_j) \cap V_i < \dim V_i$ for every $j$
and thus also $\dim {\mathcal Z}(p) \cap V_i < \dim V_i$,
a contradiction since $p$ vanishes on $V_i$.
If $p_j$ vanishes on $V_i$, it follows ( since $p_j$ is irreducible
and $\dim V_i = m-1$) that ${\mathcal Z}(p_j) = V_i$.
The fact that $p$ is a polynomial
of the lowest degree with ${\mathcal C} = {\mathcal C}_p(x^0)$
implies now (after possibly renumbering the irreducible factors of $p$)
that $p = p_1 \cdots p_k$ where ${\mathcal Z}(p_i) = V_i$ for every $i$.
Since $\dim V_i = m-1$ it follows from the real Nullstellensatz for principal ideals
\cite[Theorem 4.5.1]{BCR98} that the irreducible polynomials $p_i$ are uniquely determined
(up to a multiplication by a positive constant),
hence so is their product $p$.  This proves the uniqueness of $p$.

The rest of the lemma now follows easily. If ${\mathcal C} = {\mathcal C}_q(x^0)$,
then the polynomial $q$ vanishes on $\partial {\mathcal C}$ hence also on
$V = V_1 \cup \cdots \cup V_k$. Since $q$ vanishes on ${\mathcal Z}(p_i) = V_i$
and $\dim V_i = m-1$, the real Nullstellensatz for principal ideals
implies that $q$ is divisible by $p_i$; this holds for every $i$ hence
$q$ is divisible by $p = p_1 \cdots p_k$, i.e., $q = p h$.
It is obvious that $h$ must be strictly positive on a dense connected subset of ${\mathcal C}$.
\end{proof}
 
\printindex

\newpage

\tableofcontents


\begin{thebibliography}{999}
\scriptsize

\bibitem[AM]{AM+}
J. Agler, J.E. McCarthy:
Global holomorphic functions in several non-commuting variables,
{\em preprint} \url{http://arxiv.org/abs/1305.1636}


\bibitem[BB07]{BB07}
J.A. Ball, V. Bolotnikov:
Interpolation in the noncommutative Schur-Agler class,
{\em J. Operator Theory} {\bf 58} (2007) 83--126.


\bibitem[Bar02]{Ba} A. Barvinok:
{\it A course in convexity},
Graduate Studies in Mathematics {\bf 54}, Amer. Math. Soc., 2002.


\bibitem[BN93]{BN93}
E. Becker, R. Neuhaus:
Computation of real radicals of polynomial ideals,
In: {\it Computational algebraic geometry $($Nice, 1992$)$}, 1--20,
Progr. Math. {\bf 109}, Birkh\"auser, 1993.

\bibitem[B\^ oc07]{Boc07}
M. B\^ ocher:
{\it Introduction to Higher Algebra}, Dover, 1907.


\bibitem[BCR98]{BCR98}
J. Bochnack, M. Coste, M.-F. Roy: {\it Real algebraic geometry},
Ergebnisse der Mathematik und ihrer Grenzgebiete {\bf 3},
Springer, 1998.


\bibitem[Bon48]{Bon48} F. Bohnenblust:
Joint positiveness of matrices,
{\em technical report,}
California Institute of Technology, 1948.
Available from
\url{
http://orion.uwaterloo.ca/~hwolkowi/henry/book/fronthandbk.d/Bohnenblust.pdf}




\bibitem[BV96]{BV96} S. Boyd, L. Vandenberghe: Semidefinite programming, {\em
SIAM Rev.} {\bf 38} (1996), no. 1, 49--95.

\bibitem[CKP11]{CKP11}
K. Cafuta, I. Klep, J. Povh:
\href{http://ncsostools.fis.unm.si}{\tt NCSOStools}: a computer algebra system
for symbolic and numerical computation with noncommutative polynomials,
{\em Optim. Methods Softw.} {\bf 26} (2011) 363--380.
 Available from \url{http://ncsostools.fis.unm.si}

\bibitem[CF96]{CF96}
R. Curto, L. Fialkow:
Solution of the truncated complex moment problem for flat data,
{\em Mem. Amer. Math. Soc.} {\bf 119} (1996).

\bibitem[CF98]{CF98}
R. Curto, L. Fialkow:
Flat extensions of positive moment matrices: Recursively generated relations,
{\em Mem. Amer. Math. Soc.} {\bf 136} (1998).


 \bibitem[CHMN13]{chmn}
J. Cimpri\v c, J.W. Helton, S. McCullough, C.S. Nelson: A non-commutative real
Nullstellensatz corresponds to a non-commutative real ideal; Algorithms,
{\em Proc. Lond. Math. Soc.}
{\bf 106} (2013) 1060--1086.

\bibitem[dKTR00]{dKTR00}
E. de Klerk, T. Terlaky, K. Roos:
Self-dual embeddings. In: 
{\it Handbook of Semidefinite Programming},
111--138, Kluwer, 2000.

\bibitem[dKl02]{dKl02}
E. de Klerk: 
{\it Aspects of semidefinite programming: interior point algorithms and selected applications}, 
Kluwer, 2002.


\bibitem[DLTW08]{DLTW08}
A.C. Doherty, Y.-C. Liang, B. Toner, S. Wehner:
The quantum moment problem and bounds on entangled multi-prover games,
{\it Twenty-Third Annual IEEE Conference on Computational Complexity} 199--210,
IEEE Computer Soc., 2008.

\bibitem[Gre00]{Gre00}
E. Green: Multiplicative bases, Gr\"obner bases and right Gr\"obner bases, {\em J. Symbolic Comput.} {\bf 29} (2000) 601--623.


\bibitem[HKM12]{HKMb} J.W. Helton, I. Klep,  S.A. McCullough, The convex
Positivstellensatz
in a free algebra, {\em Adv. Math.} {\bf 231} (2012) 516--534.

\bibitem[HKM13]{HKMa} J.W. Helton, I. Klep, S. McCullough:
The matricial relaxation of a linear matrix inequality,
{\em Math. Program.} {\bf 138} (2013) 401--445.




\bibitem[HM04]{HM}
J.W. Helton, S. McCullough:
A Positivstellensatz for noncommutative polynomials,
{\em Trans. Amer. Math. Soc.} {\bf 356} (2004) 3721--3737.


\bibitem[HM12]{HM3}
 J.W. Helton, S. McCullough:
Every free basic convex semi-algebraic set has an LMI representation,
{\em Ann. of Math.} (2) {\bf 176} (2012) 979--1013.

\bibitem[HMP07]{HMP07}
 J.W. Helton, S. McCullough, M. Putinar:
 Strong majorization in a free $*$-algebra,
{\em Math. Z.} {\bf 255} (2007) 579--596.


\bibitem[HOSM]{HOSM}
J.W. Helton, M.C. de Oliveira, M. Stankus, R.L. Miller:
\href{http://math.ucsd.edu/~ncalg}{\tt NCAlgebra}, 2013 release edition.
Available from
\url{http://math.ucsd.edu/~ncalg}

\bibitem[HV07]{HV07}
J.W. Helton, V. Vinnikov:
Linear matrix inequality representation of sets,
{\em Commun. Pure Appl. Math.} {\bf 60} (2007) 654--674.



\bibitem[KVV]{KVV+}
D. Kalyuzhnyi-Verbovetski\u\i{}, V. Vinnikov:
Foundations of noncommutative function theory,
{\em preprint}\\ \url{http://arxiv.org/abs/1212.6345}


\bibitem[KP10]{KP10} I. Klep and J. Povh, Semidefinite programming and sums of
hermitian
squares of noncommutative polynomials, {\em J. Pure Appl. Algebra} {\bf 214} (2010) 740--749.


\bibitem[KS11]{KS11}
I. Klep, M. Schweighofer:
Infeasibility Certificates for Linear Matrix Inequalities,
{\em Oberwolfach Preprints} (OWP) {\bf 28} (2011).

\bibitem[KS13]{KS}
I. Klep, M. Schweighofer,
An exact duality theory for semidefinite programming based on sums of squares,
{\em Math. Oper. Res.} {\bf 38} (2013) 569--590. 

\bibitem[Las10]{Las10}
J.B. Lasserre:
{\it Moments, positive polynomials and their applications},
Imperial College Press Optimization Series {\bf 1},
2010.

\bibitem[Lau09]{Lau09}
M. Laurent:
Sums of squares, moment matrices and optimization over polynomials,
In: {\it Emerging applications of algebraic geometry} 157--270,
IMA Vol. Math. Appl. {\bf 149}, Springer, 2009.
Updated version available at 
\url{http://homepages.cwi.nl/~monique/files/moment-ima-update-new.pdf}


\bibitem[Mar08]{Mar08}
M. Marshall: {\it Positive polynomials and sums of squares},
Mathematical Surveys and Monographs {\bf 146}, Amer. Math. Soc.,
2008.



\bibitem[MS11]{MS11}
P.S. Muhly, B. Solel:
Progress in noncommutative function theory,
{\it Sci. China Ser. A} {\bf 54} (2011) 2275--2294.


\bibitem[Nel]{N}
C. Nelson: A Real Nullstellensatz for Matrices of  Non-Commutative Polynomials, {\em preprint},
\url{http://arxiv.org/abs/1305.0799}

\bibitem[Neu98]{Neu98}
R. Neuhaus:
Computation of real radicals of polynomial ideals II, 
{\em J. Pure Appl. Algebra} {\bf 124} (1998) 261--280. 

\bibitem[Pau02]{Pau02}
V. Paulsen:
{\it Completely bounded maps and operator algebras},
Cambridge University Press, 2002.



\bibitem[PD01]{PD01}
A. Prestel, C.N. Delzell: {\it Positive polynomials. From Hilbert's 17th problem
to real algebra},
Springer Monographs in Mathematics, 2001.


\bibitem[PNA10]{PNA10}
S. Pironio, M. Navascu\'es, A. Ac\'in:
Convergent relaxations of polynomial optimization problems with noncommuting
variables, {\em SIAM J. Optim.} {\bf 20} (2010) 2157--2180.


\bibitem[Poe10]{Pope10} 
G. Popescu:
Free holomorphic automorphisms of the unit ball of $B(H)^n$,
    {\em J. reine angew. Math.} {\bf 638} (2010) 119--168.

\bibitem[Pop10]{Pop10}
S. Popovych:
Positivstellensatz and flat functionals on path $*$-algebras,
{\em J. Algebra} {\bf 324} (2010)
2418--2431.

\bibitem[Put93]{Put93}
M. Putinar:
Positive polynomials on compact semi-algebraic sets, {\em Indiana Univ. Math.
J.} {\bf 42} (1993) 969--984.

\bibitem[Ren06]{Ren06}
J. Renegar: Hyperbolic programs, and their derivative relaxations, 
{\em Found. Comput. Math.} {\bf 6} (2006) 59--79.



\bibitem[Sce09]{Sce09}
C. Scheiderer:
Positivity and sums of squares: a guide to recent results. In: {\it Emerging
applications of algebraic geometry} 271--324, IMA Vol. Math. Appl. {\bf 149},
Springer, 2009.


\bibitem[\v Saf99]{Saf99}
I. \v Safarevi\v c:
{\it Algebraic geometry}, Springer, 1999.


\bibitem[Tod01]{To01} M.J. Todd: Semidefinite optimization, {\em Acta Numer.}
{\bf 10} (2001), 515--560.

\bibitem[Vin12]{Vin12}
V. Vinnikov:
LMI representations of convex semialgebraic sets and determinantal representations of algebraic hypersurfaces: past, present, and future. 
In: {\em Mathematical methods in systems, optimization, and control}, 325--349, Oper. Theory Adv. Appl. {\bf 222}, Birkh\"auser/Springer, 2012.


\bibitem[Voi04]{Voi04} 
D.-V. Voiculescu:
Free analysis questions I: Duality transform for the
coalgebra of $\partial_{X:B}$, {\em International Math. Res. Notices} {\bf 16}
(2004) 793--822.

\bibitem[Voi10]{Voi10} 
D.-V. Voiculescu:
Free analysis questions II: The Grassmannian completion and the series
expansions at the origin, {\em J. reine angew. Math.} {\bf 645} (2010) 155--236.

\bibitem[VDN92]{VDN92}
D.-V. Voiculescu, K.J. Dykema, A. Nica:
{\em Free random variables. A noncommutative probability approach to free products with applications to random matrices, operator algebras and harmonic analysis on free groups}, Amer. Math. Soc., 1992.



\bibitem[WSV00]{WSV00}
H. Wolkowicz, R. Saigal, L. Vandenberghe (editors):
{\it Handbook of semidefinite programming. Theory, algorithms,
and applications},
Kluwer, 2000.

\end{thebibliography}
\end{document}